\documentclass[10pt,english,reqno]{amsart} 

\calclayout

\usepackage[hidelinks]{hyperref}
\usepackage[OT2, T1]{fontenc}
\usepackage{graphicx}
\usepackage{amssymb}
\usepackage{epstopdf}
\usepackage{enumerate}
\usepackage{tikz-cd}
\usepackage{MnSymbol}
\usepackage{mathrsfs}

\DeclareMathSymbol{A}{\mathalpha}{operators}{`A}
\DeclareMathSymbol{B}{\mathalpha}{operators}{`B}
\DeclareMathSymbol{C}{\mathalpha}{operators}{`C}
\DeclareMathSymbol{D}{\mathalpha}{operators}{`D}
\DeclareMathSymbol{E}{\mathalpha}{operators}{`E}
\DeclareMathSymbol{F}{\mathalpha}{operators}{`F}
\DeclareMathSymbol{G}{\mathalpha}{operators}{`G}
\DeclareMathSymbol{H}{\mathalpha}{operators}{`H}
\DeclareMathSymbol{I}{\mathalpha}{operators}{`I}
\DeclareMathSymbol{J}{\mathalpha}{operators}{`J}
\DeclareMathSymbol{K}{\mathalpha}{operators}{`K}
\DeclareMathSymbol{L}{\mathalpha}{operators}{`L}
\DeclareMathSymbol{M}{\mathalpha}{operators}{`M}
\DeclareMathSymbol{N}{\mathalpha}{operators}{`N}
\DeclareMathSymbol{O}{\mathalpha}{operators}{`O}
\DeclareMathSymbol{P}{\mathalpha}{operators}{`P}
\DeclareMathSymbol{Q}{\mathalpha}{operators}{`Q}
\DeclareMathSymbol{R}{\mathalpha}{operators}{`R}
\DeclareMathSymbol{S}{\mathalpha}{operators}{`S}
\DeclareMathSymbol{T}{\mathalpha}{operators}{`T}
\DeclareMathSymbol{U}{\mathalpha}{operators}{`U}
\DeclareMathSymbol{V}{\mathalpha}{operators}{`V}
\DeclareMathSymbol{W}{\mathalpha}{operators}{`W}
\DeclareMathSymbol{X}{\mathalpha}{operators}{`X}
\DeclareMathSymbol{Y}{\mathalpha}{operators}{`Y}
\DeclareMathSymbol{Z}{\mathalpha}{operators}{`Z}

\newcommand{\fr}{\mathfrak}
\newcommand{\cal}{\mathscr}

\newcommand{\op}{\operatorname}
\newcommand{\tn}{\textnormal}

\newcommand{\Sm}{\mathrm{Sm}}

\newcommand{\Spc}{\mathrm{Spc}}
\newcommand{\Sptr}{\mathrm{Sptr}}
\newcommand{\PShv}{\mathrm{PShv}}
\newcommand{\Shv}{\mathrm{Shv}}

\newcommand{\Maps}{\mathrm{Maps}}
\newcommand{\Vect}{\mathrm{Vect}}
\newcommand{\Mod}{\mathrm{Mod}}
\newcommand{\Comod}{\mathrm{Comod}}
\newcommand{\Rep}{\mathrm{Rep}}

\newcommand{\QCoh}{\mathrm{QCoh}}

\newcommand{\Ind}{\op{Ind}}

\DeclareMathOperator*\colim{colim}

\newcommand{\Hom}{\op{Hom}}
\newcommand{\End}{\op{End}}
\newcommand{\Ker}{\op{Ker}}
\newcommand{\Aut}{\op{Aut}}
\newcommand{\Ext}{\op{Ext}}
\newcommand{\Tor}{\op{Tor}}
\newcommand{\Sym}{\op{Sym}}
\newcommand{\Ant}{\op{Ant}}

\newcommand{\Spec}{\op{Spec}}
\newcommand{\Gal}{\op{Gal}}
\newcommand{\Nm}{\op{Nm}}

\newcommand{\grp}{\mathrm{grp}}

\newcommand{\Quad}{\mathrm{Quad}}
\newcommand{\der}{\mathrm{der}}

\newcommand{\com}{\mathrm{com}}

\newcommand{\loo}[1]{(\!(#1)\!)}

\newcommand{\GL}{\mathrm{GL}}
\newcommand{\SL}{\mathrm{SL}}

\newcommand{\SO}{\mathrm{SO}}
\newcommand{\Sp}{\mathrm{Sp}}

\newcommand{\Bun}{\mathrm{Bun}}

\newcommand{\et}{\textnormal{\'et}}
\newcommand{\Zar}{\textnormal{Zar}}
\newcommand{\dR}{\textnormal{dR}}

\newcommand{\Lis}{\mathrm{Lis}}
\newcommand{\Fib}{\mathrm{Fib}}

\newcommand{\res}{\mathrm{res}}
\newcommand{\ord}{\mathrm{ord}}
\newcommand{\sgn}{\mathrm{sgn}}
\newcommand{\ad}{\mathrm{ad}}
\newcommand{\pt}{\mathrm{pt}}
\newcommand{\st}{\mathrm{st}}
\newcommand{\ev}{\mathrm{ev}}
\newcommand{\pr}{\mathrm{pr}}
\newcommand{\id}{\mathrm{id}}

\newcommand{\act}{\mathrm{act}}
\newcommand{\alg}{\mathrm{alg}}
\newcommand{\tr}{\mathrm{tr}}
\newcommand{\cl}{\mathrm{cl}}
\newcommand{\ab}{\mathrm{ab}}
\newcommand{\an}{\mathrm{an}}

\newtheorem{thm}[subsubsection]{Theorem}
\newtheorem*{thm*}{Theorem}
\newtheorem{thmx}{Theorem}

\newtheorem{prop}[subsubsection]{Proposition}
\newtheorem{lem}[subsubsection]{Lemma}

\newtheorem{cor}[subsubsection]{Corollary}

\theoremstyle{definition}
\newtheorem{defn}[subsubsection]{Definition}

\newtheorem{eg}[subsubsection]{Example}
\newtheorem{rem}[subsubsection]{Remark}

\numberwithin{equation}{section}

\newtheorem{untitledsubsubsection}[subsubsection]{}

\makeatletter
\renewcommand{\@secnumfont}{\bfseries}
\makeatother

\newenvironment{void}
{\begin{untitledsubsubsection}}
{\end{untitledsubsubsection}}

\title{\'Etale metaplectic covers of reductive group schemes}

\author{Yifei Zhao}
\email{yifei.zhao@uni-muenster.de}

\thanks{
The project was funded by the Deutsche Forschungsgemeinschaft (DFG, German Research Foundation) Project-ID 427320536 -- SFB 1442, as well as under Germany's Excellence Strategy EXC 2044 390685587, Mathematics Münster: Dynamics--Geometry--Structure.
}

\thanks{
Part of the project was carried out when the author received funding from the European Research Council (ERC) under the European Union's Horizon 2020 research and innovation programme (grant agreement No.~851146).
}

\begin{document}


\begin{abstract}
Given a reductive group scheme $G$, we give a linear algebraic description of reduced \'etale $4$-cocycles on its classifying stack $B(G)$. These cocycles form a $2$-groupoid, which we interpret as parameters of metaplectic covers of $G$. We use our linear algebraic description to define the Langlands dual of a metaplectic cover.
\end{abstract}

\maketitle

\setcounter{tocdepth}{2}
\tableofcontents


\section*{Introduction}

The goal of this article is to build a theory of covering groups of reductive group schemes $G$ over an arbitrary base scheme $S$, relying on \'etale cohomology. The idea for such a theory is due to Deligne \cite{MR1441006}, who also worked out the case where $G$ is semisimple and simply connected. The generalization to the reductive case, carried out in the present article, is motivated by the geometrization of the Langlands program \cite{beilinson1991quantization} \cite{MR3300415} \cite{fargues2021geometrization}, as well as the possibility of casting covering groups under the same framework, see \cite{MR2684259} \cite{MR2963537} \cite{MR3802418} \cite{MR3802419} \cite{MR3769731}.

There are already two existing foundations for the Langlands program for covering groups of reductive groups:
\begin{enumerate}
	\item Weissman \cite{MR3802418} has defined the $\tn L$-groups of covering groups using Brylinski and Deligne's parametrization of them by algebraic $\tn K$-theory \cite{MR1896177};
	\item Gaitsgory--Lysenko \cite{MR3769731} has defined the ``metaplectic dual data'' of factorization gerbes on the affine Grassmannian associated to $G$ and a smooth curve $X$, interpreted as parameters of the metaplectic geometric Langlands program.
\end{enumerate}

The advantage of Weissman's foundation is that it works over number fields and function fields alike, whereas the advantage of Gaitsgory and Lysenko's theory is that it is intrinsically geometric, allowing for example, a straightforward formulation of the geometric Satake equivalence.

The foundation we construct is their ``least common multiple'': it is essentially equivalent to that of \cite{MR3769731} but remains valid in the number field context. It includes all covering groups which arise from algebraic $\tn K$-theory, as well as an extra class coming from abelianized cohomology of $G$, which has applications in Langlands functoriality by Kaletha's work \cite{kaletha2022}. Furthermore, we prove a number of structural results which lift classical computations to the geometric level, including for example the behavior of certain covering groups when restricted to parabolic subgroups.

The main intended applications of the present article are to extend V.~Lafforgue's spectral decomposition \cite{MR3787407} to covering groups over function fields and Fargues and Scholze's spectral action \cite{fargues2021geometrization} to covering groups over a $p$-adic field (jointly with Gaisin, Imai, and Koshikawa). The present article partially serves as a collection of the group-theoretic inputs they require. I beg the reader's forgiveness for not including concrete applications herein.

\subsection{\'Etale metaplectic covers}

\begin{void}
Let us explain the core definition of this article. Let $S$ be a scheme, $G\rightarrow S$ be a group scheme of finite type, and $A$ be a locally constant \'etale sheaf of finite abelian groups whose order is invertible on $S$. We define an \emph{$A$-valued \'etale metaplectic cover} of $G$ to be a rigidified section of $B^4A(1)$ over the classifying stack $BG$.\footnote{We use the adjective ``metaplectic'' since ``\'etale cover'' has an established meaning, although the covering groups we study include far more instances than those having to do with the symplectic group.}

Here, $A(1)$ is the Tate twist of $A$ introduced for convention, $B^4A(1)$ is its fourth iterated classifying stack, and ``rigidified'' means being equipped with a trivialization over the netural point $e : S\rightarrow BG$.

We use the language of Higher Algebra, as developed by Lurie \cite{MR2522659} \cite{lurie2017higher}, in making this definition, although it can be avoided by working with chain complexes on simplicial schemes as in \cite{MR1441006}. However, for the definition of the $\tn L$-group, as well as the intended applications, we will need to consider algebraic structures on higher groupoids, which can be cumbersome to formulate without Higher Algebra.

Let us explain how this definition is related to other notions of covering groups.
\end{void}

\begin{void}\emph{Topological covers.}
To connect this definition to classical covering groups, we assume that $S$ is the spectrum of a nonarchimedean local field $F$ and $A$ is a constant sheaf (for instance, the abelian group $\mu(F)$ of roots of unity in $F$). Then any \'etale metaplectic cover of $G$ induces a central extension of topological groups:
\begin{equation}
\label{eq-topological-central-extension-intro}
1 \rightarrow A \rightarrow \tilde G \rightarrow G(F) \rightarrow 1.
\end{equation}

The idea is simply to consider the monoidal morphism $G\rightarrow B^3A(1)$ associated to an \'etale metaplectic cover, evaluate it on $\Spec(F)$, and use the vanishing of the \'etale cohomology group $H^3(F, A(1))$ and the Tate duality isomorphism $H^2(F, A(1)) \cong H^0(F, A)$. The archimedean case is addressed separately since $H^3$ may not vanish.

If $F$ is a global field with topological ring of ad\`eles $\mathbb A_F$, we obtain a central extension of $G(\mathbb A_F)$ (or possibly a subgroup thereof when $F$ contains real places) equipped with a canonical splitting over $G(F)$. When $A$ appears as a subgroup of the multiplicative group of a coefficient field, we arrive at the notion of ``genuine'' automorphic representations.

The most classical example of \eqref{eq-topological-central-extension-intro} is Kubota's double cover of $\SL_2(F)$, for $F$ of characteristic $\neq 2$ (see \cite{MR204422}), which is induced from the \'etale metaplectic cover represented by the mod $2$ universal second Chern class $[c_2]\in H^4(B(\SL_2), \{\pm 1\}^{\otimes 2})$.
\end{void}

\begin{void}\emph{Central extensions by $\underline K_2$.}
Let us relate \'etale metaplectic covers to central extensions of $G$ by the Zariski sheafified second algebraic $\tn K$-group $\underline K_2$, considered by Brylinski--Deligne \cite{MR1896177}. The relevant \'etale metaplectic covers have coefficient group $A = \mu_N$, for $N\ge 1$ being an integer invertible on $S$.

The essential point is that over a smooth scheme over a field, there is a canonical functor from central extensions of $G$ by $\underline K_2$ to \'etale metaplectic covers. It is a version of the \'etale realization of motivic cohomology of $BG$. The construction of this functor is essentially due to Gaitsgory \cite[\S6]{MR4117995} and involves strengthening some results of \cite[\S6]{MR1460391} proving the equivalence between $K$-cohomology of $BG$ and integral motivic cohomology in both the Zariski and \'etale topologies, see Theorem \ref{thm-motivic-versus-K2} in the main body of the text.

If $S$ is the spectrum of a nonarchimedean local field $F$ containing a primitive $N$th root of unity, then we have a commutative diagram:
$$
\begin{tikzcd}[column sep = -1em, row sep = 0.5em]
	\begin{Bmatrix}
	\text{central extensions} \\
	\text{of $G$ by $\underline K_2$}
	\end{Bmatrix}
	 \ar[dr]\ar[dd] \\
	 &
	 \begin{Bmatrix}
	 \text{\'etale metaplectic covers} \\
	 \text{of $G$ valued in $\mu_N$}
	 \end{Bmatrix}\ar[dl] \\
	 \begin{Bmatrix}
	 \text{central extensions} \\
	 \text{of $G(F)$ by $\mu_N(F)$}
	 \end{Bmatrix}
\end{tikzcd}
$$
where the vertical functor is that of \cite[\S10]{MR1896177}. There is an analogous commutative diagram for global fields. They show that \'etale metaplectic covers \emph{refine} the parametrization of covering groups by algebraic $\tn K$-theory.

For example, the fact that the restriction of the metaplectic double cover of $\Sp_{2n}$ to the Siegel parabolic is induced along the determinant of the Levi quotient $\GL_n$ is witnessed by the corresponding \'etale metaplectic cover but not the $\tn K$-theoretic one.
\end{void}

\begin{void}\emph{Geometrization.}
\label{void-geometrization-intro}
Perhaps most importantly, $A$-valued \'etale metaplectic covers define $A$-gerbes on the geometric objects relevant to the Langlands program. When $A$ is contained in the multiplicative group of a coefficient field, say $\overline{\mathbb Q}_{\ell}$ for a prime $\ell$ invertible on $S$, the induced $\overline{\mathbb Q}{}^{\times}_{\ell}$-gerbes allow us to form twisted categories of constructible sheaves.

Let us first discuss the global function field context: the base scheme is a smooth, proper, geometrically connected curve $X$ over a finite field with field of fractions $F$. Let $\Bun_G$ denote the moduli stack of $G$-bundles on $X$. Given an \'etale metaplectic cover of $G$, its pullback along the universal $G$-bundle $\Bun_G\times X \rightarrow BG$ defines a section of the complex $A(1)[4]$ over $\Bun_G\times X$. Pairing with the fundamental class of $X$ defines a map:
$$
\Gamma(\Bun_G\times X, A(1)[4]) \rightarrow \Gamma(\Bun_G, A[2]),
$$
and a section of the target is precisely an $A$-gerbe on $\Bun_G$. Constructible sheaves twisted by the induced $\overline{\mathbb Q}{}_{\ell}^{\times}$-gerbe define unramified genuine automorphic forms on the corresponding covering group of $G(\mathbb A_F)$ by the trace of Frobenius. The construction has a variant in the ramified situation as well.

Besides $\Bun_G$, \'etale metaplectic covers naturally define $\overline{\mathbb Q}{}_{\ell}^{\times}$-gerbes on the Hecke stack compatibly with the convolution structure. The variant for the local Hecke stack is furthermore compatible with its factorization structure. This observation allows a significant amount of the geometric Langlands program to be transported to the metaplectic context, as explained in \cite{MR3769731}.

When the base is a $p$-adic field $F$, \'etale metaplectic covers also define $A$-gerbes on the Fargues--Scholze $\Bun_G$: the $v$-sheaf assigning the groupoid of $G$-bundles on the Fargues--Fontaine curve $X_{(S, S^+)}$ to any affinoid perfectoid space $(S, S^+)$, see \cite{fargues2021geometrization}. The construction is a variant of the global function field case, where the role of $X$ is played by the ``mirror curve'' $\mathrm{Div}^1$. Over the open locus $[*/G(F)]$ of $\Bun_G$, the resulting $A$-gerbe is canonically rigidified and recovers the central extension $\tilde G$ of $G(F)$. The category of lisse sheaves twisted by its induced $\overline{\mathbb Q}{}_{\ell}^{\times}$-gerbe is a geometric incarnation of the category of genuine smooth representations of $\tilde G$.

In geometric applictions, \'etale metaplectic covers offer a technical advantage over their $\tn K$-theoretic counterparts, because the theory requires no regularity assumption on the base scheme, which is needed for \cite{MR1896177}.
\end{void}

\subsection{Classification}

\begin{void}
We hope to have conveyed some sense of the utility of having a theory of covering groups based on \'etale cohomology. We now turn to the main theorem of the article, which is a classification of \'etale metaplectic covers of any reductive group scheme $G\rightarrow S$: by definition, they form the space of rigidified sections $\Gamma_e(BG, B^4A(1))$. Our classification is in complete parallel with \cite[Theorem 7.2]{MR1896177}, the only difference being that $\Gamma_e(BG, B^4A(1))$ is a $2$-groupoid rather than a $1$-groupoid.

Let $\underline{\Gamma}{}_e(BG, B^4A(1))$ denote the \'etale sheaf on $S$, assigning the space of rigidified sections of $B^4A(1)$ over $BG\times_S S_1$ for any $S$-scheme $S_1$. By abstract nonsense, each homotopy sheaf $\pi_i\underline{\Gamma}(BG, B^4A(1))$ is isomorphic to $R^{4-i}p_*A(1)$, where $p : BG \rightarrow S$ denotes the projection map. The higher direct images $R^{4-i}p_*A(1)$ are computed by the \'etale cohomology of $BG$, which is standard. The part of the problem which we resolve is how these homotopy sheaves ``fit together'' in $\underline{\Gamma}{}_e(BG, B^4A(1))$.

Sections \ref{sec-torus}-\ref{sec-classification} describe $\underline{\Gamma}{}_e(BG, B^4A(1))$ in two stages.
\end{void}

\begin{void}\emph{Tori.}
\label{void-tori-intro}
Suppose that $T\rightarrow S$ is a torus with sheaf of cocharacters $\Lambda$. Then Theorem \ref{thm-classification-torus} expresses $\underline{\Gamma}{}_e(BT, B^4A(1))$ as both a pushout and a pullback (in the $\infty$-categorical sense) of constructions of linear algebraic nature:
\begin{equation}
\label{eq-second-theta-data-intro}
\begin{tikzcd}[column sep = 0.5em, row sep = 2em]
	\underline{\Maps}{}_{\mathbb Z}(\wedge^2\Lambda, A(-1)) \ar[r]\ar[d, "\Psi(-1)"] & \underline{\Maps}{}_{\mathbb Z}(\Lambda^{\otimes 2}, A(-1)) \ar[d] \\
	\underline{\Maps}{}_{\mathbb Z}(\Lambda, B^2A) \ar[r] & \underline{\Gamma}{}_e(BT, B^4A(1)) \ar[r]\ar[d] & \underline{\Maps}{}_{\mathbb Z}(\Gamma^2\Lambda, A(-1)) \ar[d, "\Psi(-1)"] \\
	& \underline{\Maps}{}_{\mathbb Z}(H^{(2)}(\Lambda), B^2A) \ar[r] & \underline{\Maps}{}_{\mathbb Z}(\wedge^2\Lambda, BA)
\end{tikzcd}
\end{equation}
Here, $\underline{\Maps}{}_{\mathbb Z}(-,-)$ denotes the sheaf of $\mathbb Z$-linear morphisms and the arrows labeled $\Psi(-1)$ are constructed using the ``Kummer torsor'' of roots of $(-1)$. The sheaf $H^{(2)}(\Lambda)$ is an extension:
$$
B(\wedge^2\Lambda) \rightarrow H^{(2)}(\Lambda) \rightarrow \Lambda,
$$
and will be introduced in \S\ref{sec-torus}; here we only mention informally that maps out of $H^{(2)}(\Lambda)$ encode ``central extensions of $\Lambda$ with prescribed commutators''. The horizontal map out of $\underline{\Gamma}{}_e(BT, B^4A(1))$ in \eqref{eq-second-theta-data-intro} attaches an $A(-1)$-valued quadratic form $Q$ to every \'etale metaplectic cover of $T$: this is its discrete invariant.

The diagram \eqref{eq-second-theta-data-intro} has concrete analogues in the $\tn K$-theoretic context. The two maps received by $\underline{\Gamma}{}_e(BT, B^4A(1))$ correspond to the constructions of central extensions by cocycles, respectively by exact sequences of abelian sheaves. The pullback square, on the other hand, is an analogue of Brylinski and Deligne's description of central extensions of $T$ by $\underline K_2$ as those of $\Lambda$ by $\mathbb G_m$ with prescribed commutators, see \cite[Theorem 3.16]{MR1896177}.
\end{void}

\begin{void}\emph{Reductive group schemes.}
\label{void-reductive-intro}
Suppose that $G\rightarrow S$ is a reductive group scheme. There is an \'etale sheaf of based root data $(\Delta\subset\Phi\subset\Lambda, \check{\Delta}\subset\check{\Phi}\subset\check{\Lambda})$, $\Phi\cong\check{\Phi}$ over $S$ associated to $G$. For instance, once a Borel subgroup $B\subset G$ is chosen, then $\Lambda$ is canonically identified with the sheaf of cocharacters of its maximal quotient torus.

Our description of $\underline{\Gamma}{}_e(BG, B^4A(1))$ is supposed to be an extension of the association of based root data to reductive group schemes. First of all, the calculation of \'etale cohomology of $BG$ yields a canonical triangle:
\begin{equation}
\label{eq-reductive-triangle-intro}
\underline{\Maps}{}_{\mathbb Z}(\pi_1G, B^2A) \rightarrow \underline{\Gamma}{}_e(BG, B^4A(1)) \rightarrow \Quad(\Lambda, A(-1))_{\st},
\end{equation}
where $\pi_1G$ is the algebraic fundamental group, and $\Quad(\Lambda, A(-1))_{\st}$ is a subsheaf of the sheaf of $A(-1)$-valued quadratic forms on $\Lambda$ characterized by the equality:
$$
b(\alpha, \lambda) = Q(\alpha)\langle\check{\alpha}, \lambda\rangle,\quad \alpha\in\Delta, \lambda\in\Lambda,
$$
for $b$ the symmetric form associated to $Q$. (It sends $\lambda_1,\lambda_2$ to $Q(\lambda_1+\lambda_2) - Q(\lambda_1) - Q(\lambda_2)$.) We call such quadratic forms \emph{strictly Weyl-invariant}, or simply \emph{strict}. The triangle \eqref{eq-reductive-triangle-intro} reduces to the middle triangle in \eqref{eq-second-theta-data-intro} for $G = T$ a torus, and yields an equivalence:
\begin{equation}
\label{eq-simply-connected-classification-intro}
\underline{\Gamma}{}_e(BG, B^4A(1)) \cong \Quad(\Lambda, A(-1))_{\st}
\end{equation}
when $G$ is simply connected.

To describe $\underline{\Gamma}{}_e(BG, B^4A(1))$, we will need to fix a Borel subgroup $B\subset G$. The restriction of an \'etale metaplectic cover along $B\subset G$ canonically descends to one of the ``universal Cartan'' $T := \Lambda\otimes \mathbb G_m$. It follows rather easily from the fiber sequence \eqref{eq-reductive-triangle-intro} that the following commutative diagram is Cartesian:
\begin{equation}
\label{eq-reductive-classification-square-intro}
\begin{tikzcd}[column sep = 1.5em, row sep = 2em]
	\underline{\Gamma}{}_e(BG, B^4A(1)) \ar[r, "\res_B"]\ar[d] & \underline{\Gamma}{}_e(BT, B^4A(1))_{\st} \ar[d] \\
	\underline{\Gamma}{}_e(BG_{\mathrm{sc}}, B^4A(1)) \ar[r, "\res_{B_{\mathrm{sc}}}"] & \underline{\Gamma}{}_e(BT_{\mathrm{sc}}, B^4A(1))_{\st}
\end{tikzcd}
\end{equation}
where $G_{\mathrm{sc}}$ denotes the simply connected form with its own universal Cartan $T_{\mathrm{sc}}$, and the horizontal functors are the restrictions along $B$ and the induced Borel subgroup $B_{\mathrm{sc}}\subset G_{\mathrm{sc}}$, see Theorem \ref{thm-reductive-classification}.

Combining the description of \'etale metaplectic covers of tori and that of simply connected groups, the Cartesian diagram \eqref{eq-reductive-classification-square-intro} classifies \'etale metaplectic covers of reductive group schemes. We state it in parallel with \cite[Theorem 7.2]{MR1896177}.
\end{void}

\begin{thmx}
\label{thmx-classification}
Suppose that $G\rightarrow S$ is a reductive group scheme equipped with a Borel subgroup $B\subset G$. Then $\underline{\Gamma}{}_e(BG, B^4A(1))$ is equivalent to the sheaf of quadruples $(Q, F, h, \varphi)$ where:
\begin{enumerate}
	\item $Q$ is a strictly Weyl-invariant quadratic form on $\Lambda$;
	\item $F : H^{(2)}(\Lambda) \rightarrow B^2A$ is a $\mathbb Z$-linear morphism;
	\item $h$ is an isomorphism between the restriction of $F$ to $B(\wedge^2\Lambda)$ and a $\mathbb Z$-linear morphism $\wedge^2(\Lambda)\rightarrow BA$ canonically attached to $Q$;
	\item $\varphi$ is an isomorphism between the restriction of the pair $(F, h)$ to $\Lambda_{\mathrm{sc}}$ (i.e.~the span of $\Delta$) and a pair defined by the restriction of $Q$ to $\Lambda_{\mathrm{sc}}$.
\end{enumerate}
\end{thmx}

There is, however, an important caveat in this description: the necessity of fixing a Borel subgroup. If two Borel subgroups $B_1,B_2\subset G$ are conjugate under a section $g\in G$, then we do find an isomorphism of functors $\res_{B_1} \cong \res_{B_2}$. However, this isomorphism \emph{depends} on the section $g$ and not on the Borel subgroups $B_1, B_2$ alone.

If the existence of Borel subgroups seems too restrictive, one can reformulate the classification Theorem \ref{thmx-classification} using a maximal torus as in \cite[Theorem 7.2]{MR1896177}, but the issue persists: for two maximal tori $T_1, T_2$, there is no canonical way to relate the functors defined by restrictions along $T_1,T_2\subset G$.

\subsection{The $\tn L$-group}

\begin{void}
Our main motivation for classifying \'etale metaplectic covers is to define their $\tn L$-groups, following Gaitsgory and Lysenko \cite{MR3769731}. It appears at first sight that \emph{op.cit.}~requires geometric input particular to the function field context. However, once \'etale metaplectic covers are classified, the construction of ``metaplectic dual data'' carries over and an $\tn L$-group in the style of Langlands can also be extracted formally. (In fact, only a small part of the classification is needed for this construction, and the reader who is only interested in the $\tn L$-group can safely skip Section \ref{sec-cup-product} and most of Sections \ref{sec-torus}-\ref{sec-classification}.)

In this construction, we fix a field of coefficients, say $\overline{\mathbb Q}{}_{\ell}$ for a prime $\ell$ invertible over the base scheme $S$, and assume that $A$ is a subsheaf of the constant sheaf $\overline{\mathbb Q}{}_{\ell}^{\times}$. From a reductive group scheme $G\rightarrow S$ equipped with an \'etale metaplectic cover $\mu$, we shall construct a triple $(H, \cal G_{Z_H(\overline{\mathbb Q}{}_{\ell})}, \epsilon)$ called \emph{metaplectic dual data}, where:
\begin{enumerate}
	\item $H$ is an \'etale local system over $S$ of pinned split reductive groups over $\overline{\mathbb Q}{}_{\ell}$;
	\item $\cal G_{Z_H(\overline{\mathbb Q}{}_{\ell})}$ is an \'etale $Z_H(\overline{\mathbb Q}_{\ell})$-gerbe over $S$. Here, $Z_H$ denotes the center of $H$, so $Z_H(\overline{\mathbb Q}{}_{\ell})$ is naturally an \'etale sheaf of abelian groups over $S$;
	\item $\epsilon$ is a homomorphism $\{\pm 1\} \rightarrow Z_H(\overline{\mathbb Q}{}_{\ell})$.
\end{enumerate}

The construction of $H$ has first appeared in the metaplectic Satake isomorphism \cite{MR2684259} \cite{MR2963537} and uses only the quadratic form $Q$ associated to an \'etale metaplectic cover under the second map of \eqref{eq-reductive-triangle-intro}. For example, the sheaf of characters of its maximal torus $T_H$ is precisely the kernel $\Lambda^{\sharp}\subset\Lambda$ of the associated symmetric form $b$. The subsheaf of roots $\Lambda^{\sharp, r}\subset\Lambda^{\sharp}$ is spanned by $\ord(Q(\alpha))\alpha$ for $\alpha\in\Delta$, etc. In particular, the character group of its center $\hat Z_H \cong \Lambda^{\sharp}/\Lambda^{\sharp, r}$ is a subquotient of $\Lambda$. The map $\epsilon$ is simply defined by the restriction of $Q$ to $\Lambda^{\sharp}$, which is $2$-torsion valued.

The construction of $\cal G_{Z_H(\overline{\mathbb Q}{}_{\ell})}$ poses the biggest challenge in both \cite{MR3769731} and Weissman \cite{MR3802418} (where it is termed the ``second twist''). Let us sketch its construction in our context, leaving the details to \S\ref{sec-metaplectic-dual}.
\end{void}

\begin{void}\emph{Construction of $\cal G_{Z_H(\overline{\mathbb Q}{}_{\ell})}$.}
We need an observation concerning \'etale metaplectic covers of tori, see \S\ref{void-tori-intro}. Namely, if the symmetric form $b$ associated to an \'etale metaplectic cover $\mu$ vanishes, then $\mu$ acquires the structure of an \emph{$\mathbb E_{\infty}$-monoidal} morphism $BT \rightarrow B^4A(1)$. Taking loop spaces yields an $\mathbb E_{\infty}$-monoidal morphism $T\rightarrow B^3A(1)$, and upon evaluation at $\mathbb G_m$ viewed as a constant group scheme over $S$, we obtain an $\mathbb E_{\infty}$-monoidal morphism $\Lambda \rightarrow B^2A$.

We shall temporarily assume that the reductive group scheme $G$ has a Borel subgroup $B\subset G$. Given an \'etale metaplectic cover $\mu$ of $G$, we apply the above observation to the torus $T^{\sharp} := \Lambda^{\sharp}\otimes\mathbb G_m$. From the commutative diagram \eqref{eq-reductive-classification-square-intro}, we shall obtain an $\mathbb E_{\infty}$-monoidal $\Lambda^{\sharp} \rightarrow B^2A$ trivialized over $\Lambda^{\sharp, r}$. This gives us an $\mathbb E_{\infty}$-monoidal morphism:
$$
F_{\mu} : \hat Z_H \rightarrow B^2A.
$$
Finally, we observe that the following inclusion:
$$
\underline{\Maps}{}_{\mathbb Z}(\hat Z_H, B^2A) \rightarrow \underline{\Maps}{}_{\mathbb E_{\infty}}(\hat Z_H, B^2A)
$$
admits a retraction, so $F_{\mu}$ defines a $\mathbb Z$-linear morphism $F_{\mu}^0 : \hat Z_H\rightarrow B^2A$. The induced $\mathbb Z$-linear morphism $\hat Z_H \rightarrow B^2(\overline{\mathbb Q}{}_{\ell}^{\times})$ is equivalent to a $Z_H(\overline{\mathbb Q}{}_{\ell})$-gerbe: this is $\cal G_{Z_H(\overline{\mathbb Q}{}_{\ell})}$.
\end{void}

\begin{void}\emph{Independence of $B$.}
In \S\ref{void-reductive-intro}, we have mentioned the dependence of the functor $\res_B$ on the Borel subgroup $B\subset G$. This fact haunts us again in the definition of $\cal G_{Z_H(\overline{\mathbb Q}{}_{\ell})}$.

However, we shall prove that $\cal G_{Z_H(\overline{\mathbb Q}{}_{\ell})}$ does \emph{not} depend on $B$, i.e.~it is functorially associated to the pair $(G, \mu)$. This fact allows us to remove the hypothesis on the existence of a Borel subgroup by \'etale descent.
\end{void}

\begin{rem}
\label{rem-function-field-comparison}
If the base scheme is a smooth curve $X$ over a field $k$, our $Z_H(\overline{\mathbb Q}{}_{\ell})$-gerbe differs slightly from the one constructed in \cite{MR3769731}.\footnote{Since \emph{op.cit.}~is limited to the split case and does not address the dependence on auxiliary data, our construction is slightly more precise.} Namely, the $\{\pm 1\}$-gerbe of \emph{$\vartheta$-characteristics}, i.e.~square roots of the canonical bundle $\omega_{X/k}$, induces a $Z_H(\overline{\mathbb Q}{}_{\ell})$-gerbe $\omega_X^{\epsilon}$ under the homomorphism $\epsilon$, and the gerbe of \emph{op.cit.}~is given by $\cal G_{Z_H(\overline{\mathbb Q}_{\ell})}\otimes\omega_X^{\epsilon}$.
\end{rem}

\begin{void}\emph{The $\tn L$-group.}
When $S$ is the spectrum of a field $F$ with a fixed algebraic closure $\bar F$, the metaplectic dual data $(H, \cal G_{Z_H(\overline{\mathbb Q}_{\ell})}, \epsilon)$ defines an extension of topological groups:
\begin{equation}
\label{eq-classical-L-group-intro}
1 \rightarrow H(\overline{\mathbb Q}{}_{\ell}) \rightarrow {}^LH_F \rightarrow \Gal(\bar F/F) \rightarrow 1,
\end{equation}
upon rigidifying $\cal G_{Z_H(\overline{\mathbb Q}_{\ell})}$ along $\Spec(\bar F)$. Here, we confuse $H$ with its fiber at $\Spec(\bar F)$.

Roughly speaking, this is because $\cal G_{Z_H(\overline{\mathbb Q}{}_{\ell})}$, equipped with its rigidification, defines a Galois $2$-cocycle with values in $Z_H(\overline{\mathbb Q}{}_{\ell})$. Such a $2$-cocycle defines an extension of $\Gal(\bar F/F)$ by $Z_H(\overline{\mathbb Q}{}_{\ell})$ compatible with the Galois action, which then induces \eqref{eq-classical-L-group-intro} along the Galois-equivariant inclusion $Z_H(\overline{\mathbb Q})\subset H(\overline{\mathbb Q}_{\ell})$.
\end{void}

\begin{rem}
If $F$ is a local field of characteristic $\neq 2$, our $\tn L$-group \eqref{eq-classical-L-group-intro} differs slightly from the one constructed in \cite{MR3802418}. Namely, ours is missing the twist by the meta-Galois group $\widetilde{\Gal}(\bar F/F)$, induced along the homomorphism $\epsilon$. There is an analogous discrepancy for a global field $F$.

This difference is in fact the same as the one in Remark \ref{rem-function-field-comparison}: the meta-Galois group \emph{is} the central extension associated to the $\{\pm 1\}$-gerbe of $\vartheta$-characteristics. The reason that we do not incorprate these twists in our definition of the metaplectic dual data, or the $\tn L$-groups, is that they have to do with special features of the base scheme but we prefer to give a purely group-theoretic definition. Incorporating them poses no difficulty in any case.
\end{rem}

\subsection{Metaplectic vs.~quantum}

\begin{void}
In the final \S\ref{sec-de-Rham}, we give a group-theoretic interpretation of the relationship between \'etale metaplectic covers and ``quantum parameters'' of the geometric Langlands program, taking place over a field $k$ of characteristic zero.

The notion of quantum parameters originated in the study of affine Kac--Moody Lie algebras, where it is given by the ``level'' of the central extension of a loop Lie algebra $\fr g\loo{t}$, see \cite{MR1104219}. When $\fr g$ arises as the Lie algebra of a split reductive group scheme $G$, these levels are classified by $G$-invariant symmetric forms $\kappa$ on $\fr g$. They correspond bijectively to the fourth de Rham cohomology classes of $BG$.

In the geometric Langlands program of Beilinson and Drinfeld \cite{beilinson1991quantization}, the base scheme is a smooth, geometrically connected curve $X$ over $k$. The study of Eisenstein series indicates that it is natural to include an extension of coherent sheaves:
$$
0 \rightarrow \Omega_{X/k}^1 \rightarrow E \rightarrow \fr g_{\ab}\otimes\cal O_X \rightarrow 0
$$
in the space of quantum parameters, where $\fr g_{\ab}$ denotes the abelianization of $\fr g$, see \cite{zhao2017quantum}.

The space of pairs $(\kappa, E)$ has the following group-theoretic intepretation: they are the rigidified fourth de Rham cochains on $BG$ living in the second Hodge filtrant, i.e.~the connective part of $\underline{\Gamma}{}_e(BG, \Omega^{\ge 2}[4])$. This definition can be stated over any smooth $k$-scheme $S$ and any reductive group scheme over $S$, not necessarily split.
\end{void}

\begin{void}
We have thus arrived at a number of notions of ``metaplectic covers'' of a reductive group scheme $G\rightarrow S$, all having to do with the fourth cohomology of $BG$. Let us summarize their relationship in the following diagram:
\begin{equation}
\label{eq-metaplectic-vs-quantum}
\begin{tikzcd}[column sep = 0.5em, row sep = 0.5em]
	\begin{Bmatrix}
		\text{central extensions} \\
		\text{of $G$ by $\underline K_2$}
	\end{Bmatrix}
	\ar[r, phantom, "\cong"] & \tau^{\le 0}\underline{\Gamma}{}_e(BG, \mathbb Z(2)[4]) \ar[d] \\
	& \tau^{\le 0}\underline{\Gamma}{}_e(BG, \mathbb Q(2)[4]) \ar[r]\ar[d] &
	\begin{Bmatrix}
		\text{quantum paramters} \\
		\text{associated to $G$}
	\end{Bmatrix}
	\\
	\begin{Bmatrix}
		\text{$\mathbb Q/\mathbb Z(1)$-valued \'etale} \\
		\text{metaplectic covers of $G$}
	\end{Bmatrix}
	\ar[r, phantom, "\cong"] & \tau^{\le 0}\underline{\Gamma}{}_e(BG, \mathbb Q/\mathbb Z(2)[4])
\end{tikzcd}
\end{equation}
Here, the column is given by the motivic cohomology groups of $BG$. Quantum parameters and \'etale metaplectic covers (generalized to infinite torsion abelian groups by taking colimits) arise from them by the de Rham, respectively \'etale realization functors. Diagram \eqref{eq-metaplectic-vs-quantum} is meant to provide some guidance in formulating the compatibility between the metaplectic and quantum geometric Langlands correspondences.

When the base field is $\mathbb C$, one has in addition a Betti realization functor of integral motivic cocycles on $BG$, whose natural target is the space of rigidified cocycles of the Deligne--Beilinson complex $\mathbb Z_D(2)[4]$ (or $\mathbb Q_D(2)[4]$) on the analytification $BG_{\an}$. These Deligne-Beilinson cocycles may be interpreted as classifying ``Chern--Simons theories'' with gauge group the compact real form of $G_{\an}$ and have been studied by Brylinski and McLaughlin \cite{MR1291698} \cite{MR1388845}.
\end{void}

\subsection*{Acknowledgements}

I thank Dennis Gaitsgory and Laurent Clozel for a number of conversations in 2020-21, which prompted me to document some facts about covering groups from a geometric perspective: this is the origin of the present article.

I thank K\k{e}stutis \v{C}esnavi\v{c}ius, Claudius Heyer, Hiroki Kato, and Ana-Maria Raclariu for helpful conversations.

Since the appearance of the first version of this article, titled ``metaplectic group schemes'', I have received generous feedback from Pierre Deligne, Wee Teck Gan, Naoki Imai, Tasho Kaletha, and Masanori Morishita. They have helped me improve the article in major ways. I am especially grateful to Deligne for his criticisms.

Finally, I thank Ildar Gaisin, Naoki Imai, and Teruhisa Koshikawa for their willingness to adopt the framework developed in the present article in an on-going investigation of covering groups over a $p$-adic field.

\medskip

\section{General topology}
\label{sec-general-topology}

The notion of \'etale metaplectic covers we develop is $2$-categorical in an essential way. The $2$-isomorphisms we shall encounter fall into two kinds: ones which exist for \emph{general} reasons and ones which exist for \emph{particular} reasons. While the bulk of the paper deals with the latter, we also need an effective language to handle the former. This language is provided by Lurie's theory of Higher Algebra.

The goal of this section is to record facts from Higher Algebra that we will use in the main body of the text, mostly without explicit mention.

\subsection{Structured spaces}

\begin{void}
\label{void-grothendieck-dictionary}
Let us begin with Grothendieck's dictionary between chain complexes in cohomological degrees $[-1, 0]$ and small, strictly commutative Picard groupoids (\cite[Expos\'e XVIII, \S1.4]{SGA4-3}). Recall: a symmetric monoidal groupoid $\cal A$ is a \emph{Picard groupoid} if the monoidal product $-\otimes a: \cal A\rightarrow\cal A$ is an equivalence for all $a\in\cal A$. It is \emph{strictly commutative} if the commutativity constraint $a_1\otimes a_2\cong a_2\otimes a_1$ is the identity map whenever $a_1 = a_2$.

To a chain complex $K^{-1} \xrightarrow{d} K^0$, we attach a small, strictly commutative Picard groupoid whose objects are elements $a\in K^0$ and there is an isomorphism $a_1\xrightarrow{\sim} a_2$ for each $f\in K^{-1}$ with $df = a_2 - a_1$. This association defines an equivalence of categories between chain complexes in cohomological degrees $[-1,0]$ and small, strictly commutative Picard groupoids (\cite[Expos\'e XVIII, Proposition 1.4.15]{SGA4-3}).

The central objects of this paper---\'etale metaplectic covers---form the $2$-categorical version of a small, strictly commutative Picard groupoid. Homotopical algebra provides us with the tools to concisely express the data defining them.
\end{void}

\begin{void}
Let $\Spc$ (resp.~$\Spc_*$) denote the $\infty$-category of (resp.~pointed) spaces. The stable $\infty$-category of spectra $\Sptr$ is the stabilization of $\Spc_*$ (c.f.~\cite[Proposition 1.4.2.24]{lurie2017higher}).

It is equipped with a canonical functor $\Omega^{\infty} : \Sptr \rightarrow \Spc_*$ and a $t$-structure such that $\Omega^{\infty}$ factors through the $\infty$-category $\Sptr^{\le 0}$ of connective spectra (\cite[Proposition 1.4.3.4]{lurie2017higher}). The resulting functor $\Sptr^{\le 0}\rightarrow \Spc_*$, still denoted by $\Omega^{\infty}$, has left adjoint $\Sigma_+^{\infty}$.
\end{void}

\begin{void}
Let us put the functor $\Omega^{\infty}$ in a different context. For an integer $0\le n\le \infty$, write $\mathbb E_n(\Spc)$ for the $\infty$-category of $\mathbb E_n$-monoid objects in $\Spc$. (See \cite[\S5.1]{lurie2017higher} for the definition of the $\mathbb E_n$-operad.) Equivalently, this is the $\infty$-category of $\mathbb E_n$-algebras in $\Spc$ with respect to the Cartesian symmetric monoidal structure (\cite[Proposition 2.4.2.5]{lurie2017higher}).

We refer to an object of $\mathbb E_n(\Spc)$ simply as an \emph{$\mathbb E_n$-space}. The $\infty$-category $\mathbb E_0(\Spc)$ is equivalent to $\Spc_*$.

For $n\ge 1$, an $\mathbb E_n$-space $\cal A$ is grouplike if and only if $\pi_0(\cal A)$ is a group with respect to the induced monoid structure (\cite[Definition 5.2.6.2, Example 5.2.6.4]{lurie2017higher}). We denote the $\infty$-category of grouplike $\mathbb E_n$-spaces by $\mathbb E_n^{\grp}(\Spc)$. A version of May's recognition theorem (\cite[Remark 5.2.6.26]{lurie2017higher}) states that there is a canonical equivalence of $\infty$-categories:
\begin{equation}
\label{eq-connective-spectra-grouplike-E-infinitey-spaces}
\Sptr^{\le 0}\cong\mathbb E_{\infty}^{\grp}(\Spc).
\end{equation}
Under \eqref{eq-connective-spectra-grouplike-E-infinitey-spaces}, the forgetful functor $\mathbb E_{\infty}^{\grp}(\Spc) \rightarrow \mathbb E_0(\Spc)$ corresponds to $\Omega^{\infty}$. This allows us to view connective spectra as spaces equipped with an ``algebraic'' structure.
\end{void}

\begin{void}
Every commutative ring $R$ may be viewed as an $\mathbb E_{\infty}$-algebra of $\Sptr$ with respect to the smash product (\cite[Remark 7.1.0.3]{lurie2017higher}). The $\infty$-categorical version of the Schwede--Shipley theorem (\cite[Theorem 7.1.2.13]{lurie2017higher}) compares the $\infty$-derived category $D(R)$ and the $\infty$-category $\Mod_R$ of $R$-module spectra. It states that there is a canonical equivalence of symmetric monoidal $\infty$-categories:
\begin{equation}
\label{eq-schwede-shipley-equivalence}
	D(R)\cong \Mod_R.
\end{equation}
Under \eqref{eq-schwede-shipley-equivalence}, the $t$-structure on $R$-module spectra corresponds to the natural $t$-structure on $D(R)$. In particular, the $\infty$-category of connective $R$-module spectra $\Mod_R^{\le 0}$ is equivalent to that of nonpositively graded complexes of $R$-modules $D^{\le 0}(R)$.

This last equivalence can be viewed as a generalization of Grothendieck's dictionary, where the role of small, strictly commutative Picard groupoids is played by connective $\mathbb Z$-module spectra. The equivalence in \S\ref{void-grothendieck-dictionary} is the special case for $R=\mathbb Z$, when both sides are restricted to $1$-coconnective objects.
\end{void}

\begin{void}
The following chain of forgetful functors relates all ``structured spaces'' which we shall consider in this paper:
\begin{align}
\label{eq-functors-structured-spaces}
	\Mod_{R}^{\le 0} \rightarrow \Mod_{\mathbb Z}^{\le 0} &\rightarrow \Sptr^{\le 0} \notag \\
	&\cong \mathbb E_{\infty}^{\grp}(\Spc) \rightarrow \mathbb E_1^{\grp}(\Spc) \rightarrow \Spc_* \rightarrow \Spc.
\end{align}
\end{void}

\begin{lem}
\label{lem-functors-structured-spaces-properties}
All functors in \eqref{eq-functors-structured-spaces} are conservative. They preserve sifted colimits and arbitrary limits.
\end{lem}
\begin{proof}
The first two functors preserve limits by \cite[Proposition 4.6.2.17]{lurie2017higher}. Since the $\mathbb E_{\infty}$-operad is coherent, we may apply \cite[Corollary 3.4.4.6]{lurie2017higher} to conclude that they preserve arbitrary colimits and deduce from \cite[Corollary 3.4.3.3]{lurie2017higher} that they are conservative.

For the functors on the bottom row, conservativity is a consequence of \cite[Lemma 3.2.2.6]{lurie2017higher}. For any integer $0\le n\le \infty$, the functor $\mathbb E_n(\Spc) \rightarrow \Spc$ preserves arbitrary limits (\cite[Corollary 3.2.2.4]{lurie2017higher}) and \emph{sifted} colimits (\cite[Proposition 3.2.3.1]{lurie2017higher}). Together with conservativity, this implies the same for the functors:
$$
\mathbb E_{\infty}(\Spc) \rightarrow \mathbb E_1(\Spc) \rightarrow \Spc_* \rightarrow \Spc.
$$
Finally, the full subcategory of grouplike objects in $\mathbb E_n(\Spc)$ ($n\ge 1$) is closed under arbitrary limits (by definition) and arbitrary colimits (\cite[Remark 5.2.6.9]{lurie2017higher}).
\end{proof}

\subsection{Sheaves}
\label{sect-sheaves}

\begin{void}
Let $\cal C$ be a site. Denote by $\PShv(\cal C)$ the $\infty$-category of presheaves of spaces on $\cal C$. It contains the full subcategory $\Shv(\cal C) \subset \PShv(\cal C)$ of sheaves of spaces, characterized by the property that for any covering sieve $\cal S\subset\Hom_{\cal C}(-, c)$, the canonical map:
$$
\cal F(c) \rightarrow \lim_{(f : c_1\rightarrow c)\in\cal S} F(c_1)
$$
is an equivalence. The $\infty$-category $\Shv(\cal C)$ is an $\infty$-topos in the sense of \cite[\S6]{MR2522659}.
\end{void}

\begin{void}
\label{void-chain-of-functors-sheaves}
For each $\infty$-category in \eqref{eq-functors-structured-spaces}, we may consider the $\infty$-category of (pre)sheaves valued in it. Since the functors connecting them preserve limits (Lemma \ref{lem-functors-structured-spaces-properties}), the corresponding functors on presheaves preserve the full subcategories of sheaves.

Lemma \ref{lem-functors-structured-spaces-properties} also implies that sheaves valued in $\mathbb E_{\infty}^{\grp}(\Spc)$ (resp.~$\mathbb E_1^{\grp}(\Spc)$ or $\Spc_*$) are canonically equivalent to grouplike $\mathbb E_{\infty}$-monoids (resp.~grouplike $\mathbb E_1$-monoids or pointed objects) in $\Shv(\cal C)$. In particular, \eqref{eq-functors-structured-spaces} gives rise to a chain of $\infty$-categories:
\begin{align}
\label{eq-sheaves-valued-in-structured-spaces}
	\Shv(\cal C, \Mod_{R}^{\le 0}) \rightarrow \Shv(\cal C, \Mod_{\mathbb Z}^{\le 0}) &\rightarrow \Shv(\cal C, \Sptr^{\le 0}) \notag \\
	&\cong \mathbb E_{\infty}^{\grp}(\Shv(\cal C)) \rightarrow \mathbb E_1^{\grp}(\Shv(\cal C)) \rightarrow \Shv_*(\cal C) \rightarrow \Shv(\cal C).
\end{align}
The functors in \eqref{eq-sheaves-valued-in-structured-spaces} are conservative and limit-preserving. We could add to \eqref{eq-sheaves-valued-in-structured-spaces} a limit-preserving (but evidently not conservative) functor $\Shv(\cal C, \Mod_R) \rightarrow \Shv(\cal C, \Mod_R^{\le 0})$ defined by $\tau^{\le 0}$ on the underlying presheaves.
\end{void}

\begin{void}
Let $\cal A$ be a symmetric monoidal $\infty$-category with arbitrary limits and colimits and such that the unit $\mathbf 1$ is both final and initial. Then for $n\ge 0$, the $\infty$-category of $\mathbb E_n$-algebras of $\cal A$ and the $\infty$-category of $\mathbb E_n$-algebras of $\cal A^{\mathrm{op}}$ (i.e.~$\mathbb E_n$-coalgebras of $\cal A$) are related by a pair of adjoint functors (\cite[Remark 5.2.3.6]{lurie2017higher}):
$$
\begin{tikzcd}[column sep = 1.5em]
	\mathrm{Bar}^{(n)} : \mathbb E_n(\cal A) \ar[r, shift left = 0.5ex] & \mathbb E_n(\cal A^{\mathrm{op}})^{\mathrm{op}} : \mathrm{Cobar}^{(n)}. \ar[l, shift left = 0.5ex]
\end{tikzcd}
$$
We shall apply this construction to $\cal A = \Shv(\cal C, \Mod_R^{\le 0})$, $\Shv(\cal C, \Sptr^{\le 0})$, $\Shv_*(\cal C)$, equipped with the \emph{Cartesian} symmetric monoidal structure.

In the first two cases, the Cartesian symmetric monoidal structure coincides with the co-Cartesian one, so the forgetful functor $\mathbb E_n(\cal A) \rightarrow \cal A$ is an equivalence (\cite[Proposition 2.4.3.9]{lurie2017higher}). Under this equivalence, $\mathrm{Bar}^{(n)}$ corresponds to $n$-fold suspension $[n]$ (\cite[Example 5.2.2.4]{lurie2017higher}). Since the functor $\Omega^{\infty} : \Shv(\cal C, \Sptr^{\le 0})\rightarrow \Shv_*(\cal C)$ is symmetric monoidal, the following diagram is commutative:
\begin{equation}
\label{eq-suspension-versus-bar}
\begin{tikzcd}
	\Shv(\cal C, \Sptr^{\le 0}) \ar[r, "{[n]}"]\ar[d, "\mathbb E_n(\Omega^{\infty})"] & \Shv(\cal C, \Sptr^{\le 0}) \ar[d, "\Omega^{\infty}"] \\
	\mathbb E_n^{\grp}(\Shv(\cal C)) \ar[r, "\mathrm{B}^n"] & \Shv_*(\cal C)
\end{tikzcd}
\end{equation}
Here, the functor $\mathrm B^n$ is obtained from $\mathrm{Bar}^{(n)}$ by composing with the functor forgetting the $\mathbb E_n$-coalgebra structure. We have an analogous commutative diagram when we replace $\Sptr^{\le 0}$ in the top row by $\Mod_R^{\le 0}$.
\end{void}

\begin{void}
\label{void-bar-construction-equivalence}
Given $\cal F\in\Shv(\cal C)$ and integer $n\ge 0$, we let $\pi_n\cal F$ denote the sheafification of the presheaf $c\mapsto \pi_n(\cal F(c))$. Write $\Shv(\cal C)_{\ge n}$ for the full subcategory of $\Shv(\cal C)$ consisting of objects $\cal F$ with $\pi_k\cal F = 0$ for all $0\le k < n$. Namely, this is the $\infty$-category of $n$-connective objects of $\Shv(\cal C)$. We also use the notation $\Shv_*(\cal C)_{\ge n}$ for the pointed version.

Because $\Shv(\cal C)$ is an $\infty$-topos, $B^{n}$ in \eqref{eq-suspension-versus-bar} defines an equivalence of $\infty$-categories onto $\Shv_*(\cal C)_{\ge n}$ (\cite[Theorem 5.2.6.15]{lurie2017higher}).
\end{void}

\begin{void}
\label{void-iterated-bar-construction-of-abelian-groups}
Let us mention a concrete description of $\mathrm B := \mathrm B^{1}$. Every $\mathbb E_1$-monoid in $\Shv(\cal C)$ defines a simplicial object $\cal F^{[n]}$ ($[n]\in\Delta^{\mathrm{op}}$) of $\Shv(\cal C)$ with $\cal F^{[0]} \cong \pt$. The functor $\tn B$ sends it to the geometric realization $\colim_{[n]} \cal F^{[n]}$, pointed by $\cal F^{[0]}$ (\cite[Example 5.2.6.13]{lurie2017higher}).

In particular, given a sheaf of groups $H$ on $\cal C$, $\mathrm B(H)$ is the classifying stack of $H$ in the classical sense. For a sheaf of abelian groups (resp.~$R$-modules) $A$ on $\cal C$ and $n\ge 1$, we view $\mathrm B^{n}(A)$ as the $n$-fold classifying stack of $A$. The commutative diagram \eqref{eq-suspension-versus-bar} tells us that $\mathrm B^{n}(A)$ is canonically identified with the image of $A[n]$ under $\Omega^{\infty}$. In particular, $\mathrm B^{n}(A)$ inherits the structure of a sheaf valued in $\Sptr^{\le 0}$ (resp.~$\Mod_R^{\le 0}$).
\end{void}

\begin{void}
\label{void-monoidal-morphisms-classified-by-groups}
For later purposes, we shall describe the mapping space $\Maps_{\mathbb E_1}(H, \mathrm B^{n+1}(A))$ (where $n\ge 1$) for $H$ a sheaf of groups and $A$ a sheaf of abelian groups. Taking fibers at $* \rightarrow \mathrm B^{n+1}(A)$ defines a functor:
\begin{equation}
\label{eq-monoidal-morphisms-classified-by-groups}
\Maps_{\mathbb E_1}(H, \mathrm B^{n+1}(A)) \rightarrow \mathbb E_1(\Shv(\cal C))_{\pi_0 = H, \pi_n = A}.
\end{equation}
The target denotes the $\infty$-category of objects $\cal G\in\mathbb E_1(\Shv(\cal C))$ equipped with isomorphisms $\pi_0\cal G \cong H$, $\pi_n\cal G \cong A$, and satisfying two properties:
\begin{enumerate}
	\item $\pi_k\cal G = 0$ for all $k\neq 0,n$;
	\item the Postnikov tower of $\cal G$ is convergent (see \cite[\S5.5.6]{MR2522659}).
\end{enumerate}
Such $\cal G$ is automatically grouplike.
\end{void}

\begin{lem}
\label{lem-monoidal-morphisms-classified-by-groups}
For each $n\ge 1$, the functor \eqref{eq-monoidal-morphisms-classified-by-groups} is an equivalence.
\end{lem}
\begin{proof}
The functor $\mathrm B$ defines an equivalence of mapping spaces (\S\ref{void-bar-construction-equivalence}):
$$
\Maps_{\mathbb E_1}(H, \mathrm B^{n+1}(A))\cong \Maps_*(\mathrm B(H), \mathrm B^{n+2}(A)).
$$
Using Postnikov truncation, the latter space is equivalent to the space of objects $\cal G_1\in\Shv_*(\cal C)$ with a convergent Postnikov tower and equipped with isomorphisms $\pi_1(\cal G_1) \cong H$, $\pi_{n+1}(\cal G_1) \cong A$, and $\pi_k(\cal G_1) = 0$ for all $k\neq 1, n+1$. Applying the inverse of $\mathrm B$ yields the desired equivalence.
\end{proof}

\begin{rem}
\label{rem-monoidal-morphisms-classified-by-groups}
\begin{enumerate}
	\item For $n = 0$, the space of $\mathbb E_1$-monoidal morphisms $H \rightarrow \mathrm B(A)$ is equivalent to the groupoid of central extensions of $H$ by $A$. By analogy, one may interpret Lemma \ref{lem-monoidal-morphisms-classified-by-groups} as saying that for $n\ge 1$, any $\mathbb E_1$-monoidal extension of $H$ by $\mathrm B^{n}A$ is automatically ``central'' (although we do not define this term).
	\item For $n\ge 1$ and $H$ abelian, one may consider $\Maps_{\mathbb E_k}(H, \mathrm B^{n+1}(A))$ for any $0\le k\le\infty$. This space is equivalent to $\mathbb E_k$-objects of $\Shv(\cal C)$ equipped with isomorphisms $\pi_0\cal G \cong H$, $\pi_n\cal G \cong A$, and satisfying the conditions of \S\ref{void-monoidal-morphisms-classified-by-groups}. Indeed, for finite $k$ one argues as in Lemma \ref{lem-monoidal-morphisms-classified-by-groups}. For $k = \infty$, one argues instead in $\Sptr^{\le 0}$.
\end{enumerate}
\end{rem}

\subsection{Complexes}
\label{sect-complexes}

\begin{void}
In light of \eqref{eq-schwede-shipley-equivalence}, an object of the $\infty$-category $\Shv(\cal C, \Mod_R)$ can be viewed as a sheaf of complexes of $R$-modules. In the classical context, we are more accustomed to complexes of sheaves of $R$-modules. Let us explain how these two points of views are related.

We continue to fix a site $\cal C$ and a commutative ring $R$. The construction of $\pi_n\cal F$ for an object $\cal F\in\Shv(\cal C, \Mod_R)$ defines a $t$-structure on $\Shv(\cal C, \Mod_R)$. Its heart is identified with the category of sheaves of $R$-modules on $\cal C$. According to \cite[Corollary 2.1.2.4]{lurie2018spectral}, we have a fully faithful, $t$-exact functor:
\begin{equation}
\label{eq-complexes-of-sheaves-to-sheaves-of-complexes}
D^+(\Shv(\cal C, \Mod_R)^{\heartsuit}) \rightarrow \Shv(\cal C, \Mod_R).
\end{equation}
The essential image of \eqref{eq-complexes-of-sheaves-to-sheaves-of-complexes} is the full subcategory $\Shv(\cal C, \Mod_R)^{>-\infty}$ of left-bounded objects.
\end{void}

\begin{void}
Let $\mathrm{ob} : \Shv(\cal C, \Mod_R) \rightarrow \PShv(\cal C, \Mod_R)$ denote the forgetful functor. Since $\mathrm{ob}$ is a right adjoint, it is left $t$-exact. We denote by $R(\mathrm{ob}^{\heartsuit})$ the right derived functor of its truncation. Explicitly, given a left-bounded complex $\cal F$ of sheaves of $R$-modules, $R(\mathrm{ob}^{\heartsuit})(\cal F)$ is the presheaf whose value at $c\in\cal C$ is the complex of $R$-modules $R\Gamma(c, \cal F)$.

The following commutative diagram arises from the analogous diagram for the sheafification functor by passing to the right adjoint:
\begin{equation}
\label{eq-forgetful-functor-derived}
\begin{tikzcd}[column sep = 1em]
	D^+(\Shv(\cal C, \Mod_R)^{\heartsuit}) \ar[r, "\cong"]\ar[d, "R(\mathrm{ob}^{\heartsuit})"] & \Shv(\cal C, \Mod_R)^{>-\infty} \ar[d, "\mathrm{ob}"] \\
	D^+(\PShv(\cal C, \Mod_R)^{\heartsuit}) \ar[r, "\cong"] & \PShv(\cal C, \Mod_R)^{>-\infty}
\end{tikzcd}
\end{equation}
It implies that the image of $\cal F$ under \eqref{eq-complexes-of-sheaves-to-sheaves-of-complexes} has an explicit description: its value at $c\in\cal C$ is the complex of $R$-modules $R\Gamma(c, \cal F)$.
\end{void}

\begin{lem}
\label{lem-homotopy-groups-iterated-bar-construction}
Suppose that $A$ is a sheaf of $R$-modules on $\cal C$. For integers $0\le k\le n$ and $c\in\cal C$, there is a canonical isomorphism of $R$-modules:
$$
\pi_k\Gamma(c, B^{n}(A)) \cong H^{n-k}(c, A).
$$
\end{lem}
\begin{proof}
This follows from \eqref{eq-forgetful-functor-derived} and the discussion in \S\ref{void-iterated-bar-construction-of-abelian-groups}.
\end{proof}

\medskip

\section{\'Etale metaplectic covers}
\label{sec-metaplectic-covers}

Let $S$ be a scheme and $G \rightarrow S$ be a group scheme of finite type. Suppose that $A$ is a locally constant \'etale sheaf of finite abelian groups on $S$ whose order is invertible on $S$.

This is the general context in which we will define metaplectic covers.

\begin{defn}
An \emph{\'etale metaplectic cover} of $G$ with values in $A$ is a morphism of pointed \'etale stacks $BG \rightarrow B^4A(1)$.

Thus, $A$-valued \'etale metaplectic covers of $G$ form the groupoid $\Gamma_e(BG, B^4A(1))$, viewed as sections of $B^4A(1)$ over $BG$ rigidified along $e : S\rightarrow BG$ (i.e.~equipped with a trivialization after pulling back by $e$).
\end{defn}

Here, the Tate twist $A(1)$ is defined as $A(1) := A\otimes\mathbb G_m[-1]$ or equivalently $\lim_N(A\otimes\mu_N)$, where the transition maps are given by $\mu_{N_1} \twoheadrightarrow \mu_N$, $a\mapsto a^{N_1/N}$ whenever $N\mid N_1$. Then $A(1)$ is again a locally constant \'etale sheaf of finite abelian groups on $S$. Likewise, we write $A(-1)$ for $\colim_N(A\otimes\mu_N^{\otimes -1})$.

In the remainder of this section, we first explain how to obtain covering groups in the classical sense from \'etale metaplectic covers. This construction is a reformulation of Deligne \cite[\S2, \S5-6]{MR1441006}, although the language of higher categories allows us to avoid direct contact with resolutions. Then we explain in \S\ref{sect-relation-with-K2} the relationship between \'etale metaplectic covers and central extensions of $G$ by $\underline K_2$, which may be viewed as an analogue of our theory with ``integral coefficients''.

\begin{rem}
\label{rem-base-point-free}
It is sometimes natural to forget the distinguished point of the \'etale stack $BG$. Indeed, a morphism of \'etale stacks $\mu : BG \rightarrow B^4A(1)$ automatically induces an \'etale metaplectic cover of $G_P := \underline{\Aut}(P)$ for any $G$-torsor $P\rightarrow S$: writing $S\xrightarrow{e_P} BG \xrightarrow{p} S$ where $e_P$ corresponds to $P$, we form $\mu_P := \mu\otimes(p^*e_P^*\mu)^{\otimes -1}$.

We may view morphisms $BG \rightarrow B^4A(1)$ as defining \'etale metaplectic covers on all pure inner forms of $G$ at once. For a chosen $G$-torsor $P\rightarrow S$, we have a projection $\mu \mapsto \mu_P$ to the groupoid of \'etale metaplectic covers of $G_P$, for which the functor of forgetting the pointing provides a section.
\end{rem}

\subsection{Local fields}
\label{sect-metaplectic-cover-local}

\begin{void}
Suppose that $S$ is the spectrum of a nonarchimedean local field $F$. Fix an algebraic closure $F\subset \bar F$ which defines the Galois group $\Gal(\bar F/F)$. In particular, $A$ may be viewed as a finite abelian group equipped with a $\Gal(\bar F/F)$-action.

From a pointed morphism $BG \rightarrow B^4A(1)$, we shall functorially attach a central extension of topological groups:
\begin{equation}
\label{eq-classical-metaplectic-cover-local}
	1 \rightarrow H_0(F, A) \rightarrow \tilde G \rightarrow G(F) \rightarrow 1,
\end{equation}
where $H_0(F, A)$ denotes the $\Gal(\bar F/F)$-coinvariants of $A$, i.e.~the zeroth homology group, and is equipped with the discrete topology.

In the case $F = \mathbb R$, we will in general only obtain a central extension of a subgroup of $G(\mathbb R)$ which contains the neutral component; this will be addressed in \S\ref{void-metaplectic-cover-local-construction-real}. The case $F = \mathbb C$ will give canonically split extensions of $G(\mathbb C)$.
\end{void}

\begin{void}
The construction of \eqref{eq-classical-metaplectic-cover-local} uses local Tate duality \cite[Theorem 2.1]{MR0175892}. Write $A^* := \underline{\Hom}(A, \mathbb Q/\mathbb Z)$. Then $A^*$ is in Cartier duality with $A(1)$, so cup product defines a perfect pairing on Galois cohomology groups:
\begin{equation}
\label{eq-tate-duality-pairing}
H^0(F, A^*) \otimes H^2(F, A(1)) \rightarrow H^2(F, \mathbb G_m) \cong \mathbb Q/\mathbb Z,
\end{equation}
where the second map is given by Brauer invariants.

Therefore, $H^2(F, A(1))$ is identified with the Pontryagin dual of $H^0(F, A^*)$, which agrees with the homology group $H_0(F, A)$. We summarize this isomorphism together with the vanishing statements of \cite[Theorem 2.1]{MR0175892}:
\begin{equation}
\label{eq-tate-duality-homology}
H^i(F, A(1)) \cong
\begin{cases}
	H_0(F, A) & i = 2 \\
	0 & i\ge 3
\end{cases}
\end{equation}
\end{void}

\begin{void}
Given a pointed morphism $\mu : BG \rightarrow B^4A(1)$, we may form an $\mathbb E_1$-monoidal morphism $\Omega(\mu) : G \rightarrow B^3A(1)$ by taking loop stacks. Evaluation at $\Spec(F)$ then yields maps of $\mathbb E_1$-monoidal groupoids:
\begin{align}
\notag
G(F) &\rightarrow  \Gamma(F, B^3A(1))\\
& \cong B(\Gamma(F, B^2A(1))) \xrightarrow{\pi_0} B(H^2(F, A(1))) \cong B(H_0(F, A)).
\label{eq-construction-classical-metaplectic-cover-local}
\end{align}
where the isomorphisms owe to the calculations \eqref{eq-tate-duality-homology}. This composition defines a central extension \eqref{eq-classical-metaplectic-cover-local} as an abstract group.
\end{void}

\begin{void}
In order to equip $\tilde G$ with a topology making the maps in \eqref{eq-classical-metaplectic-cover-local} continuous, it suffices to construct a family of local sections of $\tilde G\rightarrow G(F)$ satisfying the properties (a)--(c) of \cite[\S2.9]{MR1441006}.

Let us paraphrase the construction of \emph{op.cit.} in our language. Consider the $\mathbb E_1$-monoidal morphism $\Omega(\mu) : G \rightarrow B^3A(1)$ corresponding to $\mu$ and write $G^{\dagger}$ for the fiber of $\Omega(\mu)$ at the base point. Then $G^{\dagger}$ is an \'etale sheaf of $2$-groupoids and the projection $G^{\dagger}(F) \rightarrow G(F)$ factors through $\tilde G$.

Pullback along any $g\in G(F)$ defines a section:
$$
g^*\Omega(\mu) \in \Gamma(F, B^3A(1)) \cong \colim_{\Spec(F)\rightarrow U} \Gamma(U, B^3A(1)),
$$
the colimit being taken over \'etale neighborhoods $U$ of $g$. The null-homotopy of $g^*\Omega(\mu)$ is thus defined over some such $U$. By construction, $U \rightarrow G$ lifts to $G^{\dagger}$, giving a commutative diagram:
$$
\begin{tikzcd}[column sep = 0.5em, row sep = 1.5em]
	& G^{\dagger}(F)\ar[d] \\
	U(F) \ar[ur]\ar[dr] & \tilde G \ar[d] \\
	& G(F)
\end{tikzcd}
$$

Since $U\rightarrow G$ is \'etale and $g\in G(F)$ lifts to $U(F)$, we find local sections of $U(F)\rightarrow G(F)$ at $g$ by the implicit function theorem. (This step uses the finite type hypothesis on $G$.) They define sections for the projection $\tilde G\rightarrow G(F)$.
\end{void}

\begin{rem}
Let $\cal O\subset F$ denote the ring of integers and suppose that $A$ extends to a locally constant \'etale sheaf of finite abelian groups on $\Spec(\cal O)$ with invertible order, i.e.~indivisible by the residue characteristic.

If $G$ is an affine group scheme over $\Spec(\cal O)$, then the central extension \eqref{eq-classical-metaplectic-cover-local} acquires a canonical splitting over the subgroup $G(\cal O)\subset G(F)$. Indeed, the vanishing of $H^2(\cal O, A(1))$ implies that \eqref{eq-construction-classical-metaplectic-cover-local} is null-homotopic when restricted to $G(\cal O)$.
\end{rem}

\begin{void}
\label{void-metaplectic-cover-local-construction-real}
When $F = \mathbb R$, we replace \eqref{eq-tate-duality-pairing} by the perfect pairing on Tate cohomology groups:
\begin{equation}
\label{eq-tate-duality-real-place}
\hat H^0(\mathbb R, A^*)\otimes H^2(\mathbb R, A(1)) \rightarrow H^2(\mathbb R, \mathbb G_m) \cong \frac{1}{2}\mathbb Z/\mathbb Z.
\end{equation}
Recall that $\hat H^0(\mathbb R, A)$ is defined by an exact sequence involving the norm map:
\begin{equation}
\label{eq-tate-cohomology-groups}
0\rightarrow\hat H^{-1}(\mathbb R, A) \rightarrow H_0(\mathbb R, A) \xrightarrow{\mathrm{Nm}} H^0(\mathbb R, A)\rightarrow\hat H^0(\mathbb R, A)\rightarrow 0.
\end{equation}
We obtain from \eqref{eq-tate-duality-real-place} the following isomorphism:
\begin{equation}
\label{eq-tate-duality-real-place-isomorphism}
H^2(\mathbb R, A(1)) \cong \hat H^{-1}(\mathbb R, A).
\end{equation}

The other change in the construction has to do with possibly nonzero $H^3(\mathbb R, A(1))$. For a pointed morphism $\mu : BG \rightarrow B^4A(1)$, the induced $\mathbb E_1$-monoidal map $\Omega(\mu) : G \rightarrow B^3A(1)$ defines a homomorphism:
\begin{equation}
\label{eq-classical-metaplectic-construction-real-obstruction}
G(\mathbb R) \rightarrow H^3(\mathbb R, A(1)).
\end{equation}

We write $G(\mathbb R)^0$ for the kernel of \eqref{eq-classical-metaplectic-construction-real-obstruction}. Then the same construction as above, using \eqref{eq-tate-duality-real-place-isomorphism} instead of \eqref{eq-tate-duality-pairing}, gives a central extension of topological groups:
\begin{equation}
\label{eq-classical-metaplectic-cover-real}
1 \rightarrow \hat H^{-1}(\mathbb R, A) \rightarrow \tilde G \rightarrow G(\mathbb R)^0 \rightarrow 1.
\end{equation}
Finally, we also obtain a central extension of $G(\mathbb R)^0$ by $H_0(\mathbb R, A)$ by inducing along the first map in \eqref{eq-tate-cohomology-groups}.
\end{void}

\subsection{Global fields}
\label{sect-metaplectic-cover-global}

\begin{void}
We now suppose that $S$ is the spectrum of a global field $F$ and $G$ is affine. As before, we fix an algebraic closure $F\subset \bar F$. Let $\mathbb A_F$ denote the topological ring of ad\`eles of $F$.

Let $\mu : BG \rightarrow B^4A(1)$ be a pointed morphism over $\Spec(F)$. Then restriction of $\mu$ to each place $v$ of $F$ defines a pointed morphism $\mu_v : BG_{F_v} \rightarrow B^4A(1)$ over $\Spec(F_v)$. In particular, we find central extensions $\tilde G_v$ of $G(F_v)^0$, where the notation $G(F_v)^0$ stands for $G(F_v)$ for $v$ nonarchimedean or complex, and for the kernel of \eqref{eq-classical-metaplectic-construction-real-obstruction} for $v$ real.
\end{void}

\begin{void}
In this section, we shall combine these $\tilde G_v$ into a central extension of the ad\`elic points of $G$. More precisely, we functorially attach to $\mu$ a central extension:
\begin{equation}
\label{eq-classical-metaplectic-cover-global}
1 \rightarrow H_0(F, A) \rightarrow \tilde G \rightarrow G(\mathbb A_F)^0 \rightarrow 1,
\end{equation}
equipped with a canonical splitting over $G(F)^0$.

Here, $G(\mathbb A_F)^0 \subset G(\mathbb A_F)$ (resp.~$G(F)^0\subset G(F)$) denotes the subgroup of elements whose restriction along each real place $v$ lies in $G(F_v)^0$.
\end{void}

\begin{void}
Write $\cal O_F\subset F$ for the ring of integers. Fix a finite nonempty collection of places $v\in F$ including all the archimedean ones. Its complement may be viewed as an open subscheme $V\subset\Spec(\cal O_F)$.

We shall assume that $A$ extends to a locally constant \'etale sheaf over $V$, with order invertible on $V$. Such an extension is unique, as the morphism from $\Gal(\bar F/F)$ to the \'etale fundamental group of $V$ is surjective, see \cite[0BSD]{stacks-project}.

There is an exact sequence:
\begin{equation}
\label{eq-second-cohomology-global-field}
	H^2(V, A(1)) \rightarrow \bigoplus_{v\notin V} H^2(F_v, A(1)) \rightarrow H_0(F, A) \rightarrow 0,
\end{equation}
where the first map is the restriction, and the second map is the sum of local duality maps, composed with the projections $H_0(F_v, A) \twoheadrightarrow H_0(F, A)$. This follows from \cite[Theorem 3.1]{MR0175892}, in view of the identification $H^0(V, A^*) \cong H^0(F, A^*)$.

The same Theorem yields the calculation of $H^3(V, A(1))$, i.e.~along the restriction maps, we have an isomorphism:
\begin{equation}
\label{eq-third-cohomology-global-field}
H^3(V, A(1)) \cong \bigoplus_{v\tn{ real}}H^3(F_v, A(1)).
\end{equation}
\end{void}

\begin{void}
To construct \eqref{eq-classical-metaplectic-cover-global}, we need to extend the affine group scheme $G$ and the pointed morphism $\mu$ to an open subscheme of $\Spec(\cal O_F)$.

To address the canonicity of the construction, we consider the filtered category of triples $(U, G_U, \gamma)$ where $U\subset \Spec(\cal O_F)$ is an open subscheme where $A$ extends as a locally constant sheaf, $G_U \rightarrow U$ is a affine group scheme of finite type, and $\gamma$ is an isomorphism $G \cong (G_U)\times_U\Spec(F)$ of affine group schemes over $\Spec(F)$.
\end{void}

\begin{lem}
\label{lem-integral-model}
The functor defined by restriction along $\Spec(F)\rightarrow U$:
$$
\colim_{(U, G_U, \gamma)} \Gamma_e(BG_U, B^4A(1)) \rightarrow \Gamma_e(BG, B^4A(1))
$$
is an equivalence.
\end{lem}
\begin{proof}
The existence of an integral model implies that $G$ is identified with $\lim_{(U, G_U, \gamma)}G_U$. The assertion of the Lemma, when $BG$ is replaced by $G^n$ (for $[n]\in\Delta^{\mathrm{op}}$), follows from the local finite presentation of the stack of \'etale local systems (\cite[0GL2]{stacks-project}).

Since $B^4A(1)$ is $4$-truncated, the space of sections $\Gamma_e(BG, B^4A(1))$ is computed by a cosimplicial limit over a finite subcategory of $\Delta^{\mathrm{op}}$, so we conclude by the commutation of filtered colimits with finite limits.
\end{proof}

\begin{void}
Using Lemma \ref{lem-integral-model}, it suffices to construct a functor from $\Gamma_e(BG_U, B^4A(1))$ to the groupoid of central extensions \eqref{eq-classical-metaplectic-cover-global} equipped with a splitting over $G(F)^0$, which is moreover natural in the triplet $(U, G_U, \gamma)$.

Indeed, we consider a pointed morphism $\mu_U : BG_U \rightarrow B^4A(1)$. For any open subscheme $V\subset U$, we obtain a central extension of the topological group $\prod_{v\notin V} G(F_v)^0 \times \prod_{v\in V}G(\cal O_v)$ by taking the product of $\tilde G_v$ over $v\notin V$:
\begin{align*}
0 \rightarrow \bigoplus_{v\notin V} H^2(F_v, A(1)) &\rightarrow \prod_{v\notin V}\tilde G_v \times \prod_{v\in V}G_U(\cal O_v) \\
&\rightarrow \prod_{v\notin V}G(F_v)^0 \times \prod_{v\in V}G_U(\cal O_v) \rightarrow 1.
\end{align*}

Taking the push-out along the second map in \eqref{eq-second-cohomology-global-field}, we obtain a central extension $\tilde G_{U, V}$ of $\prod_{v\notin V}G(F_v)^0\times \prod_{v\in V}G_U(\cal O_v)$ by $H_0(F, A)$.

This central extension is canonically split over $G_U(V)^0$, the kernel of the map $G_U(V) \rightarrow H^3(V, A(1))$. Indeed, $G_U(V)^0$ equals the subgroup whose restriction to the real places belong to $G(F_v)^0$, using the commutative diagram:
$$
\begin{tikzcd}[column sep = 1em]
	G_U(V) \ar[r]\ar[d] & H^3(V, A(1)) \ar[d, "\cong"] \\
	\prod_{v\tn{ real}}G(F_v) \ar[r] & \bigoplus_{v\tn{ real}}H^3(F_v, A(1))
\end{tikzcd}
$$
where the isomorphism is \eqref{eq-third-cohomology-global-field}. Furthermore, the composition of the first two maps in \eqref{eq-second-cohomology-global-field} vanishes, so the central extension of $G_U(V)^0$ by $H^2(V, A(1))$ constructed from the recipe of \S\ref{sect-metaplectic-cover-local} induces the split extension by $H_0(F, A)$.

In summary, we obtain a diagram of topological groups:
\begin{equation}
\label{eq-metaplectic-cover-global-construction-finite}
\begin{tikzcd}[column sep = 1em]
	& &  & G_U(V)^0 \ar[d]\ar[dl] \\
	1 \ar[r] & H_0(F, A) \ar[r] & \tilde G_{U, V} \ar[r] & \prod_{v\notin V}G(F_v)^0 \times \prod_{v\in V}G_U(\cal O_v) \ar[r] & 1
\end{tikzcd}
\end{equation}
The central extension \eqref{eq-classical-metaplectic-cover-global}, together with its splitting over $G(F)^0$, is then obtained by taking colimit of \eqref{eq-metaplectic-cover-global-construction-finite} over $V\subset U$.
\end{void}

\begin{eg}
Suppose that $N\ge 1$ is an integer and $S$ is a $\mathbb Z[\frac{1}{N}]$-scheme. Consider the constant group scheme $G = \SL_2$ over $S$. Viewed as an element of the reduced cohomology, the universal second Chern class $[c_2] \in H^4_e(B(\SL_2), \mu_N^{\otimes 2})$ lifts to a pointed morphism $c_2 : B(\SL_2) \rightarrow B^4(\mu_N^{\otimes 2})$; this follows from the vanishing of the reduced cohomology groups in degrees $\le 3$, see Proposition \ref{prop-reductive-classification} below.

If $S = \Spec(F)$ where $F$ is a local or global field containing a primitive $N$th root of unity, then $c_2$ defines an $N$-fold covering group of $\SL_2(F)$, respectively $\SL_2(\mathbb A_F)$. This is isomorphic to Kubota's central extension \cite{MR204422}.
\end{eg}

\begin{rem}
More generally, classical examples of covering groups correspond to $A$ being the constant sheaf of roots of unity $\mu(F)$ in a local or global field $F$ (or a subgroup thereof). There is an alternative algebraic description of such covers, given by central extensions of $G$ by the Zariski sheaf $\underline K_2$, see \S\ref{sect-relation-with-K2} below.

In the study of genuine representations of covering groups, one typically fixes a coefficient ring, say $\mathbb C$, and an inclusion $\mu(F)\subset\mathbb C^{\times}$. Our perspective is slightly different: without assuming that $F$ contains enough roots of unity, we may still consider \'etale metaplectic covers with values in a constant abelian group $A\subset\mathbb C^{\times}$ of order invertible in $F$. The covering groups arising this way appear in Kaletha's work on Langlands functoriality \cite{kaletha2022}.
\end{rem}

\subsection{Integral metaplectic covers}
\label{sect-relation-with-K2}

\begin{void}
Suppose that $S$ is a smooth scheme over a field $F$. Let $\underline K_2$ denote the Zariski sheafification of the functor $R\mapsto K_2(R)$ sending a commutative ring to its second algebraic K-group. The Picard groupoid of central extensions of sheaves of groups on the big Zariski site of $S$:
\begin{equation}
\label{eq-brylinski-deligne-extension}
1 \rightarrow \underline K_2 \rightarrow E \rightarrow G \rightarrow 1
\end{equation}
is equivalent to $\Gamma_e(BG_{\Zar}, B^2\underline K_2)$, where $BG_{\Zar}$ stands for the stack of Zariski locally trivial $G$-torsors and $B^2\underline K_2$ is the delooping of $\underline K_2$ as a \emph{Zariski sheaf}.

If $S = \Spec(F)$ for a local or global field containing a primitive $N$th root of unity, \cite[\S10]{MR1896177} explains how to attach to \eqref{eq-brylinski-deligne-extension} a topological central extension of $G(F)$ (respectively $G(\mathbb A_F)$) by $\mu_N(F)$.

We shall show that when $G$ is reductive, this process factors through the space of \'etale metaplectic covers of $G$ with values in $A:=\mu_N$.
\end{void}

\begin{void}
Let $F$ be a field and $\Sm_{/F}$ be the category of smooth $F$-schemes. In this context, we shall construct a functor for any reductive group scheme $G\rightarrow S$ and $S\in\Sm_{/F}$:
\begin{equation}
\label{eq-etale-realization-integral-parameter}
\Gamma_e(BG_{\Zar}, B^2\underline K_2) \rightarrow \Gamma_e(BG, B^4\mu_N^{\otimes 2}).
\end{equation}

The Galois symbol defines a morphism from $\underline K_2$ to the Zariski sheafification of the presheaf $R\mapsto H^2_{\et}(R, \mu_N^{\otimes 2})$. However, this map does not lift to a morphism $\underline K_2\rightarrow B^2\mu_N^{\otimes 2}$.

The idea to circumvent this difficulty, due to Gaitsgory \cite[\S6]{MR4117995}, is to interpret the left-hand-side of \eqref{eq-etale-realization-integral-parameter} as sections of the motivic complex $\mathbb Z(2)[4]$ over $BG$, which then admits an \'etale realization given by ``mod $N$''.
\end{void}

\begin{void}
\label{void-motivic-complex}
We temporarily assume that $F$ is \emph{perfect}. For $n\ge 0$, we denote by $\mathbb Z(n)$ the complex of presheaves of abelian groups on $\Sm_{/F}$ defined in \cite[\S3]{MR2242284}.

To briefly recall its construction, we write $\mathbb Z_{\tr}(X)$ (for $X\in\Sm_{/F}$) for the presheaf whose $Y$-points are finite correspondences from $Y$ to $X$, i.e.~the $\mathbb Z$-span of integral subschemes $Y_1\subset Y\times X$ such that $Y_1\rightarrow Y$ is finite and surjective (called an ``elementary correspondence''). Denote by $\mathbb Z_{\tr}(\wedge^n\mathbb G_m)$ the cokernel:
$$
\bigoplus_{1\le i\le n}\mathbb Z_{\tr}(\mathbb G_m \times \cdots \times\{1\}\times \cdots\times\mathbb G_m) \rightarrow \mathbb Z_{\tr}(\mathbb G_m^{\times n}) \rightarrow \mathbb Z_{\tr}(\wedge^n\mathbb G_m) \rightarrow 0.
$$

The inclusion of $\mathbb A^1$-invariant complexes of presheaves into all complexes of presheaves on $\Sm_{/F}$ admits a left adjoint $L_{\mathbb A^1}$, and we set:
$$
\mathbb Z(n) := L_{\mathbb A^1}\mathbb Z_{\tr}(\wedge^n\mathbb G_m)[-n].
$$
Then $\mathbb Z(n)$ is a complex of \'etale sheaves of abelian groups on $\Sm_{/F}$ \cite[Corollary 6.4]{MR2242284}.
\end{void}

\begin{rem}
The complexes $\mathbb Z(n)$ are concentrated in cohomological degrees $\le n$, but are in general not left-bounded if $n\ge 2$.

Clearly $\mathbb Z(0) \cong \mathbb Z$. There is also an equivalence $\mathbb Z(1) \cong \mathbb G_m[-1]$ induced from the natural map $\mathbb Z_{\tr}(\mathbb G_m)\rightarrow \mathbb G_m$ sending an elementary correspondence $Y_1\subset Y\times \mathbb G_m$ to the norm of the corresponding element of $\cal O_{Y_1}^{\times}$ relative to $Y_1\rightarrow Y$, see \cite[\S4]{MR2242284}.
\end{rem}

\begin{void}
\label{void-motivic-complex-to-K-theory}
We shall construct a morphism of complexes of Zariski sheaves on $\Sm_{/F}$:
\begin{equation}
\label{eq-motivic-complex-to-K-theory}
\mathbb Z(n)[n] \rightarrow \underline{K}_n.
\end{equation}

Let us start with a map $\mathbb Z_{\tr}(\mathbb G_m^{\times n}) \rightarrow \underline K_n$ of presheaves. For an elementary correspondence $Y_1 \subset Y\times\mathbb G_m^{\times n}$ corresponding to $f_1, \cdots, f_n \in \cal O_{Y_1}^{\times}$, their norms define $\Nm(f_1), \cdots, \Nm(f_n) \in \cal O_Y^{\times}$ and we form $\Nm(f_1)\otimes\cdots\otimes\Nm(f_n)$ using the product structure $\underline K_1^{\otimes n}\rightarrow\underline K_n$.

This map evidently factors through $\mathbb Z_{\tr}(\wedge^n\mathbb G_m)$. Since $\underline K_n$ is $\mathbb A^1$-invariant on $\Sm_{/F}$ \cite[Corollary 2.5]{MR533201}, we obtain \eqref{eq-motivic-complex-to-K-theory} by adjunction.
\end{void}

\begin{void}
For a smooth affine group scheme $G\rightarrow S$, we write:
\begin{equation}
\label{eq-cochains-on-BG-limit-presentation}
\Gamma(BG, \mathbb Z(n)) := \lim_{[k]\in \Delta^{\mathrm{op}}} \Gamma(G^{\times k}, \mathbb Z(n)),
\end{equation}
as a complex of abelian groups. Here, the sections $\Gamma(G^{\times k}, \mathbb Z(n))$ are computed by treating $\mathbb Z(n)$ as a complex of \emph{\'etale} sheaves.

The corresponding Zariski version will be written as $\Gamma(BG_{\Zar}, \mathbb Z(n))$. There is also the pointed version $\Gamma_e(BG, \mathbb Z(n))$, i.e.~the fiber along $e^* : \Gamma(BG, \mathbb Z(n)) \rightarrow \Gamma(S, \mathbb Z(n))$, and likewise for $\Gamma_e(BG_{\Zar}, \mathbb Z(n))$.
\end{void}

\begin{void}
Comparison of sites and the functor \eqref{eq-motivic-complex-to-K-theory} induce two morphisms of complexes of abelian groups:
\begin{align}
	\label{eq-motivic-parameter-etale-sheafification}
	\Gamma_e(BG_{\Zar}, \mathbb Z(2)[4]) &\rightarrow \Gamma_e(BG, \mathbb Z(2)[4]),\\
	\label{eq-motivic-parameter-to-K2}
	\Gamma_e(BG_{\Zar}, \mathbb Z(2)[4]) &\rightarrow \Gamma_e(BG_{\Zar}, \underline K_2[2]).
\end{align}

The following result shows that when $G\rightarrow S$ is reductive, central extensions of $G$ by $\underline K_2$ are equivalent to weight-$2$ motivic $4$-cocycles on $BG$, computed in either the Zariski or the \'etale topology.

Our statement is more general than its counterparts in \cite[\S6]{MR1460391} and \cite[\S6.3-6.4]{MR4117995}, although the proof is essentially the same.
\end{void}

\begin{thm}
\label{thm-motivic-versus-K2}
Let $F$ be a perfect field. Suppose $S\in\Sm_{/F}$ and $G\rightarrow S$ is a reductive group scheme. Then both morphisms \eqref{eq-motivic-parameter-etale-sheafification}, \eqref{eq-motivic-parameter-to-K2} induce equivalences in degrees $\le 0$.
\end{thm}

\begin{void}
\label{void-motivic-versus-K2-strategy}
Our strategy is to first prove that \eqref{eq-motivic-parameter-etale-sheafification} induces an equivalence in degrees $\le 0$, which will imply that the association $S_1\mapsto \tau^{\le 0}\Gamma_e(B(G_{S_1})_{\Zar}, \mathbb Z(2)[4])$, where $G_{S_1} := G\times_S S_1$, is an \'etale sheaf on smooth $S$-schemes.

Since $S_1\mapsto \tau^{\le 0}\Gamma_e(B(G_{S_1})_{\Zar}, \underline K_2[2])$ also satisfies \'etale descent \cite[\S2]{MR1896177}, the assertion for \eqref{eq-motivic-parameter-to-K2} reduces to the case of a \emph{split} reductive group $G$, for which we will prove an equivalence:
$$
\Gamma_e(BG_{\Zar}, \mathbb Z(2)[4]) \cong \Gamma_e(BG_{\Zar}, \underline K_2[2]),
$$
using the rationality of $G$ provided by the Bruhat decomposition.
\end{void}

\begin{proof}[Proof of Theorem \ref{thm-motivic-versus-K2}]
To prove that \eqref{eq-motivic-parameter-etale-sheafification} induces an equivalence in degrees $\le 0$, it suffices to prove that for any smooth affine group scheme $G$, the analogous map:
\begin{equation}
\label{eq-motivic-cohomology-comparison-group}
\Gamma_e(G_{\Zar}, \mathbb Z(2)) \rightarrow \Gamma_e(G, \mathbb Z(2))
\end{equation}
induces isomorphisms on $H^i$ for $i\le 3$.

Indeed, the presentation \eqref{eq-cochains-on-BG-limit-presentation} implies:
$$
\Gamma_e(BG, \mathbb Z(2)) \cong \lim_{[k]\in\Delta^{\mathrm{op}}}\Gamma_e(G^{\times k}, \mathbb Z(2)),
$$
so we obtain a spectral sequence $E_1^{p,q} = H^p_e(G^{\times q}, \mathbb Z(2)) \Rightarrow H^{p+q}_e(BG, \mathbb Z(2))$ and analogously for the Zariski version $E_{1, \Zar}^{p,q} = H^p_e(G^{\times q}_{\Zar}, \mathbb Z(2))$. Since $E_{1, \Zar}^{4, 0} \cong E_1^{4, 0} \cong 0$, the assertion on \eqref{eq-motivic-cohomology-comparison-group} implies that $E^{p, q}_{1,\Zar} \cong E_1^{p, q}$ for $p\le 3$, which is sufficient to guarantee $H_e^n(BG_{\Zar}, \mathbb Z(2)) \cong H_e^n(BG, \mathbb Z(2))$ for $n\le 4$.

The assertion on \eqref{eq-motivic-cohomology-comparison-group} in turn follows from Kahn's calculation of weight-2 motivic cohomology \cite[Theorem 1.1 \& 1.6]{MR1423901}. Namely, for a connected smooth scheme $X$ over $F$, there hold:
\begin{equation}
\label{eq-kahn-calculation-up-to-three}
H^i(X_{\Zar}, \mathbb Z(2)) \cong H^i(X, \mathbb Z(2)) \cong
\begin{cases}
	0 & i = 0 \\
	K_3(k(X))_{\mathrm{ind}} & i = 1\\
	H^0(X_{\Zar}, \underline K_2) & i = 2\\
	H^1(X_{\Zar}, \underline K_2) & i = 3
\end{cases}
\end{equation}
Here, $K_3(k(X))_{\mathrm{ind}}$ denotes the ``indecomposable part'' of $K_3(k(X))$, i.e.~the quotient of the Quillen $\tn K_3$ by the Milnor $\tn K_3$ of the field $k(X)$. Note that Theorem 1.6 of \emph{op.cit.}~furthermore proves:
\begin{equation}
\label{eq-kahn-calculation-four-and-above}
H^i(X_{\Zar}, \mathbb Z(2)) \cong 0,\quad i\ge 4.
\end{equation}

We now turn to the proof that \eqref{eq-motivic-parameter-to-K2} induces an equivalence in degrees $\le 0$. By what we said in \S\ref{void-motivic-versus-K2-strategy}, it suffices to prove that for a \emph{split} reductive group scheme $G\rightarrow S$, the map $\mathbb Z(2)[2] \rightarrow\underline K_2$ of \eqref{eq-motivic-complex-to-K-theory} induces an equivalence:
$$
\Gamma_e(G_{\Zar}, \mathbb Z(2)[2]) \cong \Gamma_e(G_{\Zar}, \underline K_2).
$$

The calculations \eqref{eq-kahn-calculation-up-to-three}, \eqref{eq-kahn-calculation-four-and-above} imply that the restriction of $\mathbb Z(2)[2] \rightarrow \underline K_2$ to the small Zariski site of $X$ has fiber being the (shifted) constant sheaf $K_3(k(X))_{\mathrm{ind}}[1]$. We thus obtain a morphism of triangles:
\begin{equation}
\label{eq-indecomposable-K3-diagram}
\begin{tikzcd}[column sep = 1.5em]
	K_3(k(G))_{\mathrm{ind}}[1] \ar[r]\ar[d, "e^*"] & \Gamma(G_{\Zar}, \mathbb Z(2)[2]) \ar[r]\ar[d, "e^*"] & \Gamma(G_{\Zar}, \underline K_2) \ar[d, "e^*"] \\
	K_3(k(S))_{\mathrm{ind}}[1] \ar[r] & \Gamma(S_{\Zar}, \mathbb Z(2)[2]) \ar[r] & \Gamma(S_{\Zar}, \underline K_2)
\end{tikzcd}
\end{equation}

It suffices to prove that the leftmost vertical arrow is an isomorphism. Since $G\rightarrow S$ is split reductive, it is birational to $\mathbb A_S^n$ for some $n\ge 0$. Thus it suffices to prove the equivalence:
$$
K_3(k(\mathbb A_S^n)_{\mathrm{ind}}) \cong K_3(k(S))_{\mathrm{ind}}.
$$

For this statement, we may return to \eqref{eq-indecomposable-K3-diagram} with $G = \mathbb A_S^n$, for which the middle and rightmost vertical arrows are equivalences, by the $\mathbb A^1$-invariance of motivic cohomology \cite[Theorem 13.8, Proposition 13.9]{MR2242284} and $\underline K_2$-cohomology \cite[Corollary 2.5]{MR533201}.
\end{proof}

\begin{rem}
Bryslinki--Deligne \cite[Theorem 7.2]{MR1896177} gives a complete description of the groupoid $\Gamma_e(BG_{\Zar}, B^2\underline K_2)$. Theorem \ref{thm-motivic-versus-K2} implies that the same description is valid for the groupoids associated to $\tau^{\le 0}\Gamma_e(BG, \mathbb Z(2)[4])$ and $\tau^{\le 0}\Gamma_e(BG_{\Zar}, \mathbb Z(2)[4])$. In particular, they are $1$-truncated.

If $S = \Spec(F)$, we may choose a maximal torus $T\subset G$ with cocharacter lattice $\Lambda_T$. Then $\Gamma_e(BG_{\Zar}, B^2\underline K_2)$ maps to the abelian group of Weyl and Galois-invariant quadratic forms on $\Lambda_T$, with fiber being $\Gamma(F, \underline{\Hom}(\pi_1G, \mathbb G_m[1]))$, where $\pi_1G$ denotes the algebraic fundamental group of $G$. For semisimple $G$, the corresponding description of $\tau^{\le 0}\Gamma_e(BG, \mathbb Z(2)[4])$ has also been obtained by Merkurjev \cite[Theorem 5.3]{MR3506385}.
\end{rem}

\begin{void}
Let us use Theorem \ref{thm-motivic-versus-K2} to define the functor \eqref{eq-etale-realization-integral-parameter}. Denote by $F_1$ the colimit perfection of $F$, i.e.
$$
F_1 := \colim_{x\mapsto x^p}(F).
$$
Then $\Spec(F_1) \rightarrow \Spec(F)$ is a universal homeomorphism \cite[0CNF]{stacks-project}. Write $S_1$ (resp.~$G_1$) for the base change of $S$ (resp.~$G$) along this map. Topological invariance of the \'etale site \cite[04DZ]{stacks-project} then implies that the pullback functor defines an equivalence:
$$
\Gamma_e(BG, B^4\mu_N^{\otimes 2}) \cong \Gamma_e(BG_1, B^4\mu_N^{\otimes 2}).
$$

Consequently, in order to define \eqref{eq-etale-realization-integral-parameter} we may assume that $F$ is perfect. The desired functor is then given by composing the equivalences of Theorem \ref{thm-motivic-versus-K2} with the identification $\mathbb Z(2)/N \cong \mu_N^{\otimes 2}$ \cite[Theorem 10.3]{MR2242284}:
\begin{align*}
	\Gamma_e(BG, B^2\underline K_2) &\cong \tau^{\le 0}\Gamma_e(BG_{\Zar}, \mathbb Z(2)[4]) \\
	&\cong \tau^{\le 0}\Gamma_e(BG, \mathbb Z(2)[4]) \rightarrow \tau^{\le 0}\Gamma_e(BG, \mathbb Z(2)/N[4]) \cong \Gamma_e(BG, B^4\mu_N^{\otimes 2}).
\end{align*}
\end{void}

\begin{void}
Suppose that $F$ is a local field containing a primitive $N$th root of unity. From any object of $\Gamma_e(BG, B^2\underline K_2)$, one may construct a central extension of topological groups:
\begin{equation}
\label{eq-covering-groups-from-BD}
1 \rightarrow \mu_N(F) \rightarrow \tilde G \rightarrow G(F) \rightarrow 1.
\end{equation}

Namely, viewing this object as a central extension of $G$ by $\underline K_2$, we may evaluate at $F$ and take the pushout along the Hilbert symbol $K_2(F) \rightarrow \mu_N(F)$. The topology on $\tilde G$ is constructed out of the corresponding central extension of $G(F)$ by $H^2(F, \mu_N^{\otimes 2})$ defined by the Galois symbol, see \cite[\S9]{MR1896177}.

Similarly, when $F$ is a global field containing a primitive $N$th root of unity, we obtain a central extension of $G(\mathbb A_F)$ by $\mu_N(F)$ as topological groups, equipped with a canonical splitting over $G(F)$. It is defined as the colimit, over $V\subset U\subset\Spec(\cal O_F)$, of extensions defined for an integral model $G_U$ of $G$ over $U$:
$$
1 \rightarrow \mu_N(F) \rightarrow \tilde G_{U, V} \rightarrow \prod_{v\notin V} G(F_v)\times\prod_{v\in V}G_U(\cal O_v) \rightarrow 1,
$$
which are in turn induced from the local extensions by the sum of the Hilbert symbol maps $\bigoplus_{v\notin V} K_2(F_v) \rightarrow \mu_N(F_v) \cong \mu_N(F)$, see \emph{op.cit.}~for details.

We shall now compare this construction with the ones in \S\ref{sect-metaplectic-cover-local}-\ref{sect-metaplectic-cover-global}.
\end{void}

\begin{prop}
Suppose that $F$ is a local field containing a primitive $N$th root of unity. Let $G$ be a reductive group scheme over $\Spec(F)$. The following diagram is canonically commutative:
\begin{equation}
\label{eq-construction-of-metaplectic-covers-compatible}
\begin{tikzcd}[row sep = 1em, column sep = 0.5em]
	\Gamma_e(BG_{\Zar}, B^2\underline K_2)\ar[dd, swap, "\eqref{eq-etale-realization-integral-parameter}"]\ar[dr] & \\
	&
	\begin{Bmatrix}
	\tn{topological covers of} \\
	\tn{$G(F)$ by $\mu_N(F)$}
	\end{Bmatrix}. \\
	\Gamma_e(BG, B^4\mu_N^{\otimes 2})^0\ar[ur] &
\end{tikzcd}
\end{equation}
where $\Gamma_e(BG, B^4\mu_N^{\otimes 2})^0 := \Gamma_e(BG, B^4\mu_N^{\otimes 2})$ for $F$ nonarchimedean or complex, and is the full subgroupoid for which the induced map $G(F) \rightarrow H^3(F, \mu_N^{\otimes 2})$ \eqref{eq-classical-metaplectic-construction-real-obstruction} vanishes for $F$ real.

The analogous statement holds for global fields $F$ containing a primitive $N$th root of unity, where one replaces the target of \eqref{eq-construction-of-metaplectic-covers-compatible} by topological covers of $G(\mathbb A_F)$ by $\mu_N(F)$, equipped with a splitting over $G(F)$.
\end{prop}
\begin{proof}
The global statement follows from the local one. For the local statement, part of the assertion is that for $F$ real, the image of \eqref{eq-etale-realization-integral-parameter} consists only of pointed morphisms $BG \rightarrow B^4\mu_N^{\otimes 2}$ whose induced map $G(F) \rightarrow H^3(F, \mu_N^{\otimes 2})$ vanishes. By construction, this map factors through $G(F) \rightarrow H^3(F_{\Zar}, \mathbb Z(2))$, which vanishes because $\mathbb Z(2)$ is concentrated in cohomological degrees $\le 2$. (We write $F_{\Zar}$ as a shorthand for $\Spec(F)_{\Zar}$.)

The commutativity of \eqref{eq-construction-of-metaplectic-covers-compatible} unwinds into the commutativity of the square below:
$$
\begin{tikzcd}
	H^2(F_{\Zar}, \mathbb Z(2)) \ar[r, "\eqref{eq-motivic-complex-to-K-theory}"]\ar[d] & K_2(F)\ar[d, "\tn{Hilbert}"]\ar[dl, swap, "\tn{Galois}"] \\
	H^2(F, \mu_N^{\otimes 2}) \ar[r, "\tn{Tate}"] & \mu_N(F)
\end{tikzcd}
$$
This square is divided into two commutive triangles: the top arrow reduces to the map $\theta$ of \cite[\S5]{MR2242284} which defines the ``inverse'' $H^2(F, \mu_N^{\otimes 2}) \rightarrow K_2(F)/N$ of the Galois symbol. On the other hand, the Galois symbol, followed by the Tate duality map, coincides with the $N$th Hilbert symbol.
\end{proof}

\medskip

\section{$\{a, a\} \cong \{a, -1\}$}
\label{sec-cup-product}

From this section until \S\ref{sec-classification}, our goal is to describe the groupoid $\Gamma_e(BG, B^4A(1))$ in explicit terms. This will rely on the description for a torus $T\rightarrow S$, which in turn requires us to study \'etale cochains on $\mathbb G_m$.

In this section, we construct a canonical isomorphism (Theorem \ref{thm-self-cup-product}) between two \'etale $2$-cocycles on $\mathbb G_m$ rigidified along $e : S\rightarrow \mathbb G_m$: the first one is the self-cup product $\Psi\cup\Psi$ of the Kummer $1$-cocycle $\Psi : \mathbb G_m \rightarrow B(\mu_N)$, and the second one is the Yoneda product of $\Psi$ with $\Psi(-1)$, a distinguished $1$-cocycle on $S$.

The construction we give is surprisingly subtle, and the author would like to know if there is a more direct way of obtaining this isomorphism.

\begin{void}
Throughout this section, we fix an integer $N\ge 1$ and a $\mathbb Z[\frac{1}{N}]$-scheme $S$.
\end{void}

\subsection{$2$-cocycles on $\mathbb G_m$}

\begin{void}
We view $\mathbb G_m$ as a constant group scheme over $S$ with structure map $p : \mathbb G_m \rightarrow S$. Let $\underline{\Gamma}{}_e(\mathbb G_m, B^2\mu_N^{\otimes 2})$ denote the \'etale sheaf of groupoids over $S$ whose sections over $S_1 \rightarrow S$ consist of sections of $B^2(\mu_N^{\otimes 2})$ over $\mathbb G_{m, S_1} := \mathbb G_m\times_S S_1$ rigidified along $e : S_1 \rightarrow \mathbb G_{m, S_1}$.

Let $\Psi : \mathbb G_m \rightarrow B(\mu_N)$ denote the Kummer torsor. The Yoneda product with $\Psi$ defines a $\mathbb Z/N$-linear morphism:
\begin{align}
\notag
\Psi^* : B(\mu_N) &\rightarrow \underline{\Maps}{}_{\mathbb Z}(\mathbb G_m, B^2\mu_N^{\otimes 2}) \\
\label{eq-yoneda-product-with-kummer}
& \rightarrow \underline{\Gamma}{}_e(\mathbb G_m, B^2\mu_N^{\otimes 2})
\end{align}
where the second functor is the forgetful one.
\end{void}

\begin{rem}
The sheaf of $\mathbb Z$-linear maps $\mathbb G_m\rightarrow B^2\mu_N^{\otimes 2}$ is identified with the connective truncation of the internal Hom of complexes $\underline{\Hom}{}_{\mathbb Z}(\mathbb G_m, \mu_N^{\otimes 2}[2])$, although we shall see in the proof of Lemma \ref{lem-cochains-on-Gm-equivalences} that the latter is already connective.
\end{rem}

\begin{lem}
\label{lem-cochains-on-Gm-equivalences}
Both functors in \eqref{eq-yoneda-product-with-kummer} are equivalences.
\end{lem}
\begin{proof}
To see that the first functor is an equivalence, it suffices to prove the analogous statement for $\Psi^* : B(\mathbb Z/N) \rightarrow \underline{\Maps}{}_{\mathbb Z}(\mathbb G_m, B^2\mu_N)$. This follows at once after replacing $\mu_N$ in the target by the complex $\mathbb G_m\xrightarrow{N} \mathbb G_m$.

The sheaf of $\mathbb Z/N$-module spectra $\underline{\Gamma}{}_e(\mathbb G_m, B^2\mu_N^{\otimes 2})$ corresponds to the following complex of sheaves of $\mathbb Z/N$-modules on $S$, in the sense of \S\ref{sect-complexes}:
$$
R\tilde p_*(\mu_N^{\otimes 2}[2]) := \Fib(Rp_*(\mu_N^{\otimes 2}[2]) \rightarrow \mu_N^{\otimes 2}[2]).
$$
It suffices to prove that the map $\mu_N[1] \rightarrow R\tilde p_*(\mu_N^{\otimes 2}[2])$ induced from $\Psi$ is an isomorphism of complexes.

Since the \'etale cohomology of $\mathbb G_m\rightarrow S$ commutes with arbitrary base change (\cite[Rappel 1.5.1]{MR1441006}), we may replace $S$ by the spectrum of a separably closed field $\bar s$, where:
$$
H^i(\mathbb G_{m,\bar s}, \mu_N^{\otimes 2}) \cong
\begin{cases}
	\mu_N & i = 1 \\
	0 & i\ge 2
\end{cases}
$$
and the identification of $H^1$ is indeed induced from $\Psi$.
\end{proof}

\begin{void}
\label{void-two-torsion-element}
Let us consider the self-cup product of $\Psi$:
\begin{equation}
\label{eq-self-cup-product-kummer}
\Psi \cup \Psi : \mathbb G_m \rightarrow B^2(\mu_N^{\otimes 2}),
\end{equation}
which may be viewed as a section of $\underline{\Gamma}{}_e(\mathbb G_m, B^2\mu_N^{\otimes 2})$. Anti-symmetry of the cup product equips $\Psi\cup\Psi$ with a $2$-torsion structure, i.e.~a trivialization $2\cdot (\Psi\cup\Psi)\cong *$.

Under the equivalences of Lemma \ref{lem-cochains-on-Gm-equivalences}, $\Psi\cup\Psi$ corresponds to a section of $B(\mu_N)$ over $S$. Our goal is to determine this section.

To state the answer, we consider the section $\Psi(-1)$ of $B(\mu_N)$. It has a natural $2$-torsion structure: linearity of $\Psi$ yields $2\cdot\Psi(-1) \cong \Psi(1)$ and $1$ admits an $N$th root $1^N = 1$. Of course, if $N$ is odd, then any $2$-torsion section of $B(\mu_N)$ is canonically trivial.
\end{void}

\begin{thm}
\label{thm-self-cup-product}
There is a canonical isomorphism:
\begin{equation}
\label{eq-self-cup-product-isomorphism}
\Psi \cup \Psi \cong \Psi^*(\Psi(-1))
\end{equation}
in $\underline{\Gamma}{}_e(\mathbb G_m, B^2\mu_N^{\otimes 2})$ compatible with the $2$-torsion structures.
\end{thm}

\begin{rem}
The isomorphism \eqref{eq-self-cup-product-isomorphism} is an analogue of an equality in the second algebraic $\tn K$-group. Indeed, suppose that $S = \Spec(R)$. Given a section of $\mathbb G_m$ over $S$, represented by $a\in R^{\times}$, there holds $\{a, a\} = \{a, -1\}$ in $\tn K_2(R)$.

The situation in \'etale cohomology is somewhat different. Because $\underline{\Maps}{}_{\mathbb Z}(\mathbb G_m, B^2\mu_N^{\otimes 2})$ has vanishing $\pi_0$, any two sections of it are \emph{non-canonically} isomorphic locally on $S$. However, to ensure good descent properties later on, we need to exhibit a canonical isomorphism \eqref{eq-self-cup-product-isomorphism}.

The subtlety in constructing \eqref{eq-self-cup-product-isomorphism} has to do with the fact that the number $(-1)$ comes \emph{a priori} from the additive structure. We will instead make $\Psi(-1)$ appear as $\Psi$ applied to the product of all $N$th roots of unity, which exist \'etale locally on $S$.
\end{rem}

\subsection{Self-cup product}
\label{sect-self-cup-product}

\begin{void}
Let us first explain the construction of \eqref{eq-self-cup-product-kummer} in more details. In particular, it will help to directly realize $\Psi\cup\Psi$ as a $\mathbb Z$-linear morphism $\mathbb G_m\rightarrow B^2\mu_N^{\otimes 2}$.

Let $A$ be a sheaf of abelian groups on a site $\cal C$. Write $H^{(1)}(A)$ for the extension of sheaves of abelian groups:
\begin{equation}
\label{eq-heisenberg-extension-abelian-group}
0 \rightarrow \Sym^2(A) \rightarrow H^{(1)}(A) \rightarrow A\rightarrow 0,
\end{equation}
defined by the cocycle $a_1, a_2\mapsto a_1a_2$. Namely, there is a given set-theoretic splitting $A\rightarrow H^{(1)}(A)$, under which the product of the images of $a_1,a_2\in A$ in $H^{(1)}(A)$ is the $(a_1a_2)$-multiple of the image of $a_1+a_2 \in A$.
\end{void}

\begin{void}
The relevant case for us is $A = \mu_N$ on the \'etale site of $S$. There holds $\mu_N^{\otimes 2}\cong \Sym^2(\mu_N)$, so \eqref{eq-heisenberg-extension-abelian-group} defines a $\mathbb Z$-linear coboundary map:
\begin{equation}
\label{eq-heisenberg-coboundary-muN}
H^{(1)}(\mu_N) : \mu_N \rightarrow B(\mu_N^{\otimes 2}).
\end{equation}

We write $\Psi\cup\Psi := \Psi^*H^{(1)}(\mu_N)$ for the composition of $\Psi$ with \eqref{eq-heisenberg-coboundary-muN}: it is naturally a section of $\underline{\Maps}{}_{\mathbb Z}(\mathbb G_m, B^2\mu_N^{\otimes 2})$.
\end{void}

\begin{rem}
The coboundary of \eqref{eq-heisenberg-extension-abelian-group} indeed encodes cup product on cohomology, i.e.~for any object $c\in\cal C$, the coboundary $B(A) \rightarrow B^2(\Sym^2(A))$ induces the composition of:
$$
H^1(c, A) \rightarrow H^2(c, A^{\otimes 2}),\quad x\mapsto x\cup x
$$
with the natural map $H^2(c, A^{\otimes 2}) \rightarrow H^2(c, \Sym^2(A))$, see \cite[Theorem 2.5]{MR2915483}.
\end{rem}

\begin{void}
\label{void-heisenberg-quadratic-description}
Maps out of $H^{(1)}(A)$ have a pleasant description as ``quadratic functions without constant terms''.

More precisely, let $A_1$ be a sheaf of abelian groups on $\cal C$. Write $\check H^{(1)}(A, A_1)$ for the sheaf of abelian groups whose sections are maps $Q : A \rightarrow A_1$ satisfying:
\begin{enumerate}
	\item $Q(0) = 0$;
	\item $a_1,a_2\mapsto Q(a_1+a_2) - Q(a_1) - Q(a_2)$ defines a symmetric bilinear form on $A$: this is the symmetric form associated to $Q$.
\end{enumerate}

Then there is a canonical isomorphism defined by restriction along the set-theoretic section $A\subset H^{(1)}(A)$:
$$
\underline{\Maps}{}_{\mathbb Z}(H^{(1)}(A), A_1) \cong \check H^{(1)}(A, A_1).
$$
Under this bijection, restriction along $\Sym^2(A)\subset H^{(1)}(A)$ corresponds to taking the \emph{negative} of the symmetric form associated to a quadratic function $Q : A\rightarrow A_1$.
\end{void}

\begin{void}
\label{void-heisenberg-2-torsion}
The $2$-torsion structure of self-cup product is encoded by a $\mathbb Z$-linear trivialization of the map:
$$
2\cdot H^{(1)}(A) : A \rightarrow B(\Sym^2(A)),
$$
which is in turn defined by the quadratic function $A\rightarrow \Sym^2(A)$, $a\mapsto -a^2$.

In particular, we obtain a trivialization of $2\cdot\Psi\cup\Psi$ as a section of $\underline{\Maps}{}_{\mathbb Z}(\mathbb G_m, B^2\mu_N^{\otimes 2})$.
\end{void}

\subsection{The case $A = \mu_N$}

\begin{void}
Let us now specialize \eqref{eq-heisenberg-extension-abelian-group} to the sheaf $A := \mu_N$ on the \'etale site of $S$. For any homomorphism $\lambda : \mu_N^{\otimes 2} \rightarrow \mathbb G_m$ (necessarily valued in $\mu_N\subset\mathbb G_m$), we may induce \eqref{eq-heisenberg-extension-abelian-group} along $\lambda$ to obtain a section $\lambda_*H^{(1)}(\mu_N) \in \underline{\Maps}{}_{\mathbb Z}(\mu_N, B\mu_N)$.

Resolving the target $B(\mu_N)$ by the Kummer sequence yields a split triangle of sheaves of $\mathbb Z$-module spectra:
\begin{equation}
\label{eq-self-extension-of-muN-split-triangle}
\underline{\Maps}{}_{\mathbb Z}(\mu_N, \mathbb G_m) \xrightarrow{\Psi_*} \underline{\Maps}{}_{\mathbb Z}(\mu_N, B\mu_N) \rightarrow \underline{\Maps}{}_{\mathbb Z}(\mu_N, B\mathbb G_m).
\end{equation}
The first and last terms are isomorphic to $\mathbb Z/N$, respectively $B(\mathbb Z/N)$, giving:
\begin{equation}
\label{eq-self-extension-of-muN-splitting}
	\underline{\Maps}{}_{\mathbb Z}(\mu_N, B\mu_N) \cong \mathbb Z/N \oplus B(\mathbb Z/N).
\end{equation}
For a nondegenerate pairing $\lambda$, we shall determine the image of $\lambda_*H^{(1)}(\mu_N)$ under \eqref{eq-self-extension-of-muN-splitting}.
\end{void}

\begin{rem}
The inclusion $B(\mathbb Z/N) \rightarrow \underline{\Maps}{}_{\mathbb Z}(\mu_N, B\mu_N)$ in \eqref{eq-self-extension-of-muN-splitting} may be described as follows: any section of $B(\mathbb Z/N)$ defines a $\mathbb Z/N$-linear map $\mathbb Z/N\rightarrow B(\mathbb Z/N)$ which we may tensor with $\mu_N$.

The projection $\underline{\Maps}{}_{\mathbb Z}(\mu_N, B\mu_N) \rightarrow \mathbb Z/N$ can be identified with taking the isomorphism class and using $\underline{\Ext}^1_{\mathbb Z}(\mu_N, \mu_N) \cong \mathbb Z/N$.
\end{rem}

\begin{rem}
\label{rem-splitting-square-linearity}
The sheaf of $\mathbb Z$-module spectra $\underline{\Maps}{}_{\mathbb Z}(\mu_N, B\mu_N)$ has two $\mathbb Z/N$-module structures, given by multiplication on the source or on the target. Their induced $\mathbb Z/N^2$-module structures are naturally identified. In particular, we may view \eqref{eq-self-extension-of-muN-splitting} unambiguously as an isomorphism of complexes of $\mathbb Z/N^2$-modules.
\end{rem}

\begin{void}
\label{void-roots-of-unity-two-torsion}
Any primitive $N$th root of unity $\zeta \in \cal O_S^{\times}$ determines a nondegenerate pairing $\lambda : \mu_N^{\otimes 2} \rightarrow \mathbb G_m$, satisfying the equality $\lambda(\zeta\otimes\zeta) = \zeta$.

Note that such $\zeta$ satisfies:
$$
\zeta^{\binom{N}{2}} =
\begin{cases}
	1 & N \text{ odd}\\
	-1 & N \text{ even}
\end{cases}.
$$
In particular, $\Psi(\zeta^{\binom{N}{2}}) \cong \Psi(-1)$ as sections of $B(\mu_N)$.

On the other hand, we equip $\Psi(\zeta^{\binom{N}{2}})$ with a ``strange'' $2$-torsion structure: $2\cdot \Psi(\zeta^{\binom{N}{2}}) \cong \Psi(1)$ and we use $\zeta^{-1}$ to trivialize $\Psi(1)$. This is generally distinct from the $2$-torsion structure on $\Psi(-1)$ defined in \S\ref{void-two-torsion-element}.

The reason for introducing this $2$-torsion structure is that it corresponds naturally to the $2$-torsion structure on $H^{(1)}(\mu_N)$, which will play a role in the proof of Proposition \ref{prop-self-cup-product-2-torsion}.
\end{void}

\begin{prop}
\label{prop-heisenberg-muN-calculation}
Let $\zeta \in \cal O_S^{\times}$ be a primitive $N$th root of unity with corresponding nondegenerate pairing $\lambda$. Under the splitting \eqref{eq-self-extension-of-muN-splitting}, $\lambda_*H^{(1)}(\mu_N)$ corresponds to:
\begin{enumerate}
	\item $\binom{N}{2} \in \mathbb Z/N$;
	\item $\lambda_*\Psi(\zeta^{\binom{N}{2}}) \in B(\mathbb Z/N)$, where $\lambda$ is viewed as a map $ \mu_N\cong\mathbb Z/N$.
\end{enumerate}
\end{prop}
\begin{proof}
We shall use the exact sequences obtained by mapping \eqref{eq-heisenberg-extension-abelian-group} (for $A = \mu_N$) into $\mu_N$ and $\mathbb G_m$, as summarized in the diagram below:
\begin{equation}
\label{eq-heisenberg-muN-calculation}
\begin{tikzcd}[column sep = 1em, row sep = 2em]
	 & & \check H^{(1)}(\mu_N, \mu_N) \ar[r]\ar[d] & \underline{\Maps}{}_{\mathbb Z}(\mu_N^{\otimes 2}, \mu_N) \ar[r]\ar[d, "\cong"] & \underline{\Ext}^1_{\mathbb Z}(\mu_N, \mu_N) \\
	1 \ar[r] & \underline{\Maps}{}_{\mathbb Z}(\mu_N, \mathbb G_m) \ar[d, "N"]\ar[r] & \check H^{(1)}(\mu_N, \mathbb G_m) \ar[r]\ar[d, "N"] & \underline{\Maps}{}_{\mathbb Z}(\mu_N^{\otimes 2}, \mathbb G_m)\ar[r]\ar[d, "N"] & 1 \\
	1 \ar[r] & \underline{\Maps}{}_{\mathbb Z}(\mu_N, \mathbb G_m) \ar[r]\ar[d, "\ev_{\zeta}"] & \check H^{(1)}(\mu_N, \mathbb G_m) \ar[r]\ar[d, "\ev_{\zeta}"] & \underline{\Maps}{}_{\mathbb Z}(\mu_N^{\otimes 2}, \mathbb G_m) \ar[r]\ar[d, "{(\ev_{\zeta\otimes\zeta})^{\binom{N}{2}}}"] & 1 \\
	& \mu_N \ar[r] & \mathbb G_m \ar[r, "N"] & \mathbb G_m 
\end{tikzcd}
\end{equation}
Here, we have identified $\underline{\Maps}{}_{\mathbb Z}(H^{(1)}(\mu_N), \mathbb G_m)$ with $\check H^{(1)}(\mu_N, \mathbb G_m)$, the set of quadratic functions $\mu_N \rightarrow \mathbb G_m$, see \S\ref{void-heisenberg-quadratic-description}. Hence the map $\check H^{(1)}(\mu_N, \mathbb G_m) \rightarrow \underline{\Maps}{}_{\mathbb Z}(\mu_N^{\otimes 2}, \mathbb G_m)$ carries $Q$ to the bilinear pairing $b(a_1\otimes a_2) := Q(a_1)Q(a_2)Q(a_1a_2)^{-1}$.

The bottom vertical arrows in \eqref{eq-heisenberg-muN-calculation} are given by evaluating a linear (resp.~quadratic) map $\mu_n\rightarrow \mathbb G_m$ at $\zeta\in\mu_N$, see \S\ref{void-heisenberg-quadratic-description}. The lower right square commutes because of the binomial theorem:
\begin{equation}
\label{eq-binomial-theorem}
Q(a)^k = b(a\otimes a)^{\binom{k}{2}}Q(a^k).
\end{equation}

The extension class of $\lambda_*H^{(1)}(\mu_N)$ is the image of $\lambda\in \underline{\Maps}{}_{\mathbb Z}(\mu_N^{\otimes 2}, \mu_N)$ in $\underline{\Ext}^1_{\mathbb Z}(\mu_N, \mu_N)$. The isomorphism $\underline{\Ext}^1_{\mathbb Z}(\mu_N, \mu_N) \cong \mathbb Z/N$ is induced from the connecting map of the Snake Lemma applied to the two middle rows of \eqref{eq-heisenberg-muN-calculation}, so the extension class of $\lambda$ is computed by choosing any quadratic function $Q : \mu_N\rightarrow\mathbb G_m$ lifting $\lambda$ and taking $Q^N$: a linear form $\zeta\mapsto \lambda(\zeta\otimes\zeta)^{\binom{N}{2}} = \zeta^{\binom{N}{2}}$ according to \eqref{eq-binomial-theorem}. The expression (1) follows.

The section of $B(\mathbb Z/N)$ corresponding to $\lambda_*H^{(1)}(\mu_N)$ is equivalently described as the $\underline{\Maps}{}_{\mathbb Z}(\mu_N, \mathbb G_m)$-torsors of the quadratic lifts of $\lambda$, viewed as a $\mathbb G_m$-valued symmetric form. Under $\ev_{\zeta}$, it induces the $\mu_N$-torsor $\Psi(\zeta^{\binom{N}{2}})$ by the commutativity of the two bottom rows in \eqref{eq-heisenberg-muN-calculation}. Since the composition below is the identity:
$$
\underline{\Maps}{}_{\mathbb Z}(\mu_N, \mathbb G_m) \xrightarrow{\ev_{\zeta}} \mu_N \xrightarrow{\lambda} \underline{\Maps}{}_{\mathbb Z}(\mu_N, \mathbb G_m),
$$
expression (2) follows.
\end{proof}

\begin{void}
We are now ready to construct an isomorphism in $\underline{\Maps}{}_{\mathbb Z}(\mathbb G_m, B\mu_N)$:
\begin{equation}
\label{eq-self-cup-product-isomorphism-primitive}
T_{\zeta} : \lambda_*(\Psi \cup \Psi) \cong \Psi^*\lambda_*\Psi(-1),
\end{equation}
upon choosing a primitive $N$th root of unity $\zeta$ (with corresponding pairing $\lambda$).
\end{void}

\begin{proof}[Construction]
Recall the isomorphism $\Psi\cup\Psi\cong\Psi^*H^{(1)}(\mu_N)$, see \S\ref{sect-self-cup-product}. Hence $T_{\zeta}$ may be produced from the two isomorphisms below, taking place in $\underline{\Maps}{}_{\mathbb Z}(\mu_N, B\mu_N)$ respectively $\underline{\Maps}{}_{\mathbb Z}(\mathbb G_m, B^2\mu_N)$:
\begin{align}
\label{eq-self-cup-product-difference-primitive}
	\lambda_*H^{(1)}(\mu_N) - \lambda_*\Psi(-1) &\cong \Psi_*\binom{N}{2}, \\
\label{eq-null-homotopy-bockstein-term}
	\Psi^*\Psi_*\binom{N}{2} &\cong *.
\end{align}
(In \eqref{eq-self-cup-product-difference-primitive}, $\lambda_*\Psi(-1) \in B(\mathbb Z/N)$ is understood as its image under the inclusion in \eqref{eq-self-extension-of-muN-splitting}.)

The first isomorphism \eqref{eq-self-cup-product-difference-primitive} follows from Proposition \ref{prop-heisenberg-muN-calculation} and the isomorphism $\Psi(\zeta^{\binom{N}{2}}) \cong \Psi(-1)$, see \S\ref{void-roots-of-unity-two-torsion}.

To define \eqref{eq-null-homotopy-bockstein-term}, we trivialize the following composition of $\mathbb Z/N^2$-linear maps:
\begin{equation}
\label{eq-null-homotopy-bockstein-term-diagram}
\mathbb Z/N^2 \twoheadrightarrow \mathbb Z/N \xrightarrow{\Psi_*} \underline{\Maps}{}_{\mathbb Z}(\mu_N, B\mu_N) \rightarrow \underline{\Maps}{}_{\mathbb Z}(\mu_{N^2}, B\mu_N).
\end{equation}
Indeed, by adjunction it suffices to trivialize the image of $1\in\mathbb Z/N^2$. Under the composition \eqref{eq-null-homotopy-bockstein-term-diagram}, $1\in\mathbb Z/N^2$ maps to the extension $\mu_N\rightarrow \mu_{N^2} \rightarrow \mu_N$ induced along $\mu_{N^2}\twoheadrightarrow \mu_N$; it naturally splits. Since the element $\binom{N}{2} \in \mathbb Z/N$ lifts along the surjection $\mathbb Z/N^2 \twoheadrightarrow \mathbb Z/N$, the image of $\Psi_*\binom{N}{2}$ in $\underline{\Maps}{}_{\mathbb Z}(\mu_{N^2}, B\mu_N)$ is thus trivialized.

Finally, we observe that the map:
$$
\Psi^* : \underline{\Maps}{}_{\mathbb Z}(\mu_N, B\mu_N) \rightarrow \underline{\Maps}{}_{\mathbb Z}(\mathbb G_m, B\mu_N)
$$
factors through the last map in \eqref{eq-null-homotopy-bockstein-term-diagram}. This defines the isomorphism \eqref{eq-null-homotopy-bockstein-term}.
\end{proof}

\begin{rem}
We emphasize that $\lambda_*H^{(1)}(\mu_N)$ is \emph{not} isomorphic to $\lambda_*\Psi(-1)$, i.e.~the isomorphism $T_{\zeta}$ \eqref{eq-self-cup-product-isomorphism-primitive} only exists after applying $\Psi^*$ to both sides.
\end{rem}

\subsection{Properties of $T_{\zeta}$}

\begin{void}
After choosing a primitive $N$th root of unity $\zeta$ with corresponding nondegenerate pairing $\lambda : \mu_N^{\otimes 2} \rightarrow \mathbb G_m$, we have constructed an isomorphism $T_{\zeta}$ \eqref{eq-self-cup-product-isomorphism-primitive} of two sections of $\underline{\Maps}{}_{\mathbb Z}(\mathbb G_m, B^2\mu_N)$.

Inducing along the inverse of $\lambda$, i.e.~the mapping $\mu_N\cong\mu_N^{\otimes 2}$, $\zeta\mapsto\zeta\otimes\zeta$, we obtain an isomorphism:
\begin{equation}
\label{eq-self-cup-product-isomorphism-primitive-dependent}
\lambda^{-1}T_{\zeta} : \Psi\cup\Psi \cong \Psi^*(\Psi(-1)).
\end{equation}

This is our candidate for \eqref{eq-self-cup-product-isomorphism}, meaningful over the base scheme $S_1 := \mu_{N, S}$. In order to show that it descends to $S$ and satisfies the requirement of Theorem \ref{thm-self-cup-product}, we need to prove that it is compatible with the $2$-torsion structures and is independent of the choice of $\zeta$.

These properties are established below.
\end{void}

\begin{prop}
\label{prop-self-cup-product-2-torsion}
The isomorphism \eqref{eq-self-cup-product-isomorphism-primitive-dependent} relates the $2$-torsion structure on $\Psi\cup\Psi$ to the $2$-torsion structure on $\Psi(-1)$.
\end{prop}
\begin{proof}
It suffices to prove that \eqref{eq-self-cup-product-isomorphism-primitive} has this property. An isomorphism $f : a_1 \cong a_2$ of objects equipped with $2$-torson structures in a Picard groupoid $A$ defines an element in $\pi_1(A)$: the composition $* \cong 2\cdot a_1 \xrightarrow{2f} 2\cdot a_2 \cong *$. Let us call it the ``2-torsion error'' of $f$. It vanishes if and only if $f$ is compatible with the $2$-torsion structures.

Recall that \eqref{eq-self-cup-product-isomorphism-primitive} is defined by the composition of isomorphisms in the Picard groupoid $A:= \underline{\Maps}{}_{\mathbb Z}(\mathbb G_m, B^2\mu_N)$:
\begin{equation}
\label{eq-self-cup-product-isomorphism-primitive-expansion}
\Psi^*(\lambda_*H^{(1)}(\mu_N) - \lambda_*\Psi(-1)) \cong \Psi^*\Psi_*\binom{N}{2} \cong *.
\end{equation}
Under the isomorphism $\pi_1(A) \cong \mathbb Z/N$, we shall prove that the first isomorphism has $2$-torsion error $-1$ and the second isomorphism has $2$-torsion error $1$.

The first isomorphism comes from the isomorphism \eqref{eq-self-cup-product-difference-primitive} of sections of $\underline{\Maps}{}_{\mathbb Z}(\mu_N, B\mu_N)$, and it suffices to calculate the $2$-torsion error there. To do so, we first argue that the isomorphism:
\begin{equation}
\label{eq-isomorphism-two-torsion-calculation-first}
\lambda_*H^{(1)}(\mu_N) - \lambda_*\Psi(\zeta^{\binom{N}{2}}) \cong \Psi_*\binom{N}{2},
\end{equation}
where $\Psi(\zeta^{\binom{N}{2}})$ is equipped with the $2$-torsion structure defined using $\zeta^{-1}$ (see \S\ref{void-roots-of-unity-two-torsion}), is \emph{compatible} with the $2$-torsion structures. The desired claim will follow since the isomorphism $\Psi(-1) \cong \Psi(\zeta^{\binom{N}{2}})$ of sections of $B(\mu_N)$ has $2$-torsion error $\zeta$, so $-\lambda_*\Psi(-1) \cong -\lambda_*\Psi(\zeta^{\binom{N}{2}})$ has $2$-torsion error $-\lambda_*(\zeta) = -1$.

The isomorphism \eqref{eq-isomorphism-two-torsion-calculation-first} comes from the proof of Proposition \ref{prop-heisenberg-muN-calculation}. Its compatibility with $2$-torsion structures will follow once we prove that along the projection of \eqref{eq-self-extension-of-muN-splitting}:
$$
\underline{\Maps}{}_{\mathbb Z}(\mu_N, B\mu_N) \rightarrow B(\mathbb Z/N),
$$
the identification of the image of $\lambda_*H^{(1)}(\mu_N)$ with $\lambda_*\Psi(\zeta^{\binom{N}{2}})$ is compatible with the $2$-torsion structures. Recall that the image of $\lambda_*H^{(1)}(\mu_N)$ is the $\mathbb Z/N$-torsor of quadratic lifts of $\lambda$, viewed as a $\mathbb G_m$-valued symmetric form. Its $2$-torsion structure is given by the canonical quadratic lift of $2\lambda$:
\begin{equation}
\label{eq-quadratic-lift-twice-bilinear-form}
\mu_N \rightarrow \mathbb G_m,\quad a\mapsto \lambda(a\otimes a)^{-1},
\end{equation}
see \S\ref{void-heisenberg-2-torsion}. Under the bijection $Q\mapsto Q(\zeta)$ between quadratic lifts of $2\lambda$ and $N$th roots of $(\zeta^{\binom{N}{2}})^2 = 1$, the form \eqref{eq-quadratic-lift-twice-bilinear-form} passes to $\lambda(\zeta\otimes\zeta)^{-1} = \zeta^{-1}$. The desired conclusion follows.

The second isomorphism in \eqref{eq-self-cup-product-isomorphism-primitive-expansion} comes from the isomorphism \eqref{eq-null-homotopy-bockstein-term}, which already occurs in $\underline{\Maps}{}_{\mathbb Z}(\mu_{N^2}, B\mu_N)$, so we will calculate the $2$-torsion error there instead. More precisely, the image of $\Psi_*\binom{N}{2}$ in $\underline{\Maps}{}_{\mathbb Z}(\mu_{N^2}, B\mu_N)$ is the extension induced along:
$$
\begin{tikzcd}[column sep = 2em]
	& & & \mu_{N^2}\ar[d, twoheadrightarrow] \\
	& & & \mu_N \ar[d, "a\mapsto a^{\binom{N}{2}}"] \\
	1 \ar[r] & \mu_N \ar[r] & \mu_{N^2} \ar[r] & \mu_N \ar[r] & 1
\end{tikzcd}
$$
It is trivialized by the section $\mu_{N^2} \rightarrow \mu_{N^2}$, $a\mapsto a^{\binom{N}{2}}$, using the lift of $\binom{N}{2}$ along $\mathbb Z/N^2 \twoheadrightarrow \mathbb Z/N$. In particular, its square is the trivialization of $2\cdot\Psi_*\binom{N}{2} \cong \Psi_*(2\cdot\binom{N}{2})$ given by the section $\mu_{N^2} \rightarrow \mu_{N^2}$, $a\mapsto a^{2\cdot\binom{N}{2}}$.

On the other hand, the $2$-torsion structure on $\Psi_*\binom{N}{2}$ is defined by the vanishing of $2\cdot\binom{N}{2}$ in $\mathbb Z/N$ and the linearity of $\Psi_*$, so it corresponds to the section $\mu_{N^2} \rightarrow \mu_{N^2}$, $a\mapsto 1$. The $2$-torsion error of the trivialization of $\Psi_*\binom{N}{2}$ in $\underline{\Maps}{}_{\mathbb Z}(\mu_{N^2}, B\mu_N)$ is thus induced from the difference of these two sections:
$$
0 - 2\cdot\binom{N}{2} = N \in \mathbb Z/N^2,
$$
along the map:
$$
\Ker(\mathbb Z/N^2\twoheadrightarrow \mathbb Z/N) \rightarrow \pi_1\underline{\Maps}{}_{\mathbb Z}(\mu_{N^2}, B\mu_N) \cong \mathbb Z/N,
$$
which carries $N$ to $1$.
\end{proof}

\begin{prop}
The isomorphism \eqref{eq-self-cup-product-isomorphism-primitive-dependent} is independent of $\zeta$.
\end{prop}
\begin{proof}
For another primitive $N$th root of unity $\zeta_1$ with corresponding nongenerate pairing $\lambda_1$, we need to show an equality of the isomorphisms \eqref{eq-self-cup-product-isomorphism-primitive-dependent} defined by $\zeta$ and $\zeta_1$.

Let $k \in (\mathbb Z/N)^{\times}$ be the unique element so that $\zeta = \zeta_1^k$, or equvalently $\lambda_1 = \lambda^k$. We need to show that the following diagram in $\underline{\Maps}{}_{\mathbb Z}(\mathbb G_m, B^2\mu_N)$ is commutative:
\begin{equation}
\label{eq-self-cup-product-independence}
\begin{tikzcd}[column sep = 1.5em]
	k_*\lambda_*(\Psi\cup\Psi) \ar[r, "k_*T_{\zeta}"]\ar[d] & \Psi^*k_*\lambda_*\Psi(\zeta^{\binom{N}{2}}) \ar[d] \\
	(\lambda_1)_*(\Psi\cup\Psi) \ar[r, "T_{\zeta_1}"] & \Psi^*(\lambda_1)_*\Psi(\zeta^{\binom{N}{2}})
\end{tikzcd}
\end{equation}
where $k_*$ means inducing along the map $\mu_N \rightarrow \mu_N$, $a\mapsto a^k$.

We shall deduce \eqref{eq-self-cup-product-independence} from the $\mathbb Z/N^2$-linear structure of the splitting \eqref{eq-self-extension-of-muN-splitting}, see Remark \ref{rem-splitting-square-linearity}. Indeed, let $k_1 \in \mathbb Z/N^2$ be a lift of $k$. The operation $k_*$ amounts to multiplication by $k_1$. Thus \eqref{eq-self-cup-product-difference-primitive} yields a commutative diagram:
$$
\begin{tikzcd}[column sep = 1.5em]
	k_1\cdot\lambda_*H^{(1)}(\mu_N) - k_1\cdot\lambda_*\Psi(\zeta^{\binom{N}{2}}) \ar[d, "\cong"]\ar[r, "\cong"] & k_1\cdot\Psi_*\binom{N}{2} \ar[d, "\cong"] \\
	(\lambda_1)_*H^{(1)}(\mu_N) - (\lambda_1)_*\Psi(\zeta_1^{\binom{N}{2}}) \ar[r, "\cong"] & \Psi_*\binom{N}{2}
\end{tikzcd}
$$
To obtain the commutativity of \eqref{eq-self-cup-product-independence}, it remains to prove that the trivialization of $\Psi^*\Psi_*\binom{N}{2}$ is compatible with multiplication by $k_1$. This in turn follows from the $\mathbb Z/N^2$-linearity of the null-homotopy of \eqref{eq-null-homotopy-bockstein-term-diagram}.
\end{proof}

\medskip

\section{Tori}
\label{sec-torus}

The goal of this section is to classify \'etale metaplectic covers of a torus $T\rightarrow S$ in terms of its sheaf of cocharacters $\Lambda$. The main result is Theorem \ref{thm-classification-torus}, which states the answer in terms of (even) $\vartheta$-data with coefficients in $A(-1)$.

The notion of $\vartheta$-data is introduced by Beilinson--Drinfeld \cite[\S3.10]{MR2058353}. It also appeared in Brylinski--Deligne \cite[\S3]{MR1896177}, although this name was not used. The difference between their $\vartheta$-data and ours is that we allow torsion coefficients instead of $\mathbb Z$. In this situation, $\vartheta$-data naturally form a $2$-groupoid, and the usual definition needs to be reformulated in a way that is homotopy coherent. This will be carried out in \S\ref{sect-linear-algebra}-\ref{sect-theta-data}.

The main classification result is contained in \S\ref{sect-classification-tori}, and the two sections which follow \S\ref{sect-torus-covers-defined-by-cocycles}-\ref{sect-monoidal-classification-data} are elaborations on it.

Then in \S\ref{sect-commutative-covers}-\ref{sect-quadratic-structure}, we study \'etale metaplectic covers of $T$ which are ``commutative''. These induce, for example, commutative topological covers of $T(F)$ for a local field $F$. Such covers are used in the definition of the L-group of an \'etale metaplectic cover.

Although the classification Theorem \ref{thm-classification-torus} makes essential use of Theorem \ref{thm-self-cup-product}, the study of the commutative covers in \S\ref{sect-commutative-covers} (hence the definition of the L-group as well) does not rely on Theorem \ref{thm-self-cup-product}.

\subsection{Linear algebra}
\label{sect-linear-algebra}

\begin{void}
Let $\Lambda$ be a locally constant sheaf of finite free $\mathbb Z$-modules on a site $\cal C$. The permutation group $\Sigma_2$ acts on $\Lambda \otimes \Lambda$ by exchanging its factors.

Derived $\Sigma_2$-coinvariants $(\Lambda\otimes \Lambda)_{\Sigma_2}$ are computed by the complex of sheaves in cohomological degrees $\le 0$:
\begin{equation}
\label{eq-exchange-derived-coinvariants}
[\cdots \xrightarrow{\Ant} \Lambda\otimes\Lambda \xrightarrow{\Sym} \Lambda\otimes \Lambda \xrightarrow{\Ant} \Lambda\otimes\Lambda].
\end{equation}
Here, $\Sym$ (resp.~$\Ant$) denotes the (anti-)symmetrizer sending $x_1\otimes x_2$ to $x_1\otimes x_2 + x_2\otimes x_1$ (resp.~$x_1\otimes x_2 - x_2\otimes x_1$).
\end{void}

\begin{void}
Let us calculate the cohomology groups $H^i$ of \eqref{eq-exchange-derived-coinvariants}.

By definition, we have $H^0 \cong \Sym^2(\Lambda)$.

Note that the anti-symmetrizer $\Lambda\otimes\Lambda \rightarrow \Lambda\otimes\Lambda$ has image $\wedge^2(\Lambda)$. In particular, we have a short exact sequence:
\begin{equation}
\label{eq-linear-algebra-exact-sequence}
0 \rightarrow \wedge^2(\Lambda) \rightarrow \Lambda\otimes\Lambda \rightarrow \Sym^2(\Lambda) \rightarrow 0.
\end{equation}

We find $H^{-1} \cong \Lambda/2$, according to a commutative diagram of four additional short exact sequences:
\begin{equation}
\label{eq-linear-algebra-diagram}
\begin{tikzcd}[column sep = 1.5em]
	& 0 \ar[d] & 0\ar[d] \\
	& \Sym^2(\Lambda) \ar[r, "\cong"]\ar[d] & \Sym^2(\Lambda) \ar[d] \\
	0 \ar[r] & \Gamma^2(\Lambda) \ar[r]\ar[d] & \Lambda\otimes\Lambda \ar[r]\ar[d] & \wedge^2(\Lambda)\ar[d, "\cong"]\ar[r] & 0 \\
	0 \ar[r] & \Lambda/2 \ar[r]\ar[d] & \Ant^2(\Lambda)\ar[r]\ar[d] & \wedge^2(\Lambda) \ar[r] & 0 \\
	& 0 & 0
\end{tikzcd}
\end{equation}

Next, $H^{-2} = 0$ again by the exact sequence \eqref{eq-linear-algebra-exact-sequence}. The rest of the negative cohomology groups alternate between $\Lambda/2$ and $0$.
\end{void}

\begin{rem}
Let $\check{\Lambda}$ denote the sheaf of $\mathbb Z$-modules dual to $\Lambda$. Items in \eqref{eq-linear-algebra-diagram} admit descriptions as integral ``forms'' on $\check{\Lambda}$. Bilinear forms are classified by $\Lambda\otimes\Lambda$ in the obvious way, $\wedge^2(\Lambda)$ consists of alternating forms via the inclusion in \eqref{eq-linear-algebra-exact-sequence} and $\Gamma^2(\Lambda)$ consists of symmetric bilinear forms via the inclusion in \eqref{eq-linear-algebra-diagram}.

Next, $\Sym^2(\Lambda)$ consists of quadratic forms $Q$ whose value at $x\in\check{\Lambda}$ is given by $c(x,x)$ for any lift $c\in\Lambda\otimes\Lambda$ along \eqref{eq-linear-algebra-exact-sequence}. Finally, $\Ant^2(\Lambda)$ consists of pairs $(a, f)$ where $a$ is an alternating form on $\check{\Lambda}$ and $f : \check{\Lambda} \rightarrow \mathbb Z/2$ is a function satisfying:
$$
f(x_1 + x_2) - f(x_1) - f(x_2) = a(x_1, x_2) \mod 2.
$$
The maps $\Sym^2(\Lambda) \rightarrow \Gamma^2(\Lambda)$, respectively $\Ant^2(\Lambda)\rightarrow \wedge^2(\Lambda)$ in \eqref{eq-linear-algebra-diagram} encode the following operations: going from a quadratic form $Q$ to its symmetric bilinear form $b$, and going from a pair $(a, f)$ to its alternating form $a$.
\end{rem}

\begin{rem}
\label{rem-coinvariants-against-the-sign-character}
We may also consider the $\Sigma_2$ action on $\Lambda\otimes\Lambda$ twisted by the sign character. The derived coinvariants are computed by a complex similar to \eqref{eq-exchange-derived-coinvariants}, but which starts with $\Sym$ as the differential $d^{-1}$. In particular, we find $\Ant^2(\Lambda)$ in $H^0$. Then we have $H^{-1} = 0$, $H^{-2} \cong \Lambda/2$, and the rest of the $H^{-n}$ alternate between them.

The vanishing of $H^{-1}$ has a practical consequence: given a complex $A$ of sheaves of abelian groups in cohomological degrees $[-1, 0]$, the sheaf of $\Sigma_2$-equivariant maps $\Lambda\otimes\Lambda\rightarrow A$, \emph{a priori} given by $\underline{\Maps}{}_{\mathbb Z}((\Lambda\otimes\Lambda)_{\Sigma_2}, A)$ for the derived $\Sigma_2$-coinvariants, is actually identified with $\underline{\Maps}{}_{\mathbb Z}(\Ant^2(\Lambda), A)$.
\end{rem}

\begin{void}
We shall also meet two kinds of ``quadratic functions'' on $\Lambda$, one of which has already made an appearance in \eqref{sect-self-cup-product}.

Namely, for a sheaf $A$ of abelian groups, we write $H^{(1)}(A)$ for the extension of $A$ by $\Sym^2(A)$ defined by the cocycle $a_1, a_2\mapsto a_1a_2$.

Define $H^{(2)}(A)$ to be the chain complex in cohomological degrees $[-1, 0]$:
$$
H^{(2)}(A) := [A\otimes A \rightarrow H^{(1)}(A)],
$$
where the differential is given by the projection $A\otimes A\twoheadrightarrow\Sym^2(A)$ followed by the natural inclusion of $\Sym^2(A)$ in $H^{(1)}(A)$.
\end{void}

\begin{void}
We shall only need the case $A := \Lambda$, in which case $H^{(1)}(\Lambda)$ is torsion-free. Consider the duals as chain complexes of sheaves of $\mathbb Z$-modules:
\begin{align*}
	\check H^{(1)}(\Lambda) &:= \underline{\Hom}(H^{(1)}(\Lambda), \mathbb Z), \\
	\check H^{(2)}(\Lambda) &:= \underline{\Hom}(H^{(2)}(\Lambda), \mathbb Z).
\end{align*}
The complex $\check H^{(2)}(\Lambda)$ is concentrated in degrees $[0, 1]$.

Recall that $\check H^{(1)}(\Lambda)$ is identified with the abelian group of $\mathbb Z$-valued quadratic functions on $\Lambda$, and restriction along $\Sym^2(\Lambda)\rightarrow H^{(1)}(\Lambda)$ yields the \emph{negative} of the associated symmetric form, see \S\ref{void-heisenberg-quadratic-description} for details.
\end{void}

\begin{void}
Note that the short exact sequence \eqref{eq-linear-algebra-exact-sequence} gives rise to a fiber sequence:
\begin{equation}
\label{eq-alternatrized-second-heisenberg-extension}
\wedge^2(\Lambda)[1] \rightarrow H^{(2)}(\Lambda) \rightarrow \Lambda.
\end{equation}
Let us identify the strictly commutative Picard groupoid associated to $H^{(2)}(\Lambda)$.
\end{void}

\begin{lem}
\label{lem-second-heisenberg-central-extension}
There is a \emph{monoidal} equivalence $H^{(2)}(\Lambda) \cong \Lambda\times B(\wedge^2(\Lambda))$, under which the commutativity constraint on $H^{(2)}(\Lambda)$ corresponds to the isomorphism $\lambda_1 + \lambda_2 \cong \lambda_2 + \lambda_1$ defined by $\lambda_1\wedge\lambda_2 \in \wedge^2(\Lambda)$, for all $\lambda_1,\lambda_2\in\Lambda$ viewed as objects of $H^{(2)}(\Lambda)$.
\end{lem}
\begin{proof}
As a pointed morphism, the section $\Lambda\rightarrow H^{(2)}(\Lambda)$ of \eqref{eq-alternatrized-second-heisenberg-extension} is induced from the natural map $v : H^{(1)}(\Lambda)\rightarrow H^{(2)}(\Lambda)$, restricted along the section $\Lambda\subset H^{(1)}(\Lambda)$. It is equipped with an $\mathbb E_1$-monoidal structure coming from the trivialization of $v$ over $\Lambda\otimes\Lambda$.

The commutativity constraint in $H^{(2)}(\Lambda)$ for two objects $\lambda_1,\lambda_2$ coming from $\Lambda$ is then the automorphism of $\lambda_1 + \lambda_2$ determined by $\lambda_1\otimes\lambda_2 - \lambda_2\otimes\lambda_1$ in the subgroup $\wedge^2(\Lambda)\subset\Lambda\otimes\Lambda$.
\end{proof}

\begin{void}
\label{void-second-heisenberg-central-extension}
By Lemma \ref{lem-second-heisenberg-central-extension}, for any sheaf of abelian groups $A_1$, $\mathbb Z$-linear maps $H^{(2)}(\Lambda) \rightarrow B(A_1)$ correspond to pairs $(a, \Lambda_1)$ where $a : \wedge^2(\Lambda)\rightarrow A_1$ is an alternating form and $\Lambda_1$ is a central extension of $\Lambda$ by $A_1$ with commutator $\lambda_1,\lambda_2\mapsto a(\lambda_1, \lambda_2)$.

Restriction along the first map in \eqref{eq-alternatrized-second-heisenberg-extension} corresponds to the functor $(a, \Lambda_1)\mapsto a$. The central extension $\Lambda_1$ is recovered by the restriction along the \emph{monoidal} section $\Lambda\rightarrow H^{(2)}(\Lambda)$.
\end{void}

\begin{void}
\label{void-cosimplicial-system-character-lattice}
The sheaf of $\mathbb Z$-modules $\check{\Lambda}$ has the natural coalgebra structure with respect to the Cartesian symmetric monoidal structure. It is defined by the cosimplicial object $[n] \mapsto \check{\Lambda}^{\oplus n}$, whose coface maps are insertions and whose degeneracy maps are projections. We display the maps between $\check{\Lambda}$ and $\check{\Lambda}^{\oplus 2}$:
$$
d^2_i(x) =
\begin{cases}
	(0, x) & i = 0\\
	(x, x) & i = 1\\
	(x, 0) & i = 2
\end{cases}
;\quad
s^1_i(x_0, x_1) = 
\begin{cases}
x_1 & i = 0\\
x_0 & i = 1
\end{cases}.
$$

Functoriality of the constructions $\check{\Lambda}\mapsto F(\check{\Lambda}) := \check{\Lambda}\otimes\check{\Lambda}$, $\Gamma^2(\check{\Lambda})$, $\wedge^2(\check{\Lambda})$, $\Sym^2(\check{\Lambda})$, $\Ant^2(\check{\Lambda})$, $\check H^{(1)}(\Lambda)$, $\check H^{(2)}(\Lambda)$ determine cosimplicial objects $[n]\mapsto F(\check{\Lambda}^{\oplus n})$ in the category of complexes of sheaves of $\mathbb Z$-modules.
\end{void}

\begin{lem}
\label{lem-calculation-cosimplicial-limit}
There are canonical isomorphisms of complexes:
\begin{enumerate}
	\item $\lim_{[n]}(\check{\Lambda}^{\oplus n}) \cong \check{\Lambda}[-1]$;
	\item $\lim_{[n]}(\check{\Lambda}^{\oplus n}\otimes\check{\Lambda}^{\oplus n}) \cong \check{\Lambda}\otimes\check{\Lambda}[-2]$;
	\item $\lim_{[n]}F_1(\check{\Lambda}^{\oplus n}) \cong F_2(\check{\Lambda})[-2]$ for all functors notated $F_1\Rightarrow F_2$ below:
	$$
	\Gamma^2 \Rightarrow \wedge^2 \Rightarrow \Sym^2 \Rightarrow \Ant^2;
	$$
	\item $\lim_{[n]}(\check H^{(1)}(\Lambda^{\oplus n})) \cong \check H^{(2)}(\Lambda)[-1]$.
\end{enumerate}
\end{lem}
\begin{proof}
The cosimplicial limit of $F(\check{\Lambda}^{\oplus n})$ is equivalent to the complex in degrees $\ge 0$:
\begin{equation}
\label{eq-cosimplicial-limit-as-a-complex}
[0 \xrightarrow{d} F(\check{\Lambda}) \xrightarrow{d} F(\check{\Lambda}^{\oplus 2}) \xrightarrow{d} \cdots],
\end{equation}
where $d$ denotes the alternating sum of the face maps. Inclusion of the subcomplex $\bigcap\ker(s_i)$ of nondegenerate cochains is a homotopy equivalence. This already implies (1), since the subcomplex of nondegenerate cochains is equivalent to $\check{\Lambda}[-1]$. For (2), the subcomplex of nondegenerate cochains vanish in degrees $\ge 3$, while in degrees $[1,2]$ we find:
\begin{equation}
\label{eq-calculation-cosimplicial-limit-square}
\check{\Lambda}\otimes\check{\Lambda} \xrightarrow{d} (\check{\Lambda}\otimes\check{\Lambda})\oplus (\check{\Lambda}\otimes\check{\Lambda}),
\end{equation}
where the differential is the diagonal inclusion. Its quotient is then $\check{\Lambda}\otimes\check{\Lambda}$.

For (3) with $F = \Gamma^2$, $\wedge^2$, or $\Sym^2$, the essential observation is that nondegenerate $n$-cochains vanish for $n\ge 3$ and nondegenerate $2$-cochains are equivalent to $\check{\Lambda}\otimes\check{\Lambda}$. The differential it receives from nondegenerate $1$-cochains $d : F(\check{\Lambda}) \rightarrow \check{\Lambda}\otimes\check{\Lambda}$ is the natural inclusion, so we conclude by the short exact sequences in \eqref{eq-linear-algebra-exact-sequence}, \eqref{eq-linear-algebra-diagram} (see also \cite[\S3.9]{MR1896177}).

For (4), we identify elements of $\check H^{(1)}(\Lambda^{\oplus n})$ with quadratic functions $Q^{(n)} : \Lambda^{\oplus n} \rightarrow \mathbb Z$. Thus nondegenerate cochains for $n\ge 3$ vanish and are isomorphic to $\check{\Lambda}\otimes\check{\Lambda}$ for $n = 2$: those $Q^{(2)}$ which vanish on $\Lambda\oplus 0$ and $0\oplus\Lambda$ are uniquely determined by the pairing between $\Lambda\oplus 0$ and $0\oplus\Lambda$, which can be arbitrary. The nondegenerate cochains form the complex in cohomological degrees $[1,2]$:
$$
[\check H^{(1)}(\Lambda) \xrightarrow{d} \check{\Lambda}\otimes\check{\Lambda}],
$$
where $d$ takes $Q^{(1)}$ to the bilinear form $\lambda_1,\lambda_2\mapsto Q^{(1)}(\lambda_1) + Q^{(1)}(\lambda_2) - Q^{(1)}(\lambda_1 + \lambda_2)$ on $\Lambda$. Taking the negative sign into account, this is precisely $\check H^{(2)}(\Lambda)[-1]$.
\end{proof}

\subsection{Even $\vartheta$-data}
\label{sect-theta-data}

\begin{void}
Let $S$ be a scheme, and let $\Lambda$ be a locally constant sheaf of finite free $\mathbb Z$-modules on its \'etale site.

Denote by $\vartheta^{(1)}(\Lambda)$ the sheaf of abelian groups whose sections are pairs $(a, f)$ where $a\in\wedge^2(\check{\Lambda})$ is an integral alternating form on $\Lambda$, and $f : \Lambda \rightarrow \mathbb G_m$ is a function with $f(\lambda_1 + \lambda_2) - f(\lambda_1) - f(\lambda_2) = (-1)^{a(\lambda_1,\lambda_2)}$.

Then $\vartheta^{(1)}(\Lambda)$ is simultaneously a pullback and a pushout:
\begin{equation}
\label{eq-first-theta-data-integral}
\begin{tikzcd}[column sep = 1em]
	\Gamma^2(\check{\Lambda}) \ar[r]\ar[d, swap, "b\mapsto(\lambda\mapsto(-1)^{b(\lambda,\lambda)})"] & \check{\Lambda} \otimes \check{\Lambda}\ar[d] \\
	\check{\Lambda}\otimes\mathbb G_m \ar[r] & \vartheta^{(1)}(\Lambda) \ar[d]\ar[r] & \wedge^2(\check{\Lambda}) \ar[d, "{a\mapsto(\lambda_1,\lambda_2\mapsto(-1)^{a(\lambda_1,\lambda_2)})}"]\\
	& \check H^{(1)}(\Lambda)\otimes\mathbb G_m \ar[r] & \Gamma^2(\check{\Lambda})\otimes\mathbb G_m
\end{tikzcd}
\end{equation}
Here, the morphism $\check{\Lambda}\otimes\check{\Lambda} \rightarrow \vartheta^{(1)}(\Lambda)$ sends a bilinear form $c$ to the pair $(a, f)$, where $a$ is the anti-symmetrization of $c$ and $f$ is the map $\lambda\mapsto(-1)^{c(\lambda,\lambda)}$.

The fact that $\wedge^2(\check{\Lambda})$ is identified with the cokernel of $\Gamma^2(\check{\Lambda})\rightarrow\check{\Lambda}\otimes\check{\Lambda}$ implies that the middle sequence in \eqref{eq-first-theta-data-integral} is exact.
\end{void}

\begin{rem}
Sections of $\vartheta^{(1)}(\Lambda)$ are in canonical bijection with pointed morphisms from $T:=\Lambda\otimes\mathbb G_m$ to $\underline K_2$, when the base scheme $S$ is regular of finite type over a field, see \cite[Construction 3.5]{MR1896177}.
\end{rem}

\begin{void}
We define $\vartheta^{(2)}(\Lambda)$ to be the limit of the cosimplicial diagram $[n] \mapsto B^2(\vartheta^{(1)}(\Lambda^{\oplus n}))$ of sheaves of connective $\mathbb Z$-module spectra, see \S\ref{void-cosimplicial-system-character-lattice}.

Applying the same limit to the diagram \eqref{eq-first-theta-data-integral} and using Lemma \ref{lem-calculation-cosimplicial-limit}, we see that $\vartheta^{(2)}(\Lambda)$ is simultaneously a pullback and a pushout:
\begin{equation}
\label{eq-second-theta-data-integral}
\begin{tikzcd}[column sep = 1em]
	\wedge^2(\check{\Lambda}) \ar[r]\ar[d, "(1)"] & \check{\Lambda}\otimes\check{\Lambda} \ar[d] \\
	\check{\Lambda}\otimes B\mathbb G_m \ar[r] & \vartheta^{(2)}(\Lambda) \ar[r]\ar[d] & \Sym^2(\check{\Lambda}) \ar[d, "(2)"] \\
	& \check H^{(2)}(\Lambda)\otimes B\mathbb G_m \ar[r] & \wedge^2(\check{\Lambda})\otimes\mathbb G_m
\end{tikzcd}
\end{equation}

The arrow labeled (1) is given by the compositions:
$$
\wedge^2(\check{\Lambda}) \xrightarrow{\Ant(\check{\Lambda})} B(\check{\Lambda}/2) \xrightarrow{(-1)}\check{\Lambda}\otimes B\mathbb G_m,
$$
where the first map is the coboundary of the bottom exact sequence of \eqref{eq-linear-algebra-diagram} and the second map is the tensor product of $\check{\Lambda}$ with $\mathbb Z/2\rightarrow\mathbb G_m$, $a\mapsto (-1)^a$.

The arrow labeled (2) is the composition:
\begin{equation}
\label{eq-quadratic-form-to-alternating-pairing-integral}
\Sym^2(\check{\Lambda}) \xrightarrow{Q\mapsto b} \wedge^2(\check{\Lambda})/2 \xrightarrow{(-1)} \wedge^2(\check{\Lambda})\otimes\mathbb G_m,
\end{equation}
where the first map sends $Q$ to its symmetric form $b$, viewed as an \emph{alternating} form valued in $\mathbb Z/2$, and the second map is the tensor product of $\wedge^2(\check{\Lambda})$ with $\mathbb Z/2\rightarrow\mathbb G_m$, $a\mapsto (-1)^a$.

Furthermore, the middle sequence in \eqref{eq-second-theta-data-integral} is canonically a triangle.
\end{void}

\begin{rem}
\label{rem-second-heisenberg-theta-data}
Interpreting sections of $\check H^{(2)}(\Lambda)\otimes B\mathbb G_m$ as central extensions of $\Lambda$ by $\mathbb G_m$ with prescribed commutators (\S\ref{void-second-heisenberg-central-extension}), the pullback square in \eqref{eq-second-theta-data-integral} says that sections of $\vartheta^{(2)}(\Lambda)$ are pairs $(Q, \Lambda_1)$, where $Q$ is a quadratic form on $\Lambda$, and $\Lambda_1$ is a central extension of $\Lambda$ by $\mathbb G_m$ whose commutator is $\lambda_1, \lambda_2 \mapsto (-1)^{b(\lambda_1, \lambda_2)}$, for $b$ being the symmetric form associated to $Q$.

This description shows that sections of $\vartheta^{(2)}(\Lambda)$ recover the ``even $\vartheta$-data'' of Beilinson--Drinfeld \cite[\S3.10.3]{MR2058353} except the terms involving $\omega_X$. They are also equivalent to central extensions of $T := \Lambda\otimes\mathbb G_m$ by $\underline K_2$, when the base scheme $S$ is regular of finite type over a field, see \cite[Theorem 3.16]{MR1896177}.
\end{rem}

\begin{rem}
\label{rem-composition-cocycle-classification-datum}
The composition $\check{\Lambda}\otimes\check{\Lambda} \rightarrow \check H^{(2)}(\Lambda)\otimes B\mathbb G_m$ in \eqref{eq-second-theta-data-integral} can be described as follows: it sends a bilinear form $c$ to the central extension of $\Lambda$ by $\mathbb G_m$ defined by the cocycle $\lambda_1, \lambda_2\mapsto (-1)^{c(\lambda_1,\lambda_2)}$.
\end{rem}

\begin{void}
\label{void-theta-data-explicit-description}
Let $A$ be a locally constant \'etale sheaf of finite abelian groups over $S$ of invertible order. Then we may form $\vartheta^{(1)}(\Lambda)\otimes A(-1)$ and $\vartheta^{(2)}(\Lambda)\otimes A(-1)$. (The Tate twists are introduced to conform with the conventions in \S\ref{sect-classification-tori} below.)

There are obvious analogues of the pullback/pushout diagrams \eqref{eq-first-theta-data-integral}, \eqref{eq-second-theta-data-integral}. In particular, sections of $\vartheta^{(2)}(\Lambda)\otimes A(-1)$ consist of triples $(Q, F, h)$ where:
\begin{enumerate}
	\item $Q$ is an $A(-1)$-valued quadratic form on $\Lambda$;
	\item $F : H^{(2)}(\Lambda) \rightarrow B^2A$ is a $\mathbb Z$-linear morphism;
	\item $h$ is a $\mathbb Z$-linear isomorphism between the restriction of $F$ to $B(\wedge^2(\Lambda))$ and the image of $Q$ along the map coming from tensoring \eqref{eq-quadratic-form-to-alternating-pairing-integral} with $A(-1)$:
	\begin{equation}
	\label{eq-quadratic-form-to-alternating-pairing}
	\Sym^2(\check{\Lambda})\otimes A(-1) \rightarrow \wedge^2(\check{\Lambda})\otimes BA.
	\end{equation}
\end{enumerate}

Furthermore, we have a canonical triangle coming from the middle triangle of \eqref{eq-second-theta-data-integral}:
\begin{equation}
\label{eq-theta-data-triangle-coefficient}
	\check{\Lambda}\otimes B^2A \rightarrow \vartheta^{(2)}(\Lambda)\otimes A(-1) \rightarrow \Sym^2(\check{\Lambda})\otimes A(-1).
\end{equation}
\end{void}

\begin{rem}
If $A(-1)$ has no $2$-torsion, then $Q$ has zero image in $\wedge^2(\check{\Lambda})\otimes BA$; indeed, \eqref{eq-quadratic-form-to-alternating-pairing} factors through $\wedge^2(\check{\Lambda})/2\otimes A(-1) = 0$. In this case, \eqref{eq-theta-data-triangle-coefficient} canonically splits.
\end{rem}

\begin{rem}
\label{rem-splitting-theta-data}
If $\Lambda$ has rank $1$, then $\wedge^2(\check{\Lambda}) = 0$ and \eqref{eq-theta-data-triangle-coefficient} also splits: the map $\Sym^2(\check{\Lambda})\otimes A(-1)\rightarrow \vartheta^{(2)}(\Lambda)\otimes A(-1)$ is defined by lifting a quadratic form uniquely to $\check{\Lambda}^{\otimes 2}\otimes A(-1)$ and applying the map induced from \eqref{eq-second-theta-data-integral}.

In fact, any basis $\Lambda\cong\bigoplus_{i\in I}\mathbb Z e_i$ induces a splitting of \eqref{eq-theta-data-triangle-coefficient} by the restriction maps composed with the splitting for $\Lambda = \mathbb Z$:
\begin{align*}
\bigoplus_{i\in I} e_i^* : \vartheta^{(2)}(\Lambda)\otimes A(-1) &\rightarrow \bigoplus_{i\in I} \vartheta^{(2)}(\mathbb Z)\otimes A(-1) \\
&\rightarrow \bigoplus_{i\in I} B^2A \cong \check{\Lambda}\otimes B^2A.
\end{align*}
\end{rem}

\subsection{Classification: tori}
\label{sect-classification-tori}

\begin{void}
\label{void-classification-tori-setup}
Let $S$ be a scheme. Suppose that $A$ is a locally constant \'etale sheaf of abelian groups over $S$ of invertible order. Let $T\rightarrow S$ be a torus with sheaf of cocharacters $\Lambda$, so $T\cong \Lambda\otimes\mathbb G_m$.

Let $\underline{\Gamma}{}_e(BT, B^4A(1))$ denote the \'etale sheaf of connective $\mathbb Z$-module spectra over the base scheme $S$, whose sections are rigidified maps $BT\rightarrow B^4A(1)$. The classification of \'etale metaplectic covers of $T$ amounts to describing the \'etale sheaf $\underline{\Gamma}{}_e(BT, B^4A(1))$.
\end{void}

\begin{thm}
\label{thm-classification-torus}
There is a canonical equivalence of \'etale sheaves of connective $\mathbb Z$-module spectra:
\begin{equation}
\label{eq-classification-torus}
\vartheta^{(2)}(\Lambda) \otimes A(-1) \cong \underline{\Gamma}{}_e(BT, B^4A(1)).
\end{equation}
In particular, $\underline{\Gamma}{}_e(BT, B^4A(1))$ is situated in a pushout square, a pullback square, and a cofiber sequence:
\begin{equation}
\label{eq-second-theta-data-coefficients}
\begin{tikzcd}[column sep = 1em]
	\wedge^2(\check{\Lambda})\otimes A(-1) \ar[r]\ar[d] & \check{\Lambda}^{\otimes 2}\otimes A(-1) \ar[d, "(1)"] \\
	\check{\Lambda}\otimes B^2A \ar[r, "(2)"] & \underline{\Gamma}{}_e(BT, B^4A(1)) \ar[r]\ar[d] & \Sym^2(\check{\Lambda})\otimes A(-1) \ar[d] \\
	& \check H^{(2)}(\Lambda)\otimes B^2A \ar[r] & \wedge^2(\check{\Lambda})\otimes BA
\end{tikzcd}
\end{equation}
\end{thm}

\begin{void}
\label{void-construction-of-torus-covers}
Let us explicitly describe the two labeled maps in \eqref{eq-second-theta-data-coefficients}. Recall that the Kummer torsor $\Psi : \mathbb G_m \rightarrow \lim_N B(\mu_N)$ (where $N$ ranges over integers $\ge 1$ invertible on $S$) induces a $\mathbb Z$-linear map $B\mathbb G_m \rightarrow \lim_NB^2(\mu_N)$, so in particular a rigidified section over $B\mathbb G_m$.

Our descriptions of the maps (1) \& (2) are as follows:
\begin{enumerate}
	\item $x_1\otimes x_2\otimes a \mapsto (x_1^*\Psi \cup x_2^*\Psi)\otimes a$, for $x_1,x_2\in\check{\Lambda}$ and $a\in A(-1)$;
	\item $x\otimes t \mapsto x^*\Psi^*(t)$, where $t \in B^2A$ is viewed as a map $\lim_N B^2(\mu_N) \rightarrow B^4A(1)$.
\end{enumerate}

It is not \emph{a priori} clear how to identify them over $\wedge^2(\check{\Lambda})\otimes A(-1)$. This identification will be constructed in the proof of Theorem \ref{thm-classification-torus}. 

In this proof, we shall first construct the analogous diagram for $\underline{\Gamma}{}_e(T, B^2A(1))$, where the necessary identification will come from the self-cup product formula $\Psi\cup\Psi \cong \Psi^*(\Psi(-1))$ of \S\ref{sec-cup-product}. The corresponding diagram for $\underline{\Gamma}{}_e(BT, B^4A(1))$ follows from a cosimplicial limit.
\end{void}

\begin{rem}
\label{rem-tori-classification-linear-part}
Informally, the functor (1) produces extensions of $T$ by $B^2A(1)$ which are ``defined by cocycles'', while the functor (2) produces $\mathbb Z$-linear extensions. We justify the second statement now but defer the first to \S\ref{sect-torus-covers-defined-by-cocycles}.

Indeed, the functor (2) in \eqref{eq-second-theta-data-coefficients} defines morphisms of complexes $T[1] \rightarrow A(1)[4]$, i.e.~$\mathbb Z$-linear maps $BT \rightarrow B^4A(1)$. In fact, it is an equivalence onto the groupoid of $\mathbb Z$-linear maps by the following calculation using Lemma \ref{lem-cochains-on-Gm-equivalences}:
\begin{align}
\notag
\underline{\Maps}{}_{\mathbb Z}(BT, B^4A(1)) &\cong \check{\Lambda}\otimes\underline{\Maps}{}_{\mathbb Z}(B\mathbb G_m, B^4A(1)) \\
\label{eq-tori-classification-linear-part}
&\cong \check{\Lambda}\otimes B^2A.
\end{align}
The cofiber sequence in \eqref{eq-second-theta-data-coefficients} implies that $\mathbb Z$-linear maps $BT \rightarrow B^4A(1)$ form a \emph{full} subgroupoid of the goupoid of all pointed maps.
\end{rem}

\begin{prop}
\label{prop-classification-torus-first-theta}
There is a canonical equivalence of \'etale sheaves of connective $\mathbb Z$-module spectra:
\begin{equation}
\label{eq-classification-torus-first-theta}
\vartheta^{(1)}(\Lambda) \otimes A(-1) \cong \underline{\Gamma}{}_e(T, B^2A(1)).
\end{equation}
\end{prop}
\begin{proof}
We shall first construct a functor from $\vartheta^{(1)}(\Lambda) \otimes A(-1)$ to $\underline{\Gamma}{}_e(T, B^2A(1))$ using the description of the former as a pushout, see \eqref{eq-first-theta-data-integral}. Indeed, let us first construct maps:
\begin{equation}
\label{eq-classification-torus-pushout-maps-first-theta}
\begin{tikzcd}[column sep = 1em]
	& \check{\Lambda}^{\otimes 2} \otimes A(-1) \ar[d, "(1)"] \\
	\check{\Lambda}\otimes A[1] \ar[r, "(2)"] & \underline{\Gamma}{}_e(T, B^2A(1))
\end{tikzcd}
\end{equation}
Writing $\Psi : \mathbb G_m \rightarrow \lim_N B(\mu_N)$ for the Kummer torsor, we describe these maps as follows:
\begin{enumerate}
	\item $x_1\otimes x_2\otimes a \mapsto (x_1^*\Psi \cup x_2^*\Psi)\otimes a$, for $x_1, x_2\in\check{\Lambda}$ and $a\in A(-1)$;
	\item $x\otimes t \mapsto x^*\Psi^*(t)$, where $t\in BA$ is viewed as a map $\lim_N B(\mu_N) \rightarrow B^2A(1)$.
\end{enumerate}
Let us construct an identification of their restrictions to $\Gamma^2(\check{\Lambda})\otimes A(-1)$.

It suffices to treat the case $A = \mu_N$ for some $N\ge 1$ invertible on $S$. In this case, \eqref{eq-classification-torus-pushout-maps-first-theta} consists of $\mathbb Z/N$-linear morphisms, so by adjunction, it suffices to identify the two circuits of the following diagram:
\begin{equation}
\label{eq-classification-torus-homotopy}
\begin{tikzcd}[column sep = 1em]
	\Gamma^2(\check{\Lambda}) \ar[r]\ar[d] & \check{\Lambda}^{\otimes 2} \ar[d, twoheadrightarrow] \\
	\check{\Lambda}/2 \ar[r, "x\mapsto x\otimes x"]\ar[d, swap, "\Psi(-1)"] & \Ant^2(\check{\Lambda}) \ar[d, "(1)"] \\
	\check{\Lambda}\otimes B\mu_N \ar[r, "(2)"] & \underline{\Gamma}{}_e(T, B^2\mu_N^{\otimes 2})
\end{tikzcd}
\end{equation}
Here, we have factored the morphism (1) through $\Ant^2(\check{\Lambda})$ using the canonical anti-symmetric structure on $x_1\otimes x_2\mapsto x_1^*\Psi \cup x_2^*\Psi$ and the vanishing of $\tn H^{-1}$ in Remark \ref{rem-coinvariants-against-the-sign-character}.

The commutativity of the lower square in \eqref{eq-classification-torus-homotopy} can in turn be constructed by identifying the two circuits after pre-composing with $\check{\Lambda}\twoheadrightarrow \check{\Lambda}/2$, and showing that this identification is compatible with $2$-torsion. Over $\check{\Lambda}$, the two circuits are defined by the Yoneda pairing:
$$
\check{\Lambda}\otimes\underline{\Gamma}{}_e(\mathbb G_m, B^2\mu_N^{\otimes 2}) \rightarrow \underline{\Gamma}{}_e(T, B^2\mu_N^{\otimes 2}),\quad x\otimes t\mapsto x^*(t)
$$
applied to sections $t := \Psi\cup\Psi$ respectively $\Psi^*(\Psi(-1))$. They are identified by Theorem \ref{thm-self-cup-product} compatibly with the $2$-torsion structures.

Having constructed the functor from $\vartheta^{(1)}(\Lambda) \otimes A(-1)$ to $\underline{\Gamma}{}_e(T, B^2A(1))$, we prove that it is an equivalence by calculating the induced maps on $\pi_n$. Since the cohomology of $T\rightarrow S$ commutes with arbitrary base change \cite[Rappel 1.5.1]{MR1441006}, we may replace $S$ by the spectrum of a separably closed field. In this case, there holds:
$$
H^i(T, A(1)) \cong
\begin{cases}
	\check{\Lambda}\otimes A & i = 1\\
	\wedge^2(\check{\Lambda})\otimes A(-1) & i = 2
\end{cases}
$$
and the isomorphisms are defined by $\Psi$ and the cup product.

On the other hand, $\vartheta^{(1)}(\Lambda)\otimes A(-1)$ is an extension of $\wedge^2(\check{\Lambda})\otimes A(-1)$ by $\check{\Lambda}\otimes BA$ by the maps defined in \eqref{eq-first-theta-data-integral}. This finishes the proof.
\end{proof}

\begin{proof}[Proof of Theorem \ref{thm-classification-torus}]
Consider the cosimplicial system $[n] \mapsto \check{\Lambda}^{\oplus n}$ defined by the coalgebra structure on $\check{\Lambda}$, see \S\ref{void-cosimplicial-system-character-lattice}. Since $\vartheta^{(2)}(\Lambda)\otimes A(-1)$ is identified with $\lim_{[n]}(\vartheta^{(1)}(\Lambda^{\oplus n})\otimes A(-1)[2])$ and Proposition \ref{prop-classification-torus-first-theta} identifies $\vartheta^{(1)}(\Lambda)\otimes A(-1)$ with $\underline{\Gamma}{}_e(T, B^2A(1))$ naturally in $T$, it suffices to establish an equivalence:
\begin{equation}
\label{eq-classification-torus-descent}
\underline{\Gamma}{}_e(BT, B^4A(1)) \cong \lim_{[n]} B^2\underline{\Gamma}{}_e(T^{\times n}, B^2A(1)).
\end{equation}

We argue that expressing $BT$ as the colimit of the simplicial system $[n]\mapsto T^{\times n}$ defines equivalences whose composition gives \eqref{eq-classification-torus-descent}:
\begin{align}
	\label{eq-descent-torus-pointed}
	\underline{\Gamma}{}_e(BT, B^4A(1)) &\cong \lim_{[n]} \underline{\Gamma}{}_e(T^{\times n}, B^4A(1)) \\
	\label{eq-descent-torus-pointed-zero-vanishing}
	&\cong \lim_{[n]} B\underline{\Gamma}{}_e(T^{\times n}, B^3 A(1)) \\
	\label{eq-descent-torus-pointed-first-vanishing}
	&\cong \lim_{[n]} B^2\underline{\Gamma}{}_e(T^{\times n}, B^2A(1)).
\end{align}

Indeed, \eqref{eq-descent-torus-pointed} follows from writing $\underline{\Gamma}(S, B^4A(1))$ as the limit of $[n]\mapsto\underline{\Gamma}(S, B^4A(1))$ and commuting the limit with fiber product. The isomorphism \eqref{eq-descent-torus-pointed-zero-vanishing} follows because the zeroth simplex $\underline{\Gamma}{}_e(S, B^3A(1))$ is contractible. Next, we observe that:
$$
\pi_0\underline{\Gamma}{}_e(T^{\times n}, B^3A(1))\cong \wedge^3(\Lambda^{\oplus n}) \otimes A(-2).
$$
However, we have $\lim_{[n]}(\wedge^3(\check{\Lambda}^{\oplus n})) \cong \Sym^3(\check{\Lambda})[-3]$ as chain complexes, so the limit of spaces $\lim_{[n]}B(\pi_0\underline{\Gamma}{}_e(T^{\times n}, B^3A(1)))$ is contractible, giving \eqref{eq-descent-torus-pointed-first-vanishing}.
\end{proof}

\subsection{Covers defined by cocycles}
\label{sect-torus-covers-defined-by-cocycles}

\begin{void}
Let us study the morphism (1) in \eqref{eq-second-theta-data-coefficients} more closely. We may regard any pointed morphism $BT \rightarrow B^4A(1)$ as an $\mathbb E_1$-monoidal morphism $T \rightarrow B^3A(1)$ and consequently an $\mathbb E_1$-monoidal extension $T^{\dagger}$ of $T$ by $B^2A(1)$ (Lemma \ref{lem-monoidal-morphisms-classified-by-groups}).

If the underlying pointed morphism of $T\rightarrow B^3A(1)$ is trivialized, then $T^{\dagger} \rightarrow T$ admits a section $s$ as a pointed stack. Consequently, we obtain a morphism:
\begin{equation}
\label{eq-torus-covers-cocycle}
T\times T\rightarrow B^2A(1),\quad (t_1, t_2)\mapsto s(t_1t_2)s(t_1)^{-1}s(t_2)^{-1}.
\end{equation}
\end{void}

\begin{void}
\label{void-birigidified-morphism}
Suppose that $X_1$, $X_2$ are stacks over $S$ pointed by $e_1 : S\rightarrow X_1$, $e_2 : S\rightarrow X_2$ and $n\ge 0$ is an integer.

We call a section of $B^nA(1)$ over $X_1\times X_2$ \emph{bi-rigidified} if it is equipped with trivializations along $e_1\times X_2$, $X_1\times e_2$ and an isomorphism of the two induced trivializations along $e_1\times e_2$. We denote the \'etale sheaf of bi-rigidified morphism by:
$$
\underline{\Gamma}{}_{e_1, e_2}(X_1\times X_2, B^nA(1)).
$$

Note that any section $x_1\otimes x_2\otimes a$ of $\check{\Lambda}^{\otimes 2}\otimes A(-1)$ naturally defines a bi-rigidified morphism $T\times T\rightarrow B^2A(1)$ given by $(p_1^*x_1^*\Psi \cup p_2^*x_2^*\Psi)\otimes a$, where $p_1$, $p_2$ denote the two projections from $T\times T$ to $T$. Here, $\Psi$ denotes the Kummer torsor on $\mathbb G_m$.
\end{void}

\begin{lem}
\label{lem-birigidified-morphism-classification}
The above functor and its analogue for $BT\times BT$ define equivalences:
\begin{align*}
\check{\Lambda}^{\otimes 2}\otimes A(-1) &\cong \underline{\Gamma}{}_{e,e}(T\times T, B^2A(1)), \\
\check{\Lambda}^{\otimes 2}\otimes A(-1) &\cong \underline{\Gamma}{}_{e, e}(BT\times BT, B^4A(1)).
\end{align*}
\end{lem}
\begin{proof}
Indeed, morphisms $T\times T \rightarrow B^2A(1)$ rigidified along $T\times e$ are equivalent to morphisms
\begin{equation}
\label{eq-birigidified-morphism-one-sided}
T\rightarrow \underline{\Gamma}{}_e(T, B^2A(1)).
\end{equation}
Hence rigidifying it along $e\times T$, \emph{as} a morphism rigidified along $T\times e$, is equivalent to rigidifying the corresponding morphism \eqref{eq-birigidified-morphism-one-sided}.

By Proposition \ref{prop-classification-torus-first-theta}, the neutral component of $\underline{\Gamma}{}_e(T, B^2A(1))$ is identified with $B^2(\check{\Lambda}\otimes A)$. Hence, a rigidified morphism \eqref{eq-birigidified-morphism-one-sided} is equivalent to a rigidified morphism $BT \rightarrow B^2(\check{\Lambda}\otimes A)$, which by Proposition \ref{prop-classification-torus-first-theta} again (applied to coefficient group $\check{\Lambda}\otimes A$) is equivalent to a section of $\check{\Lambda}^{\otimes 2}\otimes A(-1)$.

The statment for $\underline{\Gamma}{}_{e, e}(BT\times BT, B^4A(1))$ is proved in the same manner, substituting Proposition \ref{prop-classification-torus-first-theta} by Theorem \ref{thm-classification-torus}.
\end{proof}

\begin{void}
Consider the functor of forgetting the $\mathbb E_1$-monoidal structure:
\begin{align}
\notag
\underline{\Gamma}{}_e(BT, B^4A(1)) &\cong \underline{\Maps}{}_{\mathbb E_1}(T, B^3A(1)) \\
\label{eq-forgetting-monoidal-structure}
&\rightarrow \underline{\Gamma}{}_e(T, B^3A(1)).
\end{align}

An object in the fiber of \eqref{eq-forgetting-monoidal-structure} naturally defines a bi-rigidified morphism \eqref{eq-torus-covers-cocycle} as the cocycle of the corresponding extension of $T$ by $B^2A(1)$, which gives a section of $\check{\Lambda}^{\otimes 2}\otimes A(-1)$ by Lemma \ref{lem-birigidified-morphism-classification}: we call it the \emph{cocycle} of this object.

The following result is a formulation of the idea that sections of $\check{\Lambda}^{\otimes 2}\otimes A(-1)$ yield \'etale metaplectic covers ``defined by cocycles''.
\end{void}

\begin{prop}
\label{prop-torus-cocycle-covers}
The functor \tn{(1)} of \eqref{eq-second-theta-data-coefficients} defines an equivalence between $\check{\Lambda}^{\otimes 2}\otimes A(-1)$ and the fiber of \eqref{eq-forgetting-monoidal-structure}. Its inverse is given by the association of cocycles.
\end{prop}
\begin{proof}
Recall from the proof of Theorem \ref{thm-classification-torus} that the functor (1) of \eqref{eq-second-theta-data-coefficients} is defined by the cosimplicial limit of a system of maps of complexes parametrized by $[n]\in\Delta$:
$$
(\check{\Lambda}^{\oplus n})^{\otimes 2}\otimes A(-1)[2] \rightarrow \underline{\Gamma}{}_e(T^{\times n}, B^2A(1))[2],
$$
using Lemma \ref{lem-calculation-cosimplicial-limit}(2). The forgetful functor \eqref{eq-forgetting-monoidal-structure} is defined by the evaluation at $[1]\in\Delta$:
\begin{align*}
\lim_{[n]} \underline{\Gamma}{}_e(T^{\times n}, B^2A(1))[2]  &\rightarrow \underline{\Gamma}{}_e(T, B^2A(1))[1] \\
&\cong B\underline{\Gamma}{}_e(T, B^2A(1)) \subset \underline{\Gamma}{}_e(T, B^3A(1)).
\end{align*}

In particular, we obtain a commutative diagram:
\begin{equation}
\label{eq-forgetful-monoidal-diagram}
\begin{tikzcd}
	\check{\Lambda}^{\otimes 2}\otimes A(-1) \ar[d, "(1)"] \ar[r, "\ev_{[1]}"] & B(\check{\Lambda}^{\otimes 2}\otimes A(-1)) \ar[d] \\
	\underline{\Gamma}{}_e(BT, B^4A(1)) \ar[r, "\ev_{[1]}"] & B\underline{\Gamma}{}_e(T, B^2A(1))
\end{tikzcd}
\end{equation}
The top horizontal arrow is canonically trivialized by splitting the quotient of \eqref{eq-calculation-cosimplicial-limit-square}. Therefore, the lower circuit of \eqref{eq-forgetful-monoidal-diagram} is canonically trivialized, so we obtain a functor from $\check{\Lambda}^{\otimes 2}\otimes A(-1)$ to the fiber of the forgetful functor.

To conclude that it is an equivalence, it suffices to observe that this lower circuit induces a long exact sequence of homotopy groups:
\begin{align*}
0 \rightarrow A \cong A \rightarrow 0 \rightarrow 0 &\rightarrow \wedge^2\check{\Lambda}\otimes A(-1) \\
&\rightarrow \check{\Lambda}^{\otimes 2}\otimes A(-1) \rightarrow \Sym^2(\check{\Lambda})\otimes A(-1) \rightarrow 0.
\end{align*}
The assertion on the inverse is left to the interested reader.
\end{proof}

\subsection{The $\mathbb E_1$-monoidal morphism $F : \Lambda\rightarrow B^2A$}
\label{sect-monoidal-classification-data}

\begin{void}
Viewing a pointed morphism $\mu : BT \rightarrow B^4A(1)$ as an $\mathbb E_1$-monoidal morphism $T \rightarrow B^3A(1)$, we may apply the functor $\underline{\Gamma}{}_e(\mathbb G_m, -)$ to obtain an $\mathbb E_1$-monoidal morphism:
\begin{align}
\notag
\Lambda &\rightarrow \underline{\Gamma}{}_e(\mathbb G_m, T) \\
&\rightarrow \underline{\Gamma}{}_e(\mathbb G_m, B^3A(1)) \cong B^2(A),
\label{eq-monoidal-morphism-from-sections}
\end{align}
where the last isomorphism follows from the proof of Lemma \ref{lem-cochains-on-Gm-equivalences}, using the vanishing of $H^3(\mathbb G_{m, \bar s}, A(1))$ at a geometric point $\bar s$.

The following Proposition tells us how to read off the $\mathbb E_1$-monoidal morphism \eqref{eq-monoidal-morphism-from-sections} from the $\vartheta$-datum corresponding to $\mu$.
\end{void}

\begin{prop}
\label{prop-monoidal-morphism-from-theta-data}
The following diagram is canonically commutative:
\begin{equation}
\label{eq-monoidal-morphism-from-theta-data}
\begin{tikzcd}[column sep = 1em]
	\underline{\Gamma}{}_e(BT, B^4A(1)) \ar[r, phantom, "\cong"]\ar[d] & \underline{\Maps}{}_{\mathbb E_1}(T, B^3A(1)) \ar[d, "{\underline{\Gamma}{}_e(\mathbb G_m,-)}"] \\
	\check H^{(2)}(\Lambda)\otimes B^2A \ar[r] & \underline{\Maps}{}_{\mathbb E_1}(\Lambda, B^2A)
\end{tikzcd}
\end{equation}
where the left vertical arrow is as in \eqref{eq-second-theta-data-coefficients} and the bottom horizontal arrow is the restriction along the monoidal section of the map $H^{(2)}(\Lambda)\rightarrow \Lambda$ in \eqref{eq-alternatrized-second-heisenberg-extension}.
\end{prop}
\begin{proof}
As in the proof of Theorem \ref{thm-classification-torus}, we shall produce this commutative diagram from an analogous commutative diagram for $\underline{\Gamma}{}_e(T, B^2A(1))$ via a cosimplicial limit.

The relevant diagram for $\underline{\Gamma}{}_e(T, B^2A(1))$ is the following one:
\begin{equation}
\label{eq-pointed-morphism-from-theta-data}
\begin{tikzcd}[column sep = 1em]
	\underline{\Gamma}{}_e(T, B^2A(1)) \ar[d]\ar[r, phantom, "\cong"] & \underline{\Gamma}{}_e(T, B^2A(1)) \ar[d, "{\underline{\Gamma}{}_e(\mathbb G_m, -)}"] \\
	\check H^{(1)}(\Lambda)\otimes BA \ar[r] & \underline{\Gamma}{}_e(\Lambda, BA)
\end{tikzcd}
\end{equation}
To prove that it is commutative, we use the description of $\underline{\Gamma}{}_e(T, B^2A(1))$ as a pushout, provided by Proposition \ref{prop-classification-torus-first-theta} and \eqref{eq-first-theta-data-integral}. Namely, we identify the two circuits of \eqref{eq-pointed-morphism-from-theta-data} after pre-composing with the two maps of \eqref{eq-classification-torus-pushout-maps-first-theta}, and check that these identifications agree over $\Gamma^2(\check{\Lambda})\otimes A(-1)$.

The identification over $\check{\Lambda}\otimes BA$ is clear. To make the identification over $\check{\Lambda}^{\otimes 2}\otimes A(-1)$, we may assume $A = \mu_N$ for some $N$ invertible on $S$. Along the upper circuit of \eqref{eq-pointed-morphism-from-theta-data}, we thus find the following map by adjunction:
\begin{equation}
\label{eq-pointed-morphism-from-theta-data-calculation}
\check{\Lambda}^{\otimes 2}\rightarrow \underline{\Gamma}{}_e(\Lambda, BA),\quad x_1\otimes x_2\mapsto (\lambda\mapsto \Psi^{x_1(\lambda)}\cup \Psi^{x_2(\lambda)}),
\end{equation}
where $\Psi^{x_1(\lambda)}\cup \Psi^{x_2(\lambda)}$ is viewed as a section of $B(A)$ via the equivalence $\underline{\Gamma}{}_e(\mathbb G_m, B^2A(1)) \cong B(A)$ of Lemma \ref{lem-cochains-on-Gm-equivalences}, defined by $\Psi^*$.

On the other hand, Theorem \ref{thm-self-cup-product} gives us an identification:
$$
\Psi^{x_1(\lambda)} \cup \Psi^{x_2(\lambda)} \cong \Psi^*(\Psi(-1)^{x_1(\lambda)x_2(\lambda)}),
$$
so \eqref{eq-pointed-morphism-from-theta-data-calculation} sends a bilinear form $c$ to the pointed map $\lambda\mapsto \Psi(-1)^{c(\lambda, \lambda)}$. This is precisely the restriction of the lower circuit of \eqref{eq-pointed-morphism-from-theta-data} to $\check{\Lambda}^{\otimes 2}$, see the description after \eqref{eq-first-theta-data-integral}. We omit the verification that these identifications agree over $\Gamma^2(\check{\Lambda})\otimes A(-1)$.
\end{proof}

\subsection{$\mathbb E_{\infty}$-monoidal covers}
\label{sect-commutative-covers}

\begin{void}
Let $S$, $A$, and $T = \Lambda\otimes\mathbb G_m$ be as in \S\ref{void-classification-tori-setup}. Viewing sections of $\underline{\Gamma}{}_e(BT, B^4A(1))$ as $\mathbb E_0$-monoidal morphisms $BT \rightarrow B^4A(1)$, we see that it receives forgetful functors:
\begin{align}
	\label{eq-commutative-structure-functors}
	\underline{\Maps}{}_{\mathbb E_{\infty}}(BT, B^4A(1)) \rightarrow \cdots &\rightarrow \underline{\Maps}{}_{\mathbb E_1}(BT, B^4A(1)) \\
	\label{eq-monoidal-structure-functor}
	&\rightarrow \underline{\Gamma}{}_e(BT, B^4A(1)).
\end{align}

According to Theorem \ref{thm-classification-torus} and the triangle \eqref{eq-theta-data-triangle-coefficient}, $\underline{\Gamma}{}_e(BT, B^4A(1))$ admits a canonical functor to $\Sym^2(\check{\Lambda})\otimes A(-1)$; we call the image of $\mu$ the quadratic form $Q$ \emph{associated to $\mu$}. Then $Q$ defines an $A(-1)$-valued symmetric form $b$.
\end{void}

\begin{prop}
\label{prop-commutative-covers}
The following statements hold:
\begin{enumerate}
	\item the functors in \eqref{eq-commutative-structure-functors} are equivalences;
	\item the functor \eqref{eq-monoidal-structure-functor} is fully faithful and its essential image consists of sections $\mu\in\underline{\Gamma}{}_e(BT, B^4A(1))$ whose associated symmetric form $b$ vanishes.
\end{enumerate}
\end{prop}
\begin{proof}
The Bar construction defines an equivalence for all $k\ge 1$:
$$
\underline{\Maps}{}_{\mathbb E_{k-1}}(BT, B^4A(1)) \cong \underline{\Gamma}{}_e(B^{k}T, B^{k+4}A(1)).
$$
For $k = 1$, the simplicial system $[n]\mapsto BT^{\times n}$ yields an equivalence by descent:
\begin{align*}
\underline{\Gamma}{}_e(B^2T, B^5A(1)) &\cong \lim_{[n]} \underline{\Gamma}{}_e(BT^{\times n}, B^5A(1)) \\
& \cong \lim_{[n]} B\underline{\Gamma}{}_e(BT^{\times n}, B^4A(1)).
\end{align*}
The triangle in \eqref{eq-second-theta-data-coefficients} shows that we have a triangle of \emph{complexes}:
\begin{equation}
\label{eq-commutative-cover-fiber-sequence-limit}
\lim_{[n]}(\check{\Lambda}^{\oplus n}\otimes A[3]) \rightarrow \lim_{[n]}(\Gamma_e(BT^{\times n}, B^4A(1))[1]) \rightarrow \lim_{[n]}(\Sym^2(\check{\Lambda}^{\oplus n})\otimes A(-1)[1]),
\end{equation}

Applying the isomorphisms $\lim_{[n]}(\check{\Lambda}^{\oplus n}) \cong \check{\Lambda}[-1]$, $\lim_{[n]}(\Sym^2(\check{\Lambda})^{\oplus n}) \cong \Ant^2(\check{\Lambda})[-2]$ of Lemma \ref{lem-calculation-cosimplicial-limit}, we obtain a cofiber sequence of connective complexes by truncating \eqref{eq-commutative-cover-fiber-sequence-limit} in degrees $\le 0$:
\begin{equation}
\label{eq-commutative-cover-fiber-sequence}
\check{\Lambda}\otimes B^2A \rightarrow \underline{\Gamma}{}_e(B^2T, B^5A(1)) \rightarrow \Tor^1(\check{\Lambda}/2, A(-1)).
\end{equation}
Here, the last term appears through the isomorphism below, as $\wedge^2(\check{\Lambda})$ is torsion-free:
$$
\Tor^1(\check{\Lambda}/2, A(-1)) \cong \Tor^1(\Ant^2(\check{\Lambda}), A(-1)).
$$

Now, the cofiber sequence \eqref{eq-commutative-cover-fiber-sequence} is fixed under the operation $\lim_{[n]}B(-)$ applied to the cosimplicial system $[n]\mapsto\check{\Lambda}^{\oplus n}$. The same descent argument then gives statement (1).

For statement (2), we observe that the canonical functor:
\begin{equation}
\label{eq-commutative-cover-loop-functor}
\underline{\Gamma}{}_e(B^2T, B^5A(1)) \rightarrow \underline{\Gamma}{}_e(BT, B^4A(1))
\end{equation}
comes from the morphism of triangles from \eqref{eq-commutative-cover-fiber-sequence-limit} to the triangle in \eqref{eq-second-theta-data-coefficients} defined by evaluation at the first simplex $[1]$. The isomorphisms exhibited in the proof of Lemma \ref{lem-calculation-cosimplicial-limit} then shows that \eqref{eq-commutative-cover-loop-functor} fits into a map of cofiber sequences:
\begin{equation}
\label{eq-commutative-cover-loop-functor-calculation}
\begin{tikzcd}[column sep = 1em]
	\check{\Lambda}\otimes B^2A \ar[r]\ar[d, "\cong"] & \underline{\Gamma}{}_e(B^2T, B^5A(1)) \ar[r]\ar[d] & \Tor^1(\check{\Lambda}/2, A(-1)) \ar[d] \\
	\check{\Lambda}\otimes B^2A \ar[r] & \underline{\Gamma}{}_e(BT, B^4A(1)) \ar[r] & \Sym^2(\check{\Lambda})\otimes A(-1)
\end{tikzcd}
\end{equation}
where the rightmost arrow comes from the short exact sequence in \eqref{eq-linear-algebra-diagram}:
$$
0\rightarrow \Sym^2(\check{\Lambda}) \rightarrow \Gamma^2(\check{\Lambda}) \rightarrow \check{\Lambda}/2 \rightarrow 0.
$$
(It is deduced from the extension of $\Ant^2(\check{\Lambda})$ by $\Sym^2(\check{\Lambda})$ defined by $\check{\Lambda}^{\otimes 2}$, the nondegenerate cochains in $\Sym^2(\check{\Lambda}^{\oplus 2})$.) In other words, $\Tor^1(\check{\Lambda}/2, A(-1))$ appears in \eqref{eq-commutative-cover-loop-functor-calculation} as the subgroup of $A(-1)$-valued quadratic form $Q$ on $\Lambda$ whose symmetric form vanishes.
\end{proof}

\begin{rem}
\label{rem-commutative-cover-cofiber-sequence}
Using the conclusion of Proposition \ref{prop-commutative-covers}, we may reproduce the cofiber sequence \eqref{eq-commutative-cover-fiber-sequence} as follows:
\begin{equation}
\label{eq-commutative-cover-cofiber-sequence}
\underline{\Maps}{}_{\mathbb Z}(BT, B^4A(1)) \rightarrow \underline{\Maps}{}_{\mathbb E_{\infty}}(BT, B^4A(1)) \rightarrow \underline{\Maps}{}_{\mathbb Z}(\Lambda/2, A(-1)),
\end{equation}
where the second functor takes $\mu$ to its associated quadratic form $Q$, viewed as a $2$-torsion homomorphism $\Lambda \rightarrow A(-1)$.
\end{rem}

\begin{void}
If $\mu : BT \rightarrow B^4A(1)$ is an $\mathbb E_{\infty}$-monoidal morphism, then the corresponding composition \eqref{eq-monoidal-morphism-from-sections} inherits an $\mathbb E_{\infty}$-monoidal structure.

In other words, taking $\underline{\Gamma}{}_e(\mathbb G_m, -)$ defines a functor:
\begin{align}
	\notag
	\underline{\Maps}{}_{\mathbb E_{\infty}}(BT, B^4A(1)) & \cong \underline{\Maps}{}_{\mathbb E_{\infty}}(T, B^3A(1)) \\
	&\rightarrow \underline{\Maps}{}_{\mathbb E_{\infty}}(\Lambda, B^2A).
	\label{eq-commutative-cover-classification}
\end{align}
\end{void}

\begin{void}
\label{void-self-commutativity-functor}
On the other hand, there is a canonical functor:
\begin{equation}
\label{eq-self-commutativity-functor}
\mathrm{inv} : \underline{\Maps}{}_{\mathbb E_{\infty}}(\Lambda, B^2A) \rightarrow \underline{\Maps}{}_{\mathbb Z}(\Lambda/2, A),
\end{equation}
with fiber $\underline{\Maps}{}_{\mathbb Z}(\Lambda, B^2A)$.

To define it, we observe that an $\mathbb E_{\infty}$-monoidal morphism $F : \Lambda\rightarrow B^2(A)$ defines a symmetric monoidal extension of $\Lambda$ by $B(A)$, whose commutativity constraint is an anti-symmetric map $c : \Lambda\otimes\Lambda \rightarrow A$. Then \eqref{eq-self-commutativity-functor} sends $F$ to the homomorphism $\lambda\mapsto c(\lambda,\lambda)$.
\end{void}

\begin{prop}
\label{prop-commutative-cover-classification}
The functor \eqref{eq-commutative-cover-classification} is an equivalence and identifies the cofiber sequence \eqref{eq-commutative-cover-cofiber-sequence} with that defined by \eqref{eq-self-commutativity-functor}:
\begin{equation}
\label{eq-quadratic-form-as-linearity-invariant}
\begin{tikzcd}[column sep = 1em]
	\underline{\Maps}{}_{\mathbb Z}(BT, B^4A(1)) \ar[r]\ar[d, "\cong"] & \underline{\Maps}{}_{\mathbb E_{\infty}}(BT, B^4A(1)) \ar[r, "\mu\mapsto Q"]\ar[d, "\cong"] & \underline{\Maps}{}_{\mathbb Z}(\Lambda/2, A(-1)) \ar[d, "\cong"]\\
	\check{\Lambda}\otimes B^2A \ar[r] & \underline{\Maps}{}_{\mathbb E_{\infty}}(\Lambda, B^2A) \ar[r, "\mathrm{inv}"] & \underline{\Maps}{}_{\mathbb Z}(\Lambda/2, A)
\end{tikzcd}
\end{equation}
Here, the right vertical isomorphism uses the canonical identification between the $2$-torsion subgroup of $A(-1)$ and that of $A$.
\end{prop}
\begin{proof}
The fact that \eqref{eq-commutative-cover-classification} is an equivalence will follow once we construct a morphism of fiber sequences as in \eqref{eq-quadratic-form-as-linearity-invariant}. By construction, the functor \eqref{eq-commutative-cover-classification} carries $\mathbb Z$-linear morphisms $BT\rightarrow B^4A(1)$ into $\check{\Lambda}\otimes B^2A$. Hence it remains to show that the right square in \eqref{eq-quadratic-form-as-linearity-invariant} is commutative.

For this statement, we may first replace the functor denoted ``$\mathrm{inv}$'' with the analogously defined functor (which recovers \eqref{eq-self-commutativity-functor} upon taking $\underline{\Gamma}{}_e(\mathbb G_m, -)$):
\begin{align*}
\mathrm{inv} : \underline{\Maps}{}_{\mathbb E_{\infty}}(T, B^3A(1)) &\rightarrow \underline{\Maps}{}_{\mathbb Z}(T/2, BA(1)) \\
&\cong \underline{\Maps}{}_{\mathbb Z}(\Lambda/2, A).
\end{align*}

By functoriality of the construction, we further reduce to the case $T = \mathbb G_m$ and $A = \{\pm 1\}$. Then we need to check the commutativity of the diagram:
$$
\begin{tikzcd}[column sep = 1em]
	\underline{\Maps}{}_{\mathbb E_{\infty}}(B\mathbb G_m, B^4\{\pm 1\}^{\otimes 2}) \ar[r, "\mu\mapsto Q"]\ar[d] & \underline{\Maps}{}_{\mathbb Z}(\mathbb Z/2, \mathbb Z/2) \ar[d, "\cong"] \\
	\underline{\Maps}{}_{\mathbb E_{\infty}}(\mathbb G_m, B^4\{\pm 1\}^{\otimes 2}) \ar[r, "\mathrm{inv}"] & \underline{\Maps}{}_{\mathbb Z}(\mathbb Z/2, \{\pm 1\})
\end{tikzcd}
$$
Since the right column consists only of copies of $\mathbb Z/2$, this follows from the fact that the left vertical arrow is an equivalence and preserves the fibers of the horizontal maps.
\end{proof}

\subsection{Quadratic structure}
\label{sect-quadratic-structure}

\begin{void}
We have seen in Proposition \ref{prop-commutative-covers} that a pointed morphism $\mu : BT \rightarrow B^4A(1)$ whose symmetric form $b$ vanishes acquires an $\mathbb E_1$ (or equivalently $\mathbb E_{\infty}$)-monoidal structure.

It turns out that in general, $\mu$ has a ``quadratic structure'' whose associated bilinear form is a pairing $BT\times BT \rightarrow B^4A(1)$ determined by $b$. This quadratic structure allows us to calculate the commutator of the $\mathbb E_1$-monoidal morphism $T \rightarrow B^3A(1)$ corresponding to $\mu$ (Corollary \ref{cor-commutator-symmetric-form}). Similar calculations have also been performed in \cite[\S4]{MR1441006}.
\end{void}

\begin{void}
Indeed, a pointed morphism $\mu : BT \rightarrow B^4A(1)$ defines a bi-rigidified morphism (see \S\ref{void-birigidified-morphism}):
\begin{equation}
\label{eq-pointed-morphism-quadratic-expression}
(m^*\mu) \otimes (p_1^*\mu^{\otimes -1})\otimes (p_2^*\mu^{\otimes -1}) : BT\times BT \rightarrow B^4A(1),
\end{equation}
where $m$ denotes the product morphism $BT\times BT\rightarrow BT$.

The following result shows that $\mu$ is automatically a ``quadratic function'' whose associated symmetric form is defined by $b$.
\end{void}

\begin{prop}
\label{prop-pointed-morphism-quadratic}
Let $\mu : BT\rightarrow B^4A(1)$ be a pointed morphism with associated symmetric form $b$. Then the bi-rigidified morphism \eqref{eq-pointed-morphism-quadratic-expression} corresponds to $b$ under Lemma \ref{lem-birigidified-morphism-classification}.
\end{prop}
\begin{proof}
If $\mu$ comes from a $\mathbb Z$-linear morphism $BT\rightarrow B^4A(1)$, then \eqref{eq-pointed-morphism-quadratic-expression} is trivial and $b = 0$. Hence it suffices to determine the morphism \eqref{eq-pointed-morphism-quadratic-expression} when $\mu$ is defined by a bilinear form $c \in \check{\Lambda}^{\otimes 2}\otimes A(-1)$, see \S\ref{void-construction-of-torus-covers}.

Since the statement is compatible with tensor product of the coefficient group $A$, we may assume that $A = \mu_N$ for an integer $N\ge 1$ invertible on $S$. Because we wish to prove the equality of two elements of a discrete groupoid, namely $\check{\Lambda}\otimes\check{\Lambda}/N$, we may consider an integral lift $c = x_1\otimes x_2$ for $x_1, x_2\in\check{\Lambda}$.

The morphism \eqref{eq-pointed-morphism-quadratic-expression} is then given by:
\begin{align*}
e_1, e_2 &\mapsto (\Psi_{x_1(e_1e_2)} \cup \Psi_{x_2(e_1e_2)})\cdot (\Psi_{x_1(e_1)} \cup \Psi_{x_2(e_1)})^{-1}  \cdot (\Psi_{x_1(e_2)} \cup \Psi_{x_2(e_2)})^{-1} \\
&\cong (\Psi_{x_1(e_1)}\cup \Psi_{x_2(e_2)})\cdot (\Psi_{x_1(e_2)}\cup \Psi_{x_2(e_1)}),
\end{align*}
using the linearity of $x_1$, $x_2$, $\Psi$, and the bilinearity of the cup product. Since the cup product of sections of $B^2\mu_N$ is \emph{symmetric}, the last expression is precisely the value of the map $BT\times BT\rightarrow B^4A(1)$ at $(e_1, e_2)$ corresponding to $x_1\otimes x_2 + x_2\otimes x_1$ under Lemma \ref{lem-birigidified-morphism-classification}, i.e.~the symmetrization of $c$.
\end{proof}

\begin{void}
A pointed morphism $\mu : BT \rightarrow B^4A(1)$ is equivalently described by an $\mathbb E_1$-monoidal morphism $T \rightarrow B^3A(1)$ which we denote by the same letter $\mu$.

Such a morphism has a ``commutator'', which is a bi-rigidified morphism:
\begin{equation}
\label{eq-birigidified-morphism-torus}
\com(\mu) : T \times T \rightarrow B^2A(1),
\end{equation}
It is defined as the following loop in the groupoid of bi-rigidified morphisms $T\times T\rightarrow B^3A(1)$ (using the equivalence between the loop space of such bi-rigidified morphisms with bi-rigidified morphisms $T\times T\rightarrow B^2A(1)$):
$$
\begin{tikzcd}[column sep = 1em]
	\mu \circ m \ar[r, "\cong"] & m\circ (\mu\times \mu) \ar[d, "\cong"] \\
	\mu \circ m^{\mathrm{op}} \ar[u, "\cong"] & m^{\mathrm{op}} \circ (\mu\times \mu) \ar[l, "\cong"]
\end{tikzcd}
$$
where $m^{\mathrm{op}}$ denotes the reversed multiplication $a_1,a_2\mapsto a_2a_1$ on $T$ and $B^3A(1)$, the horizontal arrows are defined by the $\mathbb E_1$-monoidal structure of $\mu$, and the vertical maps are defined by the commutative structures $m \cong m^{\mathrm{op}}$.
\end{void}

\begin{rem}
Regarding an $\mathbb E_1$-monoidal morphism $T \rightarrow B^3A(1)$ as an extension of $\mathbb E_1$-monoidal groupoids $B^2A(1) \rightarrow T^{\dagger} \rightarrow T$, the value of $\com(\mu)$ at sections $t_1, t_2\in T$ is identified with the $A(1)$-gerbe of isomorphisms $\tilde t_1\tilde t_2 \cong \tilde t_2\tilde t_1$, for any of their lifts $\tilde t_1, \tilde t_2 \in T^{\dagger}$.
\end{rem}

\begin{cor}
\label{cor-commutator-symmetric-form}
Let $\mu : BT \rightarrow B^4A(1)$ be a pointed morphism with associated symmetric form $b$. Then the commutator \eqref{eq-birigidified-morphism-torus} of its induced $\mathbb E_1$-monoidal morphism $T\rightarrow B^3A(1)$ correponds to $b$ under Lemma \ref{lem-birigidified-morphism-classification}.
\end{cor}
\begin{proof}
Given a bi-rigidified morphism $BT\times BT \rightarrow B^4A(1)$, we may iteratively consider loop spaces along the two factors to obtain a bi-rigidified morphism $T\times T\rightarrow B^2A(1)$.

This operation renders the following diagram commutative, where the horizontal arrows associate to $\mu$ the bi-rigidified morphisms \eqref{eq-pointed-morphism-quadratic-expression} and \eqref{eq-birigidified-morphism-torus}:
$$
\begin{tikzcd}[column sep = 1em]
	\underline{\Gamma}{}_e(BT, B^4A(1)) \ar[r]\ar[d, "\Omega"] & \underline{\Gamma}{}_{e,e}(BT\times BT, B^4A(1)) \ar[d, "\Omega\times\Omega"]\ar[r, "\cong"] & \check{\Lambda}^{\otimes 2}\otimes A(-1)\ar[d, "\cong"] \\
	\underline{\Maps}{}_{\mathbb E_1}(T, B^3A(1))\ar[r] & \underline{\Gamma}{}_{e, e}(T\times T, B^2A(1)) \ar[r, "\cong"] & \check{\Lambda}^{\otimes 2}\otimes A(-1)
\end{tikzcd}
$$
Thus the assertion follows from Proposition \ref{prop-pointed-morphism-quadratic}.
\end{proof}

\medskip

\section{Reductive group schemes}
\label{sec-classification}

The first goal of this section is to complete the classification of \'etale metaplectic covers for a reductive group scheme. We have already treated the case of a torus in the previous section. The new input is a canonical fiber sequence (Proposition \ref{prop-reductive-classification}):
$$
\underline{\Maps}{}_{\mathbb Z}(\pi_1(G), B^2A) \rightarrow \underline{\Gamma}{}_e(BG, B^4A(1)) \rightarrow \Quad(\Lambda, A(-1))_{\st},
$$
linking \'etale metaplectic covers to ``abelianized cohomology'' and ``strict'' (or ``strictly Weyl-invariant'') quadratic forms on its sheaf of cocharacters.

The statement of our classification Theorem \ref{thm-reductive-classification} requires either choosing a Borel subgroup or a maximal torus: it cannot be formulated intrinsically to $G$. However we prove in \S\ref{sect-dependence-on-borel} that after restriction to a sublattice inside $\Lambda$, the dependence goes away. This fact is important for the definition of the metaplectic $\tn L$-group.

In \S\ref{sect-abelian-covers}, we study \'etale metaplectic covers of $G$ which descend to its abelianization. One natural source for such covers is the restriction of certain \'etale metaplectic covers of $G$ to suitable parabolic subgroups.

Sections \S\ref{sect-metaplectic-center}-\ref{sect-adjoint-equivariance} address some technical points about the center $Z_G$ and the adjoint group $G_{\ad}$. They are included for future applications.

In \S\ref{sect-real-place}, we make some remarks about the case $S = \Spec(\mathbb R)$.

\begin{void}
In this section, $S$ denotes a scheme and $A$ denotes a locally constant \'etale sheaf of abelian groups of order invertible on $S$.
\end{void}

\subsection{Classification}

\begin{void}
\label{void-reductive-group-notations}
Suppose that $G\rightarrow S$ is a reductive group scheme. Denote by $G_{\der}\subset G$ its derived subgroup and $G_{\mathrm{sc}} \twoheadrightarrow G_{\der}$ its simply connected cover.

Denote by $T$ the \emph{universal Cartan} of $G$: it is a torus canonically attached to $G$. To define $T$, we first assume that $G$ is split and we set $T := T_1$ to be the quotient of a Borel subgroup $B_1\subset G$ by its unipotent radical. Another Borel subgroup $B_2\subset G$ is \'etale locally conjugate to the previous one by a section of $G/B_2$ \cite[Expos\'e XXII, Corollaire 5.8.3]{SGA3}, which uniquely determines an isomorphism $T_1\cong T_2$. This isomorphism satisfies the cocycle condition for a third Borel subgroup, defining $T$ canonically for split $G$. The general case then follows from \'etale descent.

Write $\Lambda$ for the cocharacter lattice of $T$ and $\check{\Lambda}$ for its dual. Then $\Lambda$ forms part of the locally constant \'etale sheaf of based (reduced) root data defined by $G$:
\begin{equation}
\label{eq-root-data}
(\Delta\subset\Phi\subset\Lambda, \check{\Delta}\subset\check{\Phi}\subset\check{\Lambda}),\quad \Phi\cong\check{\Phi}.
\end{equation}
Indeed, \eqref{eq-root-data} is constructed from \'etale local splittings of $G$, which includes the data of a Killing couple identifying the split maximal torus of $G$ with $T$.

We also write $\Lambda_{\mathrm{sc}}$ for the sheaf of cocharacters of the universal Cartan $T_{\mathrm{sc}}$ of $G_{\mathrm{sc}}$. Then the canonical map $T_{\mathrm{sc}} \rightarrow T$ identifies $\Lambda_{\mathrm{sc}}$ with the span of $\Phi$.

Denote by $W$ the Weyl group of the root datum \eqref{eq-root-data}, viewed as a locally constant \'etale sheaf of finite groups.
\end{void}

\begin{void}
Suppose that $G$ has a Borel subgroup $B\subset G$, with unipotent radical $N\subset B$. Then we may restrict along $B(B)\rightarrow B(G)$ to obtain a morphism:
\begin{align}
\notag
\res_B : \underline{\Gamma}{}_e(BG, B^4A(1)) &\rightarrow \underline{\Gamma}{}_e(B(B), B^4A(1)) \\
&\cong \underline{\Gamma}{}_e(BT, B^4A(1))
\label{eq-restriction-along-borel}
\end{align}
where the isomorphism is due to the fact $B(N)\rightarrow S$ has vanishing higher direct image in \'etale cohomology.
\end{void}

\begin{rem}
The functor $\res_B$ \emph{depends} on the choice of the Borel subgroup. This dependency will be studied in more details in \S\ref{sect-dependence-on-borel} below.
\end{rem}

\begin{void}
\label{void-associated-quadratic-form}
Recall the morphism $\underline{\Gamma}{}_e(BT, B^4A(1)) \rightarrow \Sym^2(\Lambda)\otimes A(-1)$ from Theorem \ref{thm-classification-torus} associating a quadratic form $Q$ to every \'etale metaplecic cover of $T$. Its composition with $\res_B$ \eqref{eq-restriction-along-borel} defines a map:
\begin{equation}
\label{eq-quadratic-form-invariant}
\underline{\Gamma}{}_e(BG, B^4A(1)) \rightarrow \Sym^2(\Lambda)\otimes A(-1),\quad \mu\mapsto Q.
\end{equation}

It is independent of the choice of $B$, since the target $\Sym^2(\Lambda)\otimes A(-1)$ is a sheaf of discrete groupoids and every two choices of Borel subgroups are \'etale locally conjugate and we may appeal to the isomorphism \eqref{eq-borel-comparison-isomorphism} below.

In particular, \eqref{eq-quadratic-form-invariant} is a well-defined morphism for any reductive group scheme $G\rightarrow S$, without assuming the existence of a Borel subgroup. We refer to the image of $\mu$ under \eqref{eq-quadratic-form-invariant} as its \emph{associated} quadratic form.
\end{void}

\begin{void}
Define the subsheaf of \emph{strict} quadratic forms:
$$
\Quad(\Lambda, A(-1))_{\st} \subset \Sym^2(\Lambda)\otimes A(-1)
$$
to consist of sections $Q$ whose associated symmetric forms $b$ satisfy:
\begin{equation}
\label{eq-strict-W-invariance}
b(\alpha, \lambda) = \langle\check{\alpha}, \lambda\rangle Q(\alpha),\quad (\alpha\in\Delta, \lambda\in\Lambda).
\end{equation}

Denote by $\underline{\Gamma}{}_e(BT, B^4A(1))_{\st}$ the full subgroupoid of $\underline{\Gamma}{}_e(BT, B^4A(1))$ whose associated quadratic form $Q$ belongs to $\Quad(\Lambda; A(-1))_{\st}$.
\end{void}

\begin{rem}
\label{rem-strict-quadratic-forms}
The equality \eqref{eq-strict-W-invariance} can be seen as a strengthened ``$W$-invariance'' property. Indeed, we make the following observations:
\begin{enumerate}
	\item \eqref{eq-strict-W-invariance} implies $W$-invariance, as one sees from computing $Q(s_{\alpha}(\lambda))$ for the simple reflection $s_{\alpha}$ associated to $\alpha\in\Delta$. The converse is not true in general, as one sees in the example $G = \SL_2\times\mathbb G_m$ and $A = \{\pm 1\}$.
	
	However, if $A$ has no $2$-torsion, then the converse holds. Indeed, after multiplying \eqref{eq-strict-W-invariance} by $2$, the resulting equality is equivalent to $b(s_{\alpha}(\alpha), s_{\alpha}(\lambda)) = b(\alpha, \lambda)$.
	
	\item if $Q$ belongs to $\Sym^2(\Lambda)^W\otimes A(-1)$, then the equality \eqref{eq-strict-W-invariance} holds, according to the observation (1) for $A = \mathbb Z$. The converse is also not true in general, as one sees in the example $G = \SO_5$ and $A = \{\pm 1\}$.
	
	However, if $G$ is simply connected, i.e.~$\Phi$ spans $\Lambda$, then the converse holds. Indeed, the hypothesis implies that $(\Delta\subset\Phi\subset\Lambda, \check{\Delta}\subset\check{\Phi}\subset\check{\Lambda})$ admits a $W$-invariant decomposition into a direct sum of based irreducible reduced root data, and \eqref{eq-strict-W-invariance} implies that distinct summands are orthogonal under $b$. Hence the problem reduces to a single summand, where we have an isomorphism $\Quad(\Lambda, A(-1))_{\st} \cong A(-1)$ given by evaluating $Q$ at a short coroot.
\end{enumerate}
Furthermore, once the equality \eqref{eq-strict-W-invariance} holds for $\alpha\in\Delta$, it also holds for all $\alpha\in\Phi$, so the sheaf of strict quadratic forms is defined independently of the base of the root data \eqref{eq-root-data}.
\end{rem}

\begin{void}
\label{void-fundamental-group}
Denote by $\pi_1(G) := \Lambda/\Lambda_{\mathrm{sc}}$ the algebraic fundamental group of $G$, viewed as a locally constant \'etale sheaf of abelian groups on $S$. We shall construct a functor:
\begin{equation}
\label{eq-abelian-covers}
	\underline{\Maps}{}_{\mathbb Z}(\pi_1(G), B^2A) \rightarrow \underline{\Gamma}{}_e(BG, B^4A(1)),
\end{equation}
whose targets consists of \'etale metaplectic covers coming from ``abelianized cohomology'' in the sense of Borovoi \cite{MR1401491}.

Recall that the conjugation action of $G$ on $G_{\der}$ induces one on $G_{\mathrm{sc}}$ by functoriality of the simply connected cover, which we continue to call the \emph{conjugation action} of $G$ on $G_{\mathrm{sc}}$. In particular, the quotient stack $G/G_{\mathrm{sc}}$ inherits a monoidal structure.
\end{void}

\begin{rem}
\label{rem-borovoi-abelianization-description}
As a sheaf of monoidal groupoids, objects of $G/G_{\mathrm{sc}}$ are those of $G$ and for $g_1, g_2\in G$, a morphism $g_1\rightarrow g_2$ is a lift of $g_2g_1^{-1}$ to $G_{\mathrm{sc}}$.

The monoidal operation on $G/G_{\mathrm{sc}}$ carries objects $g_1$, $g_2$ to $g_1g_2$. On morphisms, it carries a lift $\tilde g$ of $g_2g_1^{-1}$ and a lift $\tilde g'$ of $g_2'g_1'{}^{-1}$ to the lift $(g_2\tilde g'g_2^{-1})\tilde g$ of $(g_2g_2')(g_1g_1')^{-1}$, where we have used the conjugation action of $G$ on $G_{\mathrm{sc}}$.
\end{rem}

\begin{void}
\label{void-borovoi-abelianization-isomorphism}
Let us construct a canonical isomorphism of monoidal stacks:
\begin{equation}
\label{eq-borovoi-abelianization-isomorphism}
T/T_{\mathrm{sc}}\cong G/G_{\mathrm{sc}}.
\end{equation}
(Note: the left-hand-side is isomorphic to $\pi_1(G)\otimes\mathbb G_m$.)
\end{void}

\begin{proof}[Construction]
This is a variant of \cite[Lemma 3.8.1]{MR1401491}. We first treat the case where $G$ splits. For any Killing couple $T_1\subset B_1\subset G$, we let \eqref{eq-borovoi-abelianization-isomorphism} be the map induced from $T_1\rightarrow G$, in view of the isomorphism $T\cong T_1$. The map $f_{B_1} : T_1/T_{1,\mathrm{sc}} \rightarrow G/G_{\mathrm{sc}}$ is an isomorphism of monoidal stacks, as $T_{1,\mathrm{sc}}$ is the preimage of $T$ along $G_{\mathrm{sc}}\rightarrow G$. Hence, we only need to treat its independence on the Killing couple (as a monoidal morphism).

Let $T_2\subset B_2\subset G$ be another Killing couple. Since Killing couples are parametrized by $G/T_1 \cong G_{\mathrm{sc}}/T_{1,\mathrm{sc}}$, \'etale locally on $S$ we may choose $\tilde g\in G_{\mathrm{sc}}$ which conjugates $T_1\subset B_1$ into $T_2\subset B_2$. There is a commutative square:
$$
\begin{tikzcd}[column sep = 1.5em]
	T_1/T_{1,\mathrm{sc}} \ar[r, "f_{B_1}"]\ar[d, "\alpha_{1,2}"] & G/G_{\mathrm{sc}} \ar[d, "\mathrm{int}_{\tilde g}"] \\
	T_2/T_{2, \mathrm{sc}} \ar[r, "f_{B_2}"] & G/G_{\mathrm{sc}}
\end{tikzcd}
$$
The left vertical arrow is induced from the canonical identifications $T_1\cong T_2$, $T_{1,\mathrm{sc}}\cong T_{2,\mathrm{sc}}$, which depend only on the class of $\tilde g$ in $G_{\mathrm{sc}}/T_{1,\mathrm{sc}}$. We shall construct a $2$-isomorphism between $\mathrm{int}_{\tilde g}$ and the identity endomorphism on $G/G_{\mathrm{sc}}$:
\begin{equation}
\label{eq-borovoi-abelianization-inner-automorphism-trivialization}
\mathrm{int}_{\tilde g} \cong \mathrm{id}_{G/G_{\mathrm{sc}}}
\end{equation}

The value of \eqref{eq-borovoi-abelianization-inner-automorphism-trivialization} at $g\in G$ is defined by the lift $\tilde g(g\tilde g^{-1}g^{-1})$ of $\mathrm{int}_{\tilde g}(g)g^{-1}$ to $G_{\mathrm{sc}}$ using the $G$-action on $G_{\mathrm{sc}}$ by conjugation. This lift satisfies the following properties:
\begin{enumerate}
	\item for $g$ coming from $G_{\mathrm{sc}}$, it agrees with the lift of $\mathrm{int}_{\tilde g}(g)g^{-1}$ given by lifting $g$;
	\item for $g\in T_1$ and $\tilde g\in T_{1,\mathrm{sc}}$, it agrees with the lift of $\mathrm{int}_{\tilde g}(g)g^{-1}$ given by the identity, using the equality $\mathrm{int}_{\tilde g}(g) = g$ in $G$.
	\item for $g_1, g_2\in G$, the lifts $\tilde g_1$ of $\mathrm{int}_{\tilde g}(g_1)g_1^{-1}$ and $\tilde g_2$ of $\mathrm{int}_{\tilde g}(g_2)g_2^{-1}$ satisfy the property that $(\mathrm{int}_{\tilde g}(g_1)\tilde g_2\mathrm{int}_{\tilde g}(g_1)^{-1})\tilde g_1$ agrees with the chosen lift of $\mathrm{int}_{\tilde g}(g_1g_2)(g_1g_2)^{-1}$.
\end{enumerate}

Property (1) shows that the isomorphism $g\cong \mathrm{int}_{\tilde g}(g)$ specified above for $g\in G$ defines a $2$-isomorphism \eqref{eq-borovoi-abelianization-inner-automorphism-trivialization}. Property (2) shows that the induced $2$-isomorphism:
\begin{equation}
\label{eq-borovoi-abelianization-independence}
f_{B_1} \cong f_{B_2} \circ \alpha_{1,2}
\end{equation}
depends only on the class of $\tilde g$ in $G_{\mathrm{sc}}/T_{1,\mathrm{sc}}$. Property (3) shows that \eqref{eq-borovoi-abelianization-independence} is a monoidal isomorphism, in view of the description in Remark \ref{rem-borovoi-abelianization-description}.

It remains to show that for a third Killing couple $T_3\subset B_3\subset G$, the $2$-isomorphisms $f_{B_1} \cong f_{B_2} \circ \alpha_{1,2}$ and $f_{B_2}\cong f_{B_3}\circ \alpha_{2,3}$ constructed above concatenate into the $2$-isomorphism $f_{B_1} \cong f_{B_3}\circ \alpha_{1,3}$. The desired statement being an equality, we may choose $\tilde g_1, \tilde g_2\in G_{\mathrm{sc}}$ which conjugate $T_1\subset B_1$ into $T_2\subset B_2$, respectively $T_2\subset B_2$ into $T_3\subset B_3$. This follows from the equality $\mathrm{int}_{\tilde g_2}\circ\mathrm{int}_{\tilde g_1} = \mathrm{int}_{\tilde g_2\tilde g_1}$ and the fact that \eqref{eq-borovoi-abelianization-inner-automorphism-trivialization} is compatible with compositions. The latter compatibility amounts to the equality:
$$
(\tilde g_2[\tilde g_1, g]\tilde g_2^{-1})[\tilde g_2, g] = [\tilde g_2\tilde g_1, g]
$$
for each $g\in G$, where $[-,-]$ is a shorthand for the commutator.

Having constructed the canonical isomorphism $T/T_{\mathrm{sc}} \cong G/G_{\mathrm{sc}}$ for a split reductive group $G$, the general case follows from \'etale descent.
\end{proof}

\begin{void}
The isomorphism $\check{\Lambda}\otimes B^2A \cong \underline{\Maps}{}_{\mathbb Z}(BT, B^4A(1))$ of Remark \ref{rem-tori-classification-linear-part} and its analogue for $\Lambda_{\mathrm{sc}}$ then define \eqref{eq-abelian-covers} as the following composition:
\begin{align*}
	\underline{\Maps}{}_{\mathbb Z}(\pi_1(G), B^2A) & \cong \underline{\Maps}{}_{\mathbb Z}(B(T/T_{\mathrm{sc}}), B^4A(1)) \\
	& \rightarrow \underline{\Gamma}{}_e(B(T/T_{\mathrm{sc}}), B^4A(1)) \\
	& \cong \underline{\Gamma}{}_e(B(G/G_{\mathrm{sc}}), B^4A(1)) \rightarrow \Gamma_e(BG, B^4A(1)),
\end{align*}
where the second isomorphism appeals to \eqref{eq-borovoi-abelianization-isomorphism}.

The following result is a consequence of the calculation of $H^4(BG)$ with torsion coefficients, which is supposedly well known.
\end{void}

\begin{prop}
\label{prop-reductive-classification}
The morphisms \eqref{eq-abelian-covers} and \eqref{eq-quadratic-form-invariant} determine a cofiber sequence:
\begin{equation}
\label{eq-reductive-classification-cofiber-sequence}
\underline{\Maps}{}_{\mathbb Z}(\pi_1(G), B^2A) \rightarrow \underline{\Gamma}{}_e(BG, B^4A(1)) \rightarrow \Quad(\Lambda, A(-1))_{\st}.
\end{equation}
\end{prop}
\begin{proof}
We first observe that the composition is indeed the zero map. Thus it remains to show that \eqref{eq-reductive-classification-cofiber-sequence} induces a long exact sequence on cohomology groups. Since the \'etale cohomology of $BG \rightarrow S$ commutes with arbitrary base change (\cite[Rappel 1.5.1]{MR1441006}), we may replace $S$ with the spectrum of $\mathbb C$.

Furthermore, we may replace $A$ by the constant sheaf of abelian groups $\frac{1}{N}\mathbb Z/\mathbb Z$, and using the short exact sequence:
$$
0 \rightarrow \frac{1}{N}\mathbb Z/\mathbb Z \rightarrow \mathbb Q/\mathbb Z \xrightarrow{\cdot N} \mathbb Q/\mathbb Z \rightarrow 0,
$$
the assertion follows from the analogous one for the coefficient group $\mathbb Q/\mathbb Z$. For a \emph{divisible} abelian group $A$ (e.g.~$\mathbb Q/\mathbb Z$), we have the calculation:
$$
H^i(BG, A(1)) \cong
\begin{cases}
	0 & i = 1;\\
	\Hom(\pi_1(G), A) & i = 2;\\
	\Ext^1(\pi_1(G), A) & i = 3;\\
	\Quad(\Lambda, A(-1))_{\st} & i = 4,
\end{cases}
$$
performed in the \emph{corrected version} of \cite[Appendix B]{MR3769731}. Note that these isomorphisms of \emph{op.cit.}~are defined by pulling back to the classifying stack of a chosen Borel subgroup, which agree with the maps induced from \eqref{eq-abelian-covers} and \eqref{eq-quadratic-form-invariant}.
\end{proof}

\begin{rem}
If $G$ is simply connected, Proposition \ref{prop-reductive-classification} asserts that $\underline{\Gamma}{}_e(BG, B^4A(1))$ is canonically equivalent to $\Quad(\Lambda, A(-1))_{\st}$, or $\Sym^2(\check{\Lambda})^W\otimes A(-1)$ according to Remark \ref{rem-strict-quadratic-forms}(2). In particular, $\underline{\Gamma}{}_e(BG, B^4A(1))$ is discrete for simply connected $G$.

On the other hand, if $G$ is a torus, then the cofiber sequence \eqref{eq-reductive-classification-cofiber-sequence} is identified with the cofiber sequence in \eqref{eq-second-theta-data-coefficients}.
\end{rem}

\begin{thm}
\label{thm-reductive-classification}
Suppose that $G\rightarrow S$ is a reductive group scheme equipped with a Borel subgroup $B\subset G$. Then \eqref{eq-restriction-along-borel} gives rise to a Cartesian diagram:
\begin{equation}
\label{eq-reductive-classification-square}
\begin{tikzcd}[column sep = 1.5em]
	\underline{\Gamma}{}_e(BG, B^4A(1)) \ar[r, "\res_B"]\ar[d] & \underline{\Gamma}{}_e(BT, B^4A(1))_{\st} \ar[d] \\
	\underline{\Gamma}{}_e(BG_{\mathrm{sc}}, B^4A(1)) \ar[r, "\res_{B_{\mathrm{sc}}}"] & \underline{\Gamma}{}_e(BT_{\mathrm{sc}}, B^4A(1))_{\st}
\end{tikzcd}
\end{equation}
where $B_{\mathrm{sc}}$ denotes the induced Borel subgroup of $G_{\mathrm{sc}}$.
\end{thm}
\begin{proof}
The square \eqref{eq-reductive-classification-square} is commutative by functoriality of the construction of $\res_B$. To show that it is Cartesian, it suffices to prove that the horizontal morphisms induce an isomorphism on fibers.

We apply Proposition \ref{prop-reductive-classification} to all four terms in \eqref{eq-reductive-classification-square}, so these fibers fit into a map of fiber sequences:
$$
\begin{tikzcd}[column sep = 1em]
	\Fib(\res_B) \ar[r]\ar[d] & \underline{\Maps}{}_{\mathbb Z}(\pi_1(G), B^2A) \ar[r]\ar[d] & \check{\Lambda}\otimes B^2A\ar[d] \\
	\Fib(\res_{B_{\mathrm{sc}}}) \ar[r] & 0 \ar[r] & \check{\Lambda}_{\mathrm{sc}}\otimes B^2A
\end{tikzcd}
$$
The right square is Cartesian by the definition of $\pi_1(G)$, implying that the leftmost vertical arrow is an isomorphism.
\end{proof}

\begin{rem}
\label{rem-reductive-classification-explicit}
Theorem \ref{thm-reductive-classification}, combined with Theorem \ref{thm-classification-torus}, gives a complete classification of \'etale metaplectic covers of reductive group schemes $G$ equipped with a Borel subgroup $B$.

Namely, it shows that any pointed morphism $\mu : BG \rightarrow B^4A(1)$ is uniquely determined by quadruple $(Q, F, h, \varphi)$ where:
\begin{enumerate}
	\item $Q$ is a strict $A(-1)$-valued quadratic form on $\Lambda$;
	\item $F : H^{(2)}(\Lambda) \rightarrow B^2A$ is a $\mathbb Z$-linear morphism;
	\item $h$ is an isomorphism between the restriction of $F$ to $\wedge^2(\Lambda)$ and the map determined by $Q$, see \S\ref{void-theta-data-explicit-description};
	\item $\varphi$ is an isomorphism between the restriction of $(F, h)$ to $\Lambda_{\mathrm{sc}}$ and the pair defined by the restriction of $Q$ to $\Lambda_{\mathrm{sc}}$ under $\res_{B_{\mathrm{sc}}}$.
\end{enumerate}
\end{rem}

\begin{rem}
If, instead of choosing a Borel subgroup $B$, we fix a maximal torus $T_1\subset G$. Then there is a Cartesian diagram analogous to \eqref{eq-reductive-classification-square}, where the horizontal morphisms are replaced by restrictions along $B(T_1)\rightarrow BG$, respectively $B(T_{1,\mathrm{sc}})\rightarrow BG$, for $T_{1,\mathrm{sc}}\subset G_{\mathrm{sc}}$ being the induced maximal torus. In view of the explicit description in Remark \ref{rem-reductive-classification-explicit}, this statement gives an analogue of \cite[Theorem 7.2]{MR1896177} in \'etale cohomology.
\end{rem}

\begin{void}
Recall from Remark \ref{rem-splitting-theta-data} that the inclusion $B^2A \rightarrow \underline{\Gamma}{}_e(B\mathbb G_m, B^4A(1))$ has a canonical splitting.

If $G$ is split, then its sheaf of based root data is constant. Hence we may view each $\alpha\in\Delta$ as a homomorphism $\mathbb G_m\rightarrow T$ defined globally over $S$. Pulling back along it and composing with the above splitting defines a functor:
\begin{equation}
\label{eq-pullback-along-simple-coroot}
	\bigoplus_{\alpha\in\Delta}\alpha^* : \underline{\Gamma}{}_e(BT, B^4A(1))_{\st} \rightarrow \bigoplus_{\alpha\in\Delta} A[2].
\end{equation}

When $G$ is equipped with a pinning, the classification Theorem \ref{thm-reductive-classification} can be reformulated in much more explicit terms.
\end{void}

\begin{cor}
\label{cor-reductive-classification-simple}
Suppose that $G\rightarrow S$ is a pinned split reductive group scheme. Then $\underline{\Gamma}{}_e(BG, B^4A(1))$ is canonically identified with the fiber of \eqref{eq-pullback-along-simple-coroot}.
\end{cor}

\begin{void}
In order to prove Corollary \ref{cor-reductive-classification-simple}, we need an observation about $\SL_2$ and its diagonal maximal torus $\mathbb G_m\subset \SL_2$, $a\mapsto \mathrm{diag}(a, a^{-1})$.

Being simply connected, $\underline{\Gamma}{}_e(B(\SL_2), B^4A(1))$ is equivalent to $A(-1)$, by evaluating the corresponding quadratic form on a coroot. Consider the \'etale metaplectic cover $\mu_a$ corresponding to $a\in A(-1)$.
\end{void}

\begin{lem}
\label{lem-restriction-sl2}
The restriction of $\mu_a$ along $\mathbb G_m\subset\SL_2$ is naturally identified with the image of $a$ under the map \tn{(1)} of \eqref{eq-second-theta-data-coefficients}.
\end{lem}
\begin{proof}
Indeed, there are equivalences:
\begin{align*}
	\underline{\Gamma}{}_e(B(\SL_2), B^4A(1)) & \cong \underline{\Gamma}{}_e(\SL_2, B^3A(1)) \\
	& \cong \underline{\Gamma}{}_e(\mathbb A^2\backslash\{0\}, B^3A(1)),
\end{align*}
where the second map is induced from the map $\SL_2\rightarrow\mathbb A^2\backslash\{0\}$ by acting on $(x,y) = (1, 0)$, so $\mathbb A^1\backslash\{0\}$ is pointed by $e = (1, 0)$. The \'etale metaplectic cover $\mu_a$ is thus defined by gluing the trivial sections of $B^3A(1)$ on the charts $x\neq 0$ and $y\neq 0$ along the section of $B^2A(1)$ over $\mathbb G_m\times\mathbb G_m$ defined by $a$.

The restriction of this section of $B^3A(1)$ along $\mathbb G_m\subset \SL_2\rightarrow \mathbb A^2\backslash\{0\}$ admits a canonical trivialization, since the image belongs to the chart $x\neq 0$.

Note that $\mathbb E_1$-monoidal morphisms $\mathbb G_m\rightarrow B^3A(1)$ equipped with a trivialization of the underlying pointed morphism are in bijection with $A(-1)$: the maps in two directions are given by (1) of \eqref{eq-second-theta-data-coefficients}, respectively taking cocycle which defines a bi-rigidified morphism $\mathbb G_m\times\mathbb G_m \rightarrow B^2A(1)$, or equivalently an element of $A(-1)$ (Proposition \ref{prop-torus-cocycle-covers}). The fact that the cocycle defined by $a\in A(-1)$ has quadratic form $a$ shows that the cocycle corresponding to the restriction of $\mu_a$ must be $a$.
\end{proof}

\begin{rem}
The proof of Lemma \ref{lem-restriction-sl2} also shows that the section of $B^3A(1)$ over $\SL_2$ underlying $\mu_a$ is nontrivial, i.e.~the induced $\mathbb E_1$-monoidal extension of $\SL_2$ by $B^2A(1)$ is not ``defined by a cocycle.'' This stands in contrast with the corresponding central extension of $\SL_2(F)$ over a local field $F$, which can be defined by a cocycle, see \cite{MR204422}.
\end{rem}

\begin{proof}[Proof of Corollary \ref{cor-reductive-classification-simple}]
In view of the Cartesian diagram \eqref{eq-reductive-classification-square}, it suffices to argue that the row below forms a fiber sequence:
$$
\begin{tikzcd}[column sep = 1.5em]
	& \check{\Lambda}_{\mathrm{sc}}\otimes B^2A \ar[d] \\
	\underline{\Gamma}{}_e(BG_{\mathrm{sc}}, B^4A(1)) \ar[r, "\res_{B_{\mathrm{sc}}}"] & \underline{\Gamma}{}_e(BT_{\mathrm{sc}}, B^4A(1))_{\st} \ar[r, "\bigoplus_{\alpha\in\Delta}\alpha^*"]\ar[d] & \bigoplus_{\alpha\in\Delta} B^2A.\\
	& \Quad(\Lambda_{\mathrm{sc}}, A(-1))_{\st}
\end{tikzcd}
$$

Since we already know that the column forms a fiber sequence (Theorem \ref{thm-classification-torus}) and the compositions given equivalences:
\begin{align*}
\check{\Lambda}_{\mathrm{sc}}\otimes B^2A &\cong \bigoplus_{\alpha\in\Delta}B^2A\\
\underline{\Gamma}{}_e(BG_{\mathrm{sc}}, B^4A(1)) &\cong \Quad(\Lambda_{\mathrm{sc}}, A(-1))_{\st}
\end{align*}
it only remains to construct a null-homotopy of the composition $(\bigoplus_{\alpha\in\Delta}\alpha^*)\circ \res_{B_{\mathrm{sc}}}$.

For this statement, we may work with an individual $\alpha$. Let $N_{\alpha}\subset G_{\mathrm{sc}}$ denote the root subgroup associated to $\alpha$. There exists a map $f_{\alpha} : \SL_2\rightarrow G_{\mathrm{sc}}$ inducing $\alpha$ on the maximal tori and sending the unipotent radical $N^+$ of the upper triangular Borel subgroup $B^+\subset \SL_2$ to $N_{\alpha}$. Furthermore, requiring $f_{\alpha}$ to match the standard pinning $N^+\cong \mathbb G_a$ with the given one on $N_{\alpha}$ determines it uniquely.

Therefore, we may identify $\alpha^* \circ \res_{B_{\mathrm{sc}}}$ with the restriction along:
$$
B(\mathbb G_m) \rightarrow B(\SL_2) \xrightarrow{f_{\alpha}} BG_{\mathrm{sc}},
$$
so the desired null-homotopy follows from Lemma \ref{lem-restriction-sl2}.
\end{proof}

\subsection{Dependence on $B$}
\label{sect-dependence-on-borel}

\begin{void}
\label{void-borel-comparison-isomorphism}
Suppose that $G\rightarrow S$ is a reductive group scheme. We keep the notations of \S\ref{void-reductive-group-notations} associated to $G$.

If $G$ is equipped with a Borel subgroup $B_1\subset G$, then we have a functor $\res_{B_1}$ as in \eqref{eq-restriction-along-borel}. If $B_2\subset G$ is another Borel subgroup and $g\in G$ is a section such that $gB_1g^{-1} = B_2$, then we shall construct an isomorphism of functors from $\underline{\Gamma}{}_e(BG, B^4A(1))$ to $\underline{\Gamma}{}_e(BT, B^4A(1))$:
\begin{equation}
\label{eq-borel-comparison-isomorphism}
F_g : \res_{B_1} \cong \res_{B_2}.
\end{equation}
\end{void}

\begin{proof}[Construction]
Inner automorphism by $g$ induces a commutative diagram:
$$
\begin{tikzcd}[column sep = 1.5em]
	G\ar[d, "\mathrm{int}_g"] & B_1 \ar[l]\ar[r]\ar[d, "\mathrm{int}_g"] & T_1 \ar[d, "\mathrm{int}_g"]\ar[r, "\cong"] & T \ar[d, "\cong"]\\
	G & B_2 \ar[l]\ar[r] & T_2 \ar[r, "\cong"] & T
\end{tikzcd}
$$
where $T_1$, $T_2$ denotes the maximal quotient tori of $B_1$, $B_2$. Hence it suffices to construct an isomorphism between $(\mathrm{int}_g)^*$ and the identity endofunctor on $\underline{\Gamma}{}_e(BG, B^4A(1))$.

Recall that a pointed morphism $\mu : BG \rightarrow B^4A(1)$ is equivalent to an $\mathbb E_1$-monoidal morphism $G \rightarrow B^3A(1)$. Denoting the latter by the same letter $\mu$, we have a commutative diagram of $\mathbb E_1$-monoidal stacks:
$$
\begin{tikzcd}[column sep = 1.5em]
	G \ar[d, "\mathrm{int}_g"] \ar[r, "\mu"] & B^3A(1) \ar[d, "\mathrm{int}_{\mu(g)}"] \\
	G \ar[r, "\mu"] & B^3A(1)
\end{tikzcd}
$$

Since the $\mathbb E_1$-monoidal structure on $B^3A(1)$ lifts to an $\mathbb E_{\infty}$-monoidal one, the automorphism $\mathrm{int}_{\mu(g)}$ is trivialized, giving the desired isomorphism $(\mathrm{int}_g)^*(\mu) \cong \mu$.
\end{proof}

\begin{rem}
\label{rem-borel-comparison-isomorphism-composition}
The isomorphism \eqref{eq-borel-comparison-isomorphism} is compatible with compositions in the following sense: given a third Borel subgroup $B_3\subset G$ with $g'\in G$ such that $g'B_2(g')^{-1} = B_3$, there is a canonical isomorphism between $F_{g'}\circ F_g$ and $F_{g'g}$. For a fourth Borel subgroup $B_4\subset G$ with $g''\in G$ such that $g'' B_3(g'')^{-1} = B_4$, a cocycle condition is satisfied.
\end{rem}

\begin{void}
Suppose that $B = B_1 = B_2$ and $g \in B$. The isomorphism $F_g$ \eqref{eq-borel-comparison-isomorphism} is in general not the identity automorphism of $\res_B$.

For $\mu \in \underline{\Gamma}{}_e(BG, B^4A(1))$, the association $g \mapsto F_g(\mu)$ defines a pointed morphism from $B$ to the sheaf of automorphisms of $\res_B(\mu)$. By Theorem \ref{thm-classification-torus}, the latter is canonically identified with $\check{\Lambda}\otimes A[1]$, so we obtain a pointed morphism:
\begin{equation}
\label{eq-borel-equivariance-morphism}
B \rightarrow \check{\Lambda}\otimes A[1].
\end{equation}

By the vanishing of \'etale cohomology of $N\rightarrow S$ and the calculation of \'etale cohomology of $T\rightarrow S$, pointed morphisms $B\rightarrow\check{\Lambda}\otimes A[1]$ are parametrized by $\check{\Lambda}^{\otimes 2}\otimes A(-1)$ (see the proof of Lemma \ref{lem-birigidified-morphism-classification}). In particular, associating \eqref{eq-borel-equivariance-morphism} to $\mu$ defines a map:
\begin{equation}
\label{eq-borel-equivariance-obstruction}
\underline{\Gamma}{}_e(BG, B^4A(1)) \rightarrow \check{\Lambda}^{\otimes 2}\otimes A(-1).
\end{equation}

It may be viewed as the obstruction of the isomorphism $F_g$ to depend \emph{only} on the Borel subgroups $B_1$, $B_2$, as opposed to the section $g\in G$.
\end{void}

\begin{lem}
\label{lem-borel-equivariance-bostruction}
The image of $\mu$ under \eqref{eq-borel-equivariance-obstruction} is given by the symmetric form $b$ associated to its quadratic form $Q$.
\end{lem}
\begin{proof}
Since the target of \eqref{eq-borel-equivariance-obstruction} is discrete, we may work \'etale locally on $S$ and choose a splitting of $B_1\rightarrow T_1$, so $T$ is identified with a subgroup of $G$. The functor $\res_B$ is canonically identified with $\res_T$: restriction along $BT \rightarrow BG$. The pointed morphism \eqref{eq-borel-equivariance-morphism} may also be calculated after pre-composing with $T\rightarrow B$.

For a section $t\in T$, the automorphism $F_t(\mu)$ of $\res_T(\mu)$ is given by the following loop in  $\underline{\Gamma}{}_e(BT, B^4A(1))$, or equivalently in $\underline{\Maps}{}_{\mathbb E_1}(T, B^3A(1))$:
\begin{equation}
\label{eq-borel-equivariance-obstruction-loop}
\begin{tikzcd}[column sep = 1.5em]
	\res_T(\mathrm{int}_t^*\mu) \ar[r, "\cong"] & \res_T(\mu) \ar[d, "\cong"] \\
	\mathrm{int}_t^*\res_T(\mu) \ar[u, "\cong"] & \res_T(\mu) \ar[l, "\cong"]
\end{tikzcd}
\end{equation}
where the top isomorphism comes from identification $\mathrm{int}_t^*(\mu)\cong\mu$ in $\underline{\Gamma}{}_e(BG, B^4A(1))$ constructed in \S\ref{void-borel-comparison-isomorphism} and the bottom isomorphism comes from the commutativity of $T$.

The loop \eqref{eq-borel-equivariance-obstruction-loop}, evaluated at a section $t_1\in T$, is precisely the commutator of $\res_T(\mu) : T \rightarrow B^3A(1)$ evaluated at the sections $t, t_1\in T$, so the claim follows from Corollary \ref{cor-commutator-symmetric-form}.
\end{proof}

\begin{void}
Suppose that $G\rightarrow S$ is a reductive group scheme (without assuming the existence of a Borel subgroup). For a strict quadratic form $Q \in \Quad(\Lambda, A(-1))_{\st}$, we denote by:
$$
\underline{\Gamma}{}_e(BG, B^4A(1))_Q \subset \underline{\Gamma}{}_e(BG, B^4A(1)).
$$
the full subgroupoid of \'etale metaplectic covers of $G$ with associated quadratic form $Q$ (in the sense of \S\ref{void-associated-quadratic-form}).

On the other hand, we may let $\Lambda^{\sharp}\subset\Lambda$ denote the kernel of the symmetric form $b$ associated to $Q$. It defines a torus $T^{\sharp} := \Lambda^{\sharp}\otimes\mathbb G_m$ equipped with an isogeny to $T$.
\end{void}

\begin{void}
\label{void-canonical-morphism-to-sharp-torus}
We shall construct a canonical morphism:
\begin{equation}
\label{eq-canonical-morphism-to-sharp-torus}
\res_{T^{\sharp}} : \underline{\Gamma}{}_e(BG, B^4A(1))_Q \rightarrow \underline{\Maps}{}_{\mathbb E_{\infty}}(BT^{\sharp}, B^4A(1)).
\end{equation}
\end{void}

\begin{proof}[Construction]
We first work \'etale locally on $S$ and assumes the existence of a Borel subgroup $B_1\subset G$. In this case, we define $\res_{T^{\sharp}}$ as the composition of $\res_{B_1}$ and the restriction along $B(T^{\sharp})\rightarrow BT$:
\begin{align}
\notag
	\underline{\Gamma}{}_e(BG, B^4A(1))_Q &\subset \underline{\Gamma}{}_e(BG, B^4A(1)) \\
\label{eq-restriction-to-sharp-torus-borel}
	&\xrightarrow{\res_{B_1}} \underline{\Gamma}{}_e(BT, B^4A(1)) \rightarrow \underline{\Gamma}{}_e(BT^{\sharp}, B^4A(1)).
\end{align}
Since the image of this composition has vanishing symmetric form, it belongs to the full subgroupoid of $\mathbb E_{\infty}$-monoidal morphisms $BT^{\sharp} \rightarrow B^4A(1)$ by Proposition \ref{prop-commutative-covers}.

The key point to address is the independence of the Borel subgroup. For two Borel subgroups $B_1, B_2\subset G$ and a section $g\in G$ such that $gB_1g^{-1} = B_2$, we have the isomorphism:
\begin{equation}
\label{eq-borel-comparison-isomorphism-sharp-cover}
F_g : \res_{B_1} \cong \res_{B_2}
\end{equation}
from \S\ref{void-borel-comparison-isomorphism}. We claim that the composition of $F_g$ with the restriction along $BT^{\sharp}\rightarrow BT$ depends only on $B_1$, $B_2$ (and not on the section $g\in G$).

Indeed, any other section $g'\in G$ such that $g'B_1(g')^{-1} = B_2$ differs from $g$ by an element of $B_1$. By the isomorphism $F_{gg'}\cong F_g\circ F_{g'}$ (see Remark \ref{rem-borel-comparison-isomorphism-composition}), it suffices to prove that for $B = B_1 = B_2$ and $g\in B$, the composition of $F_g$ with the restriction along $BT^{\sharp}\rightarrow BT$ is the identity. For this statement, we recall that as an automorphism of $\res_B(\mu)$, the section:
$$
F_g(\mu) \in \check{\Lambda}\otimes A[1]
$$
is the image of $g\in B$ under \eqref{eq-borel-equivariance-morphism}, which is in turn given by the symmetric form associated to $Q$ (Lemma \ref{lem-borel-equivariance-bostruction}). This section restricts to zero in $\check{\Lambda}^{\sharp}\otimes A[1]$, as desired.

Therefore, for two Borel subgroups $B_1, B_2\subset G$, we have constructed a canonical identifications between the compositions \eqref{eq-restriction-to-sharp-torus-borel} defined by $B_1$, $B_2$. The coherence data for three Borel subgroups and the cocycle condition for a fourth one follow from the corresponding facts about the isomorphism \eqref{eq-borel-comparison-isomorphism-sharp-cover}, see Remark \ref{rem-borel-comparison-isomorphism-composition}.
\end{proof}

\subsection{Covers of $\pi_1(G)\otimes\mathbb G_m$}
\label{sect-abelian-covers}

\begin{void}
Let $G\rightarrow S$ be a reductive group scheme. We keep the notations of \S\ref{void-reductive-group-notations} and recall the canonical morphism of monoidal stacks constructed in \S\ref{void-borovoi-abelianization-isomorphism}:
\begin{equation}
\label{eq-abelianization-map}
	G \rightarrow \pi_1(G)\otimes\mathbb G_m \cong T/T_{\mathrm{sc}}.
\end{equation}

It is surjective in the \'etale topology, with fiber canonically identified with $G_{\mathrm{sc}}$. If $G_{\der}$ is simply connected, then $\pi_1(G)\otimes\mathbb G_m$ is identified with the maximal quotient torus (i.e.~coradical) of $G$.

The following result, combined with Proposition \ref{prop-reductive-classification} for $G_{\mathrm{sc}}$, identifies ``\'etale metaplectic covers'' of $T/T_{\mathrm{sc}}$, i.e.~rigidified sections of $B^4A(1)$ over $B(T/T_{\mathrm{sc}})$, as those of $G$ characterized by the vanishing of the associated quadratic form on $\Lambda_{\mathrm{sc}}$.
\end{void}

\begin{prop}
\label{prop-abelian-covers}
Pulling back along \eqref{eq-abelianization-map} defines a fiber sequence:
\begin{equation}
\label{eq-abelian-covers-descent-sequence}
\underline{\Gamma}{}_e(B(T/T_{\mathrm{sc}}), B^4A(1)) \rightarrow \underline{\Gamma}{}_e(BG, B^4A(1)) \rightarrow \underline{\Gamma}{}_e(BG_{\mathrm{sc}}, B^4A(1)).
\end{equation}
\end{prop}
\begin{proof}
The last term $\underline{\Gamma}{}_e(BG_{\mathrm{sc}}, B^4A(1))$ is discrete and the composition is evidently zero. Thus the construction of \eqref{eq-abelian-covers-descent-sequence} requires no additional data and can be treated \'etale local on $S$.

By Theorem \ref{thm-reductive-classification}, it suffices to identify $\underline{\Gamma}{}_e(B(T/T_{\mathrm{sc}}), B^4A(1))$ as the fiber of:
\begin{equation}
\label{eq-restriction-to-simply-connected-torus-strict}
\underline{\Gamma}{}_e(BT, B^4A(1))_{\st} \rightarrow \underline{\Gamma}{}_e(BT_{\mathrm{sc}}, B^4A(1))_{\st},
\end{equation}
under the evident map. Since a strict quadratic form on $\Lambda$ which vanishes on $\Lambda_{\mathrm{sc}}$ factors through $\pi_1(G)$, the fiber of \eqref{eq-restriction-to-simply-connected-torus-strict} is also the full subgroupoid of the fiber of:
\begin{equation}
\label{eq-restriction-to-simply-connected-torus}
\underline{\Gamma}{}_e(BT, B^4A(1)) \rightarrow \underline{\Gamma}{}_e(BT_{\mathrm{sc}}, B^4A(1)),
\end{equation}
characterized by the property of $Q$ factoring through $\pi_1(G)$.

Since $B(T/T_{\mathrm{sc}})$ is the quotient of $BT$ by $BT_{\mathrm{sc}}$, we may identify pointed sections of $B^4A(1)$ over $B(T/T_{\mathrm{sc}})$ as a cosimplicial limit:
\begin{equation}
\label{eq-abelian-cover-cofiber-sequence}
\begin{tikzcd}[column sep = 0.5em]
	& \lim_{[n]} (\check{\Lambda}\oplus\check{\Lambda}_{\mathrm{sc}}^{\oplus n})\otimes B^2A\ar[d] \ar[r, phantom, "\cong"] & \underline{\Maps}{}_{\mathbb Z}(\pi_1(G), B^2A) \\
	\underline{\Gamma}{}_e(B(T/T_{\mathrm{sc}}), B^4A(1)) \ar[r, phantom, "\cong"] & \lim_{[n]} \underline{\Gamma}{}_e(BT\times BT_{\mathrm{sc}}^{\times n}, B^4A(1))\ar[d] \\
	& \lim_{[n]}\Sym^2(\check{\Lambda}\oplus\check{\Lambda}_{\mathrm{sc}}^{\oplus n})\otimes A(-1) \ar[r, phantom, "\cong"] & \Quad(\pi_1(G), A(-1))
\end{tikzcd}
\end{equation}
where the vertical triangle is that of \eqref{eq-second-theta-data-coefficients}. The identification of the top limit follows from the fact that the complex $[\check{\Lambda} \xrightarrow{d} \check{\Lambda}\oplus\check{\Lambda}_{\mathrm{sc}}\xrightarrow{d}\cdots]$ is equivalent to its subcomplex $[\check{\Lambda}\rightarrow \check{\Lambda}_{\mathrm{sc}}]$ of nondegenerate cochains, which is in turn isomorphic to the $\mathbb Z$-linear dual of $\pi_1(G)$. The identification of the bottom limit follows the fact that a quadratic form on $\Lambda$ whose two pullbacks to $\Lambda\oplus\Lambda_{\mathrm{sc}}$ is precisely one that factors through $\pi_1(G)$.

The vertical sequence in \eqref{eq-abelian-cover-cofiber-sequence} is a cofiber sequence because $\underline{\Maps}{}_{\mathbb Z}(\pi_1G, B^2A)$ has vanishing $H^1$. Comparing it with the canonical map from $\underline{\Gamma}{}_e(B(T/T_{\mathrm{sc}}), B^4A(1))$ to the fiber of \eqref{eq-restriction-to-simply-connected-torus} yields the desired conclusion.
\end{proof}

\begin{rem}
The proof of Proposition \ref{prop-abelian-covers} yields another statement. For any sheaf of finitely generated abelian groups $\Gamma$, there is a cofiber sequence:
\begin{equation}
\label{eq-general-abelian-group-cofiber-sequence}
\underline{\Maps}{}_{\mathbb Z}(\Gamma, B^2A) \rightarrow \underline{\Gamma}{}_e(B(\Gamma\otimes\mathbb G_m), B^4A(1)) \rightarrow \Quad(\Gamma, A(-1)).
\end{equation}
The cofiber sequence in \eqref{eq-second-theta-data-coefficients} is recovered as the special case where $\Gamma$ is torsion-free.

For $\Gamma = \pi_1(G)$, we note that \eqref{eq-abelian-covers-descent-sequence} and \eqref{eq-general-abelian-group-cofiber-sequence} combine into the cofiber sequence \eqref{eq-reductive-classification-cofiber-sequence}, which can be informally summarized as follows: the \'etale metaplectic covers with vanishing quadratic form on $\Lambda$ not only desends to $\pi_1(G)\otimes\mathbb G_m$, but is also of ``abelian nature'' there.
\end{rem}

\begin{void}
Suppose that $P\subset G$ is a parabolic subgroup with unipotent radical $N_P$ and Levi quotient $M$. Since $B(N_P)\rightarrow S$ has vanishing \'etale cohomology in degrees $\ge 1$, \'etale metaplectic covers of $P$ are equivalent to those of $M$. Thus, restriction along $P$ defines a functor:
$$
\underline{\Gamma}{}_e(BG, B^4A(1)) \rightarrow \underline{\Gamma}{}_e(BM, B^4A(1)),\quad \mu\mapsto \mu_M.
$$

In the study of Eisenstein series for covering groups, the following question arises naturally: what are the ``interesting'' metaplectic covers of $G$ whose restriction to $P$ descends to $\pi_1(M)\otimes\mathbb G_m$? (I thank Laurent Clozel for posing this question to the author.)
\end{void}

\begin{cor}
\label{cor-canonical-splittings}
Suppose that there exists an \'etale metaplectic cover $\mu$ of $G$ which does not descend to $\pi_1(G)\otimes\mathbb G_m$, but $\mu_M$ descends to $\pi_1(M)\otimes\mathbb G_m$. Then at least one of the following statements hold:
\begin{enumerate}
	\item $A(-1)$ has $2$-torsion and $G_{\mathrm{sc}}$ contains a factor of type $B_n$, $C_n$, or $F_4$;
	\item $A(-1)$ has $3$-torsion and $G_{\mathrm{sc}}$ contains a factor of type $G_2$.
\end{enumerate}
\end{cor}
\begin{proof}
By Proposition \ref{prop-abelian-covers}, the descent properties of $\mu$, respectively $\mu_M$, are of \'etale local nature on $S$. Combined with Proposition \ref{prop-reductive-classification}, we see that they are controlled by the vanishing of the associated quadratic form $Q$ on simple coroots.

For $Q$ to vanish on the simple coroots contained in $M$ but not on all simple coroots of $G$, the corresponding Dynkin diagram must be non-simply laced and some $Q(\alpha)\neq 0$ ($\alpha\in\Delta$) must be annihilated by the ratio of the square length of the long and short coroots.
\end{proof}

\begin{rem}
For $G = \Sp_{2n}$ and $A = \{\pm 1\}$ over a local field $F$ of characteristic $\neq 2$, the metaplectic double cover of $\Sp_{2n}(F)$ is defined by the \'etale metaplectic cover corresponding to the $\mathbb Z/2$-valued quadratic form $Q$ taking value $1$ on a short coroot.

It follows from a classical calculation that when restricted to the Siegel parabolic, the metaplectic cover descends along $\det : \GL_n(F)\rightarrow F^{\times}$, see \cite[Corollary 5.5]{MR1197062}.

Proposition \ref{prop-abelian-covers} lifts this statement to the level of geometry: it asserts that the corresponding \'etale metaplectic cover already descends to $\mathbb G_m$. The resulting cover of $\mathbb G_m$ may be identified by restriction along the unique short simple coroot $\alpha$ of $\Sp_{2n}$:
$$
\begin{tikzcd}
	\mathbb G_m\ar[d] \ar[r, "e_n"] & \GL_n \ar[r, "\det"]\ar[d] & \mathbb G_m \\
	\SL_2 \ar[r, "f_{\alpha}"] & \Sp_{2n}
\end{tikzcd}
$$
where $e_n$ is given by the action on the last basis vector, $f_{\alpha}$ is the map corresponding to $\alpha$, and $\mathbb G_m\subset\SL_2$, $\GL_n\subset\Sp_{2n}$ are the Levi subgroups. Thus the induced cover of $\mathbb G_m$ is defined by the cocycle $\mathbb Z\otimes\mathbb Z\rightarrow\mathbb Z/2$, $1\otimes 1\mapsto 1$ by Lemma \ref{lem-restriction-sl2}.

We also note that the integral parametrization of metaplectic covers, by central extension by $\underline K_2$ for example, does not see this phenomenon.
\end{rem}

\begin{rem}
\label{rem-z-extension}
For $S$ the spectrum of a local field $F$, it is possible to describe covering groups of $G(F)$ coming from \'etale metaplectic covers of $\pi_1(G)\otimes\mathbb G_m$ with the help of a \emph{$z$-extension}, i.e.~a central extension:
$$
1\rightarrow T_2\rightarrow G'\rightarrow G\rightarrow 1
$$
of reductive groups satisfying the following properties:
\begin{enumerate}
	\item the map $G'\rightarrow G$ induces the simply connected cover $G_{\mathrm{sc}} \rightarrow G_{\der}$ upon taking derived subgroups;
	\item $T_2$ is a quasi-trivial torus.
\end{enumerate}

Letting $T_1 := G'/G'_{\der}$, we have an identification $T_1/T_2 \cong \pi_1(G)\otimes\mathbb G_m$ of monoidal stacks. Thus any \'etale metaplectic cover $\mu$ of $\pi_1(G)\otimes\mathbb G_m$ defines one for $T_1$, with equivariance structure against the $T_2$-action. Upon evaluation at $\Spec(F)$ (see \S\ref{sect-metaplectic-cover-local}), we obtain a central extension of topological groups equipped with a splitting:
$$
\begin{tikzcd}[column sep = 1.5em]
	& & & T_2(F) \ar[dl]\ar[d] \\
1 \ar[r] & H_0(F, A) \ar[r] & \tilde T_1\ar[r] & T_1(F) \ar[r] & 1
\end{tikzcd}
$$
whose image is \emph{central} in $\tilde T_1$. Indeed, the centrality assertion follows from the equivariance structure $\tilde T_1\times T_2(F) \cong \act^*\tilde T_1$ (for $\act : T_1(F)\times T_2(F)\rightarrow T_1(F)$ the action map) being a group isomorphism.

Pulling back the extension $\tilde T_1$ along $G'(F)\rightarrow T_1(F)$ and taking the quotient by $T_2(F)$, we obtain a central extension $\tilde G$ of $G(F)$ by $H_0(F, A)$, using the vanishing of $H^1(F, T_2)$. This $\tilde G$ is the covering group associated to the pullback of $\mu$ along $G \rightarrow\pi_1(G)\otimes\mathbb G_m$.
\end{rem}

\subsection{Metaplectic center}
\label{sect-metaplectic-center}

\begin{void}
Suppose that $G\rightarrow S$ is a reductive group scheme. Let $Z_G$ denote its center and $Z\subset Z_G$ its maximal torus. Equivalently, $Z$ is the radical torus of $G$. We write $\Lambda_Z$ for its sheaf of cocharacters.

The map $Z\times G \rightarrow G$, $(z, g)\mapsto zg$ being a group homomorphism, it induces an action of the monoidal stack $B(Z)$ on $BG$.
\end{void}

\begin{void}
Given an \'etale metaplectic cover $\mu$ of $G$ with associated quadratic form $Q$, we define an isogeny $Z^{\sharp} \rightarrow Z$ as follows. The symmetric form $b$ associated to $Q$ induces a pairing:
\begin{equation}
\label{eq-symmetric-form-center}
\Lambda_Z \otimes \pi_1(G) \rightarrow A(-1),\quad \lambda_1,\lambda_2\mapsto b(\lambda_1,\lambda_2)
\end{equation}
as $\Lambda_Z$ pairs trivially with $\Lambda_{\mathrm{sc}}$ according to the formula \eqref{eq-strict-W-invariance}. The kernel of this pairing is a sublattice $\Lambda_{Z^{\sharp}}\subset\Lambda_Z$ with finite cokernel, so it defines the isogeny $Z^{\sharp} \rightarrow Z$.

Let $\mu_{Z^{\sharp}}$ denote the restriction of $\mu$ along $B(Z)\rightarrow BG$. By Proposition \ref{prop-commutative-covers}, $\mu_{Z^{\sharp}}$ has the canonical structure of an $\mathbb E_1$-monidal (in fact $\mathbb E_{\infty}$-monoidal) morphism $B(Z^{\sharp}) \rightarrow B^4A(1)$.

The goal of this section is to establish the following result, whose proof gives a geometric interpretation of the pairing \eqref{eq-symmetric-form-center}.
\end{void}

\begin{prop}
\label{prop-metaplectic-center}
With respect to the $B(Z^{\sharp})$-action on $BG$, the pointed morphism $\mu : BG \rightarrow B^4A(1)$ has a canonical $B(Z^{\sharp})$-equivariance structure against the $\mathbb E_1$-monoidal character $\mu^{\sharp} : B(Z^{\sharp}) \rightarrow B^4A(1)$.
\end{prop}

\begin{void}
In concrete terms, the assertion of Proposition \ref{prop-metaplectic-center} means that along the simplicial diagram encoding the $B(Z^{\sharp})$-action on $BG$:
$$
\begin{tikzcd}
	\cdots \ar[r, shift left = 1.5ex] \ar[r, shift left = 0.5ex] \ar[r, shift left = -0.5ex] \ar[r, shift left = -1.5ex] & B(Z^{\sharp})^{\times 2}\times BG \ar[r, shift left = 1.5ex, "m\times\id"]\ar[r, shift left = -0.5ex, "\id\times\act"] \ar[r, shift left = -1.5ex, swap, "\pr"] & B(Z^{\sharp})\times BG \ar[r, shift left = 0.5ex, "\act"]\ar[r, shift left = -0.5ex, swap, "\pr"] & BG,
\end{tikzcd}
$$
we have an isomorphism of rigidified sections of $B^4A(1)$ over $B(Z^{\sharp})\times BG$:
\begin{equation}
\label{eq-metaplectic-center-equivariance}
\act^*(\mu) \cong \mu_{Z^{\sharp}}\boxtimes \mu,
\end{equation}
and a $2$-isomorphism witnessing the commutativity of:
\begin{equation}
\label{eq-metaplectic-center-cocycle-diagram}
\begin{tikzcd}[column sep = 1em]
	(m\times \id)^*\act^*(\mu) \ar[r, "\cong"]\ar[d, "\cong"] & (\mu_{Z^{\sharp}}\boxtimes \mu_{Z^{\sharp}})\boxtimes\mu \ar[d, "\cong"] \\
	(\id\times\act^*)\act^*(\mu) \ar[r, "\cong"] & \mu_{Z^{\sharp}}\boxtimes(\mu_{Z^{\sharp}}\boxtimes\mu)
\end{tikzcd}
\end{equation}
subject to the natural cocycle condition over $B(Z^{\sharp})^{\times 3}\times BG$.
\end{void}

\begin{rem}
If $S$ is the spectrum of a local field $F$, then $\mu$ defines a topological covering group $\tilde G$ of $G(F)$ (or $G(F)^0$ if $F$ is real). Denote by $\tilde Z$ its restriction to $Z^{\sharp}(F)$. The assertion of Proposition \ref{prop-metaplectic-center} implies that the image of $\tilde Z$ is central in $\tilde G$, using the same observation as in Remark \ref{rem-z-extension}.
\end{rem}

\begin{void}
\label{void-birigidified-morphism-torus-reductive}
In order to prove Proposition \ref{prop-metaplectic-center}, we first observe that for any torus $T\rightarrow S$ with sheaf of cocharacters $\Lambda_T$, the groupoid of bi-rigidified morphisms:
$$
BT \times BG \rightarrow B^4A(1)
$$
is equivalent to discrete abelian group of pairings $\Lambda_T\times \pi_1(G) \rightarrow A(-1)$.

Indeed, this assertion follows from the same proof as Lemma \ref{lem-birigidified-morphism-classification}, where we use the fiber sequence \eqref{eq-reductive-classification-cofiber-sequence} to identify the neutral component of $\underline{\Gamma}{}_e(BG, B^4A(1))$.
\end{void}

\begin{proof}[Proof of Proposition \ref{prop-metaplectic-center}]
We observe that $\act^*(\mu) \otimes (\mu_{Z^{\sharp}}\boxtimes\mu)^{\otimes -1}$ admits a natural structure as a bi-rigidified morphism $B(Z^{\sharp})\times BG \rightarrow B^4A(1)$. In order to construct the isomorphism \eqref{eq-metaplectic-center-equivariance}, it suffices to show that the corresponding pairing:
\begin{equation}
\label{eq-metaplectic-center-equivariance-pairing}
\Lambda_{Z^{\sharp}} \otimes \pi_1(G) \rightarrow A(-1)
\end{equation}
vanishes. The statement being of \'etale local nature on $S$, we may split $G$ as well as fix a Borel subgroup containing a split maximal torus $T_1\subset B_1\subset G$. Then $T_1$ is isomorphic to the universal Cartan of $G$ and the pairing \eqref{eq-metaplectic-center-equivariance-pairing} may be calculated after restriction along $\Lambda_{T_1}\rightarrow \pi_1(G)$. Proposition \ref{prop-pointed-morphism-quadratic} then shows that \eqref{eq-metaplectic-center-equivariance-pairing} is the restriction of \eqref{eq-symmetric-form-center} along $\Lambda_{Z^{\sharp}}\rightarrow\Lambda_Z$, which vanishes by the definition of $\Lambda_{Z^{\sharp}}$.

To construct the $2$-isomorphism \eqref{eq-metaplectic-center-cocycle-diagram}, we note that all four terms are naturally isomorphic over $e\times BG$, $B(Z^{\sharp})^{\times 2}\times e$, and these isomorphisms are compatible over $e\times e$. By the discreteness of bi-rigidified morphisms $B(Z^{\sharp})^{\times 2}\times BG \rightarrow B^4A(1)$, any isomorphism:
$$
(m\times\id)^*\act^*(\mu) \cong \mu_{Z^{\sharp}}\boxtimes(\mu_{Z^{\sharp}}\boxtimes\mu)
$$
compatible with the given ones, if exists, is necessarily unique. The same consideration yields the cocycle condition over $B(Z^{\sharp})^{\times 3}\times BG$.
\end{proof}

\subsection{The $(T_{\ad}, T_{\mathrm{sc}})$-commutator}
\label{sect-adjoint-equivariance}

\begin{void}
Suppose that $G\rightarrow S$ is a reductive group scheme. Write $G_{\ad}$ for the quotient $G/Z_G$ and $T_{\ad}$ its universal Cartan with sheaf of cocharacters $\Lambda_{\ad}$. The map $G\rightarrow G_{\ad}$ induces a natural map $\Lambda\rightarrow \Lambda_{\ad}$, which is generally not surjective.

On the other hand, the conjugation action of $G$ on itself factors through $G_{\ad}$. By the functoriality of the simply connected cover, we also obtain a $G_{\ad}$-action on $G_{\mathrm{sc}}$.
\end{void}

\begin{void}
Suppose that $\mu$ is an \'etale metaplectic cover of $G$ with associated quadaratic form $Q$ and symmetric form $b$. Since $Q$ is strict, the restriction of $b$ to $\Lambda\otimes\Lambda_{\mathrm{sc}}$ extends to a bilinear form according to the formula:
\begin{equation}
\label{eq-adjoint-simply-connected-pairing}
\Lambda_{\ad}\otimes \Lambda_{\mathrm{sc}} \rightarrow A(-1),\quad \lambda, \alpha\mapsto Q(\alpha)\langle\check{\alpha},\lambda\rangle,
\end{equation}
where $\alpha$ is any simple coroot. Here, the bracket denotes the natural duality pairing between $\Lambda_{\ad}$ and the root lattice.

The goal of this section is interpret \eqref{eq-adjoint-simply-connected-pairing} as a kind of ``commutator pairing'' between $T_{\ad}$ and $T_{\mathrm{sc}}$, arising from $G_{\ad}$-action on $G_{\mathrm{sc}}$.

We write $\mu_{G_{\mathrm{sc}}}$ for the restriction of $\mu$ to $G_{\mathrm{sc}}$, viewed as an $\mathbb E_1$-monoidal morphism $G_{\mathrm{sc}} \rightarrow B^3A(1)$.
\end{void}

\begin{void}
We assume the existence of a Borel subgroup $B\subset G$, which we fix from now on. The induced Borel subgroup of $G_{\ad}$ is written as $B_{\ad}$. The $B_{\ad}$-action on $G_{\mathrm{sc}}$ induces two commutative diagrams:
\begin{equation}
\label{eq-adjoint-group-action-diagram}
\begin{tikzcd}[column sep = 1.5em]
	B_{\ad}\times G_{\mathrm{sc}} \ar[d, shift left = 0.5ex, "\act"]\ar[d, shift left = -0.5ex, swap, "\pr"] & B_{\ad}\times B_{\mathrm{sc}} \ar[l]\ar[r] \ar[d, shift left = 0.5ex, "\act"]\ar[d, shift left = -0.5ex, swap, "\pr"] & T_{\ad}\times T_{\mathrm{sc}} \ar[d, shift left = 0.5ex, "\act"]\ar[d, shift left = -0.5ex, swap, "\pr"] \\
	G_{\mathrm{sc}} & B_{\mathrm{sc}} \ar[l]\ar[r] & T_{\mathrm{sc}}
\end{tikzcd}
\end{equation}

Since $G_{\mathrm{sc}}$ is simply connected, $\underline{\Gamma}{}_e(BG_{\mathrm{sc}}, B^4A(1))$ is discrete and, under the functor of forgetting the $\mathbb E_1$-monoidal structure, isomorphic to $\underline{\Gamma}{}_e(G_{\mathrm{sc}}, B^3A(1))$. In particular, $\mu_{G_{\mathrm{sc}}}$ acquires an automatic $B_{\ad}$-equivariance structure. The isomorphism $\pr^*(\mu_{G_{\mathrm{sc}}}) \cong \act^*(\mu_{G_{\mathrm{sc}}})$ may be restricted along \eqref{eq-adjoint-group-action-diagram} to give an isomorphism over $T_{\ad}\times T_{\mathrm{sc}}$:
\begin{equation}
\label{eq-adjoint-group-equivariance}
\pr^*(\res_{B_{\mathrm{sc}}}(\mu_{G_{\mathrm{sc}}})) \cong \act^*(\res_{B_{\mathrm{sc}}}(\mu_{G_{\mathrm{sc}}})).
\end{equation}

On the other hand, $\act = \pr$ as maps $T_{\ad}\times T_{\mathrm{sc}} \rightarrow T_{\mathrm{sc}}$. Hence \eqref{eq-adjoint-group-equivariance} may be viewed as an automorphism of the trivial object in the groupoid $\underline{\Gamma}{}_{e, e}(T_{\ad}\times T_{\mathrm{sc}}, B^3A(1))$, i.e.~a section of $\underline{\Gamma}{}_{e, e}(T_{\ad}\times T_{\mathrm{sc}}, B^2A(1))$. Finally, a variant of Lemma \ref{lem-birigidified-morphism-classification} (with the same proof) gives an equivalence:
\begin{equation}
\label{eq-adjoint-simply-connected-birigidified}
\underline{\Gamma}{}_{e,e}(T_{\ad}\times T_{\mathrm{sc}}, B^2A(1)) \cong \check{\Lambda}_{\ad}\otimes \check{\Lambda}_{\mathrm{sc}}\otimes A(-1).
\end{equation}
\end{void}

\begin{prop}
\label{prop-adjoint-simply-connected-commutator}
The section of \eqref{eq-adjoint-simply-connected-birigidified} defined by the isomorphism \eqref{eq-adjoint-group-equivariance} agrees with the pairing \eqref{eq-adjoint-simply-connected-pairing}.
\end{prop}
\begin{proof}
The assertion is of \'etale local nature on $S$, so we may assume that $G$ splits. Furthermore, we fix a splitting of the surjection $B\rightarrow T$ and view $T$ (resp.~$T_{\mathrm{sc}}$) as a subgroup of $G$ (resp.~$G_{\mathrm{sc}}$). The functor $\res_{B_{\mathrm{sc}}}$ may be identified with the restriction $\res_{T_{\mathrm{sc}}}$ along $T_{\mathrm{sc}} \subset G_{\mathrm{sc}}$.

With this set-up, the isomorphism \eqref{eq-adjoint-group-equivariance}, evaluated at $t \in T_{\ad}$, has a concrete interpretation as the following loop in $\underline{\Maps}{}_{\mathbb E_1}(T_{\mathrm{sc}}, B^3A(1))$:
\begin{equation}
\label{eq-adjoint-simply-connected-loop}
\begin{tikzcd}[column sep = 1.5em]
	\res_{T_{\mathrm{sc}}}(\mathrm{int}_{t}^*\mu_{G_{\mathrm{sc}}}) \ar[r, "\cong"] & \res_{T_{\mathrm{sc}}}(\mu_{G_{\mathrm{sc}}}) \ar[d, "\cong"] \\
	\mathrm{int}_{t}^*\res_{T_{\mathrm{sc}}}(\mu_{G_{\mathrm{sc}}}) \ar[u, "\cong"] & \res_{T_{\mathrm{sc}}}(\mu_{G_{\mathrm{sc}}}) \ar[l, "\cong"]
\end{tikzcd}
\end{equation}
Here, the top isomorphism is induced from $\mathrm{int}_t^*(\mu_{G_{\mathrm{sc}}}) \cong \mu_{G_{\mathrm{sc}}}$ over $G_{\mathrm{sc}}$, whereas the bottom isomorphism is due to the triviality of the $T_{\ad}$-action on $T_{\mathrm{sc}}$.

To identify this loop, we may choose a simple coroot $\alpha$ of $G$, viewed as a morphism $\mathbb G_m\rightarrow T_{\mathrm{sc}}$. It extends to an inclusion $f_{\alpha} : \SL_2\subset G_{\mathrm{sc}}$ preserved by the $T_{\ad}$-action. For each $\lambda\in\Lambda_{T_{\ad}}$, viewed as a morphism $\mathbb G_m \rightarrow T_{\ad}$, we find an induced $\mathbb G_m$-action on $\SL_2$ which is the $\langle\check{\alpha}, \lambda\rangle$-multiple of the action of the adjoint torus of $\SL_2$ on $\SL_2$. Hence the pullback of \eqref{eq-adjoint-simply-connected-loop} along $\alpha$ is given by the same diagram applied to $f_{\alpha}^*(\mu_{G_{\mathrm{sc}}})$ and $t^{\langle\check{\alpha},\lambda\rangle} \in \mathbb G_m$, viewed as a section of the adjoint torus of $\SL_2$.

On the other hand, the \'etale metaplectic cover $f_{\alpha}^*(\mu_{G_{\mathrm{sc}}})$ of $\SL_2$ is classfied by $Q(\alpha)\in A(-1)$. By construction, the loop \eqref{eq-adjoint-simply-connected-loop} applied to $t\in\mathbb G_m$ is naturally the restriction along $t\times\mathbb G_m$ of a bi-rigidified morphism:
\begin{equation}
\label{eq-adjoint-simply-connected-birigidified-SL2}
\mathbb G_m \times \mathbb G_m \rightarrow B^2A(1).
\end{equation}
It remains to show that \eqref{eq-adjoint-simply-connected-birigidified-SL2} corresponds to $Q(\alpha)\in A(-1)$ under Lemma \ref{lem-birigidified-morphism-classification}. However, we know from Lemma \ref{lem-borel-equivariance-bostruction} that when applied to $t^2\in\mathbb G_m$, \eqref{eq-adjoint-simply-connected-loop} corresponds to the commutator $b(\alpha, \alpha) = 2Q(\alpha)$. Thus \eqref{eq-adjoint-simply-connected-birigidified-SL2} corresponds to a square root of $2Q(\alpha)$.

Finally, a square root of $2a$ defined functorially for all finite abelian groups $A$ equipped with an element $a\in A$ must equal $a$, e.g.~by the integrality of $2$-adic integers.
\end{proof}

\subsection{The case $S=\Spec(\mathbb R)$}
\label{sect-real-place}

\begin{void}
We study the case $S = \Spec(\mathbb R)$ as an extended example. The \'etale metaplectic covers of $G$ define topological covers of a subgroup $G(\mathbb R)^0\subset G(\mathbb R)$ by the construction of \S\ref{void-metaplectic-cover-local-construction-real}.

In this section, we give a criterion for when $G(\mathbb R)^0 = G(\mathbb R)$ in terms of the classification of \'etale metaplectic covers.
\end{void}

\begin{void}
Let us begin with $G=\mathbb G_m$. The splitting of Remark \ref{rem-splitting-theta-data} defines a $\mathbb Z$-linear morphism:
\begin{equation}
\label{eq-projection-real-to-linear-component}
	\underline{\Gamma}{}_e(B\mathbb G_m, B^4A(1)) \rightarrow B^2A.
\end{equation}

On the other hand, every pointed morphism $\mu : B\mathbb G_m\rightarrow B^4A(1)$ defines an $\mathbb E_1$-monoidal morphism $\mathbb G_m \rightarrow B^3A(1)$, hence a map $\mathbb G_m(\mathbb R) \rightarrow H^3(\mathbb R, A(1))$.

Denote its value at $(-1)\in\mathbb G_m(\mathbb R)$ by $\sgn(\mu)$.
\end{void}

\begin{lem}
\label{lem-real-place-gm}
The element $\sgn(\mu)$ equals the class of the image of $\mu$ under \eqref{eq-projection-real-to-linear-component}, followed by the cup product with the restriction of the Kummer class to $(-1)\in\mathbb G_m(\mathbb R)$:
\begin{equation}
\label{eq-real-number-cohomology-identification}
\cup (-1)^*[\Psi] : H^2(\mathbb R, A) \rightarrow H^3(\mathbb R, A(1)).
\end{equation}
\end{lem}
\begin{proof}
If $\mu$ is the image of the splitting $A(-1) \rightarrow \underline{\Gamma}{}_e(B\mathbb G_m, B^4A(1))$ of Remark \ref{rem-splitting-theta-data}, viewing elements of the source as $A(-1)$-valued quadratic forms on $\mathbb Z$, then $\sgn(\mu) = 0$. Indeed, $\sgn(\mu)$ depends only on the underlying pointed morphism of the associated $\mathbb E_1$-monoidal morphism $\mathbb G_m \rightarrow B^3A(1)$, which is trivial in this case.

It remains to show that for a pointed morphism $\mu : B\mathbb G_m \rightarrow B^4A(1)$ defined by a section $t$ of $A[2]$, the element $\sgn(\mu)$ is the image of $t$ under \eqref{eq-real-number-cohomology-identification}. In this case, we have $\mu \cong \Psi^*(t)$.

Evaluation at $(-1)\in\mathbb G_m(\mathbb R)$ corresponds to pulling back along $(-1) : \Spec(\mathbb R) \rightarrow \mathbb G_m$. Replacing the Yoneda product by the cup product, we obtain an equality of cohomology classes:
$$
\sgn(\mu) = [t] \cup (-1)^*[\Psi] \in H^3(\mathbb R, A(1)),
$$
which is what we wanted to prove.
\end{proof}

\begin{rem}
If $A$ is the constant sheaf $\{\pm 1\}$, then \eqref{eq-real-number-cohomology-identification} is the unique isomorphism between two copies of $\mathbb Z/2$. Indeed, the cohomology ring $H^*(\mathbb R, \{\pm 1\})$ is isomorphic to $\mathbb F_2[x]$ with a degree-$1$ generator $x = (-1)^*[\Psi]$.
\end{rem}

\begin{void}
Let $G$ be a reductive group scheme over $\mathbb R$ together with a maximal torus $T$. Write $T_0\subset T$ for its maximal split subtorus. Given a pointed morphism $\mu : BG \rightarrow B^4A(1)$, we obtain a homomorphism:
\begin{equation}
\label{eq-signature-reductive-group}
\Lambda_{T_0} \rightarrow H^3(\mathbb R, A(1)),\quad \lambda\mapsto \sgn(\lambda^*\mu).
\end{equation}
where $\Lambda_{T_0}$ denotes the group of cocharacters of $T_0$.

On the other hand, recall from \S\ref{void-metaplectic-cover-local-construction-real} that the $\mathbb E_1$-monoidal morphism $G\rightarrow B^3A(1)$ associated to $\mu$ defines a morphism of topological groups $G(\mathbb R)\rightarrow H^3(\mathbb R, A(1))$, whose kernel is $G(\mathbb R)^0$.
\end{void}

\begin{lem}
\label{lem-real-place-reductive}
The map \eqref{eq-signature-reductive-group} is zero if and only if $G(\mathbb R)^0 = G(\mathbb R)$.
\end{lem}
\begin{proof}
The ``only if'' direction requires a proof. Since the morphism $G(\mathbb R)\rightarrow H^3(\mathbb R, A(1))$ is continuous (\cite[Lemme 2.10]{MR1441006}), it factors through $\pi_0(G(\mathbb R))$.

According to \cite[Th\'eor\`eme 14.4]{MR207712}, the map $\pi_0(T_0(\mathbb R)) \rightarrow \pi_0(G(\mathbb R))$ is surjective. Hence, it suffices to show that the induced map of abelian groups:
$$
\pi_0(T_0(\mathbb R)) \rightarrow \tn H^3(\mathbb R, A(1))
$$
vanishes identically. Now, $\pi_0(T_0(\mathbb R))$ is a direct sum of copies of $\mathbb Z/2$, and the hypothesis shows that its restriction to each copy of $\mathbb Z/2$ vanishes.
\end{proof}

\begin{rem}
Using Lemma \ref{lem-real-place-reductive}, we see that $G(\mathbb R)^0 = G(\mathbb R)$ for all \'etale metaplectic covers $\mu$ if $T$ does not contain a split subtorus.
\end{rem}

\medskip

\section{Metaplectic dual data}
\label{sec-metaplectic-dual}

In this section, we define the metaplectic dual data of a reductive group scheme $G\rightarrow S$ equipped with an \'etale metaplectic cover $\mu$. In the split case, the definition we give in \S\ref{sect-metaplectic-dual-pair} relies on the materials of \S\ref{sec-torus} and \S\ref{sec-classification}. The construction is functorial in the pair $(G, \mu)$ and the nonsplit case follows from \'etale descent.

Metaplectic dual data are first defined in the function field context for split reductive groups by Gaitsgory--Lysenko \cite{MR3769731} and take the form of a triple $(H, \cal G_{Z_H}, \epsilon)$. Our definition mimics theirs, although it is purely group-theoretic and is valid in general. Over a local or global field, our construction is also closely related to Weissman's metaplectic $\tn L$-group \cite{MR3802418}, with which we eventually compare in \S\ref{sect-classical-L-group}.

The main construction of this section is in \S\ref{sect-dual-category}: we obtain the category of representations of the metaplectic $\tn L$-group \emph{without} the need to construct the group itself. It is sufficient for defining $\tn L$-parameters, which are certain tensor functors (Definition \ref{defn-tannakian-L-parameter}). This Tannakian perspective is a well known alternative to Langlands' definition of the $\tn L$-group, and I believe that it adds some transparency, especially in the metaplectic context.

\begin{void}
Throughout this section, we fix a scheme $S$. We shall need a ``field of coefficients'' which is typically $\mathbb C$ or $\overline{\mathbb Q}{}_{\ell}$. For concreteness, we focus on the latter case.

More precisely, let $\ell$ be a prime invertible on $S$ and fix an algebraic closure $\mathbb Q_{\ell}\subset\overline{\mathbb Q}_{\ell}$. Let $E \subset\overline{\mathbb Q}_{\ell}$ be a finite extension of $\mathbb Q_{\ell}$. The coefficient for metaplectic covers will be a subsheaf $A\subset \underline E^{\times}$ of finite abelian groups whose order is invertible on $S$.
\end{void}

\subsection{The split case}
\label{sect-metaplectic-dual-pair}

\begin{void}
\label{void-metaplectic-dual-context}
Let $G\rightarrow S$ be a split reductive group scheme and $\mu$ be a pointed morphism $B(G) \rightarrow B^{4}A(1)$. We use the notations of \S\ref{void-reductive-group-notations} for group-theoretic notions.

We shall first construct a pair $(H, F_{\mu})$ where $H$ is a pinned reductive group scheme over $\Spec(\mathbb Z)$ and $F_{\mu}$ is an $\mathbb E_{\infty}$-monoidal morphism:
\begin{equation}
\label{eq-E-infinity-morphism-characters-of-center}
F_{\mu} : \hat Z_H \rightarrow B^{2}(A),
\end{equation}
where $Z_H$ denotes the center of $H$ and $\hat Z_H$ the abelian group of its characters, viewed as a constant \'etale sheaf of abelian groups over $S$.
\end{void}

\begin{rem}
Using the inclusion $A\subset\underline E^{\times}$, we may view $F_{\mu}$ as valued in $B^2(\underline E^{\times})$, but we will still use the same notation for it.
\end{rem}

\begin{void}
Recall that to each pointed morphism $\mu : BG \rightarrow B^4A(1)$, there is an associated strict quadratic form $Q\in\Quad(\Lambda, A(-1))_{\st}$, see \S\ref{void-associated-quadratic-form}.

The construction of $H$ depends only on $Q$ (see also \cite[2.2.5]{MR1227098}). Define $\Phi^{\sharp}\subset\Lambda^{\sharp}$ to be the subset $\{\ord(Q(\alpha))\alpha\mid\alpha\in\Phi\}$, where $\ord(Q(\alpha))\in\mathbb N$ denotes the order of $Q(\alpha)$.

Let $\check{\Lambda}^{\sharp}$ be the dual of $\Lambda^{\sharp}$. Since $Q$ is strict, $\ord(Q(\alpha))^{-1}\check{\alpha}$ takes integral values on $\Lambda^{\sharp}$. Hence we have a well-defined subset $\check{\Phi}^{\sharp} = \{\ord(Q(\alpha))^{-1}\check{\alpha}\mid\check{\alpha}\in\check{\Phi}\}$ of $\check{\Lambda}^{\sharp}$. The bijection $\Phi\cong\check{\Phi}$ induces a bijection $\Phi^{\sharp}\cong\check{\Phi}^{\sharp}$. The systems of simple coroots and roots $\Delta^{\sharp}$, $\check{\Delta}^{\sharp}$ are determined by the corresponding elements of $\Delta$, $\check{\Delta}$.

One verifies that $(\Delta^{\sharp}\subset\Phi^{\sharp}\subset\Lambda^{\sharp}, \check{\Delta}^{\sharp}\subset\check{\Phi}^{\sharp}\subset\check{\Lambda}^{\sharp})$, $\Phi^{\sharp}\cong\check{\Phi}^{\sharp}$ defines a based reduced root datum. Let $G^{\sharp}$ be the corresponding pinned reductive group over $S$. Define $H$ to be the Langlands dual of $G^{\sharp}$, viewed as a pinned split reductive group over $\Spec(\mathbb Z)$. (This means that $\Lambda_T^{\sharp}$ is the \emph{character} lattice of the maximal torus of $H$, etc.)
\end{void}

\begin{void}
To construct $F_{\mu}$, we need some more notations. Regard the quadratic form $Q$ as fixed and write $\underline{\Gamma}{}_e(BG, B^4A(1))_Q$ for the full subgroupoid of pointed morphisms $BG \rightarrow B^4A(1)$ whose associated quadratic form is $Q$.

We collect from the canonical morphism \eqref{eq-canonical-morphism-to-sharp-torus} and Proposition \ref{prop-commutative-cover-classification} the following canonical morphism:
\begin{equation}
\label{eq-metaplectic-cover-sharp-morphism}
\underline{\Gamma}{}_e(BG, B^4A(1))_Q \rightarrow \underline{\Maps}{}_{\mathbb E_{\infty}}(\Lambda^{\sharp}, B^2A).
\end{equation}
Its naturality with respect to the map $G_{\mathrm{sc}}\rightarrow G$ shows that the following diagram commutes:
\begin{equation}
\label{eq-restriction-to-commutive-covers-naturality}
\begin{tikzcd}[column sep = 1em]
	\underline{\Gamma}{}_e(BG, B^4A(1))_Q \ar[r]\ar[d] & \underline{\Maps}{}_{\mathbb E_{\infty}}(\Lambda^{\sharp}, B^2A)\ar[d] \\
	\underline{\Gamma}{}_e(BG_{\mathrm{sc}}, B^4A(1))_{Q_{\mathrm{sc}}} \ar[r] & \underline{\Maps}{}_{\mathbb E_{\infty}}(\Lambda^{\sharp, r}, B^2A)
\end{tikzcd}
\end{equation}
\end{void}

\begin{void}
\label{void-trivialization-commutative-cover-simply-connected}
Let us construct a trivialization of the bottom horizontal functor in \eqref{eq-restriction-to-commutive-covers-naturality}.
\end{void}

\begin{proof}[Construction]
The question concerns only a simply connected group $G_{\mathrm{sc}}$ so we will write $G = G_{\mathrm{sc}}$ to ease the notation. We shall construct the trivialization \'etale locally on $S$ assuming that $G$ splits using a fixed pinning. Then we argue that the trivialization does not depend on the pinning, so it exists over $S$ by \'etale descent.

In the split case, $\underline{\Gamma}{}_e(BG, B^4A(1))_Q$ is the singleton with element $Q$, corresponding to an \'etale metaplectic cover $\mu$, and we shall trivialize its image in $\underline{\Maps}{}_{\mathbb E_{\infty}}(\Lambda^{\sharp, r}, B^2A)$. The lattice $\Lambda^{\sharp, r}$ has a globally defined basis with elements $\alpha^{\sharp}\in\Delta^{\sharp}$, so we have a fiber sequence (see Remark \ref{rem-commutative-cover-cofiber-sequence}):
$$
\bigoplus_{\alpha^{\sharp}\in\Delta^{\sharp}} B^2A \rightarrow \underline{\Maps}{}_{\mathbb E_{\infty}}(\Lambda^{\sharp, r}, B^2A) \rightarrow \underline{\Maps}{}_{\mathbb Z}(\Lambda^{\sharp, r}/2, A(-1)).
$$
By the definition of $\Delta^{\sharp}$, the image of $\mu$ in $\underline{\Maps}{}_{\mathbb E_{\infty}}(\Lambda^{\sharp, r}, B^2A)$ lies in the full subgroupoid $\bigoplus_{\alpha^{\sharp}\in\Delta^{\sharp}}B^2A$, so it suffices to trivialize the restriction along each $\alpha^{\sharp}\in\Delta^{\sharp}$, viewed as a map $\mathbb Z\rightarrow\Lambda^{\sharp, r}$. Denote the corresponding section of $B^2A$ by $(\alpha^{\sharp})^*(\mu)$.

Let us choose a pinning $T_1 \subset B_1 \subset G$, $N_{\alpha_1}\cong\mathbb G_a$ for $\alpha_1 : \mathbb G_m\rightarrow T_1$ the homomorphism corresponding to the coroot $\alpha\in\Phi$ and $N_{\alpha_1}\subset G$ the corresponding root subgroup. For each $\alpha\in\Delta$, there is a unique homomorphism $f_{\alpha_1} : \SL_2 \rightarrow G$ of \emph{pinned} split reductive group schemes inducing $\alpha_1$ on the maximal tori. The section $(\alpha^{\sharp})^*(\mu)$ of $B^2A$ is given by the restriction of the \'etale metaplectic cover $\ord(Q(\alpha))\cdot f_{\alpha_1}^*(\mu)$ along $\mathbb G_m\subset \SL_2$. It is trivialized as $f_{\alpha_1}^*(\mu)$ is classified by $Q(\alpha)\in A(-1)$, which is annihilated by $\ord(Q(\alpha))$.

For a distinct pinning $T_2\subset B_2\subset G$, $N_{\alpha_2}\cong\mathbb G_a$, we need to construct a commutative diagram of sections of $B^2A$:
\begin{equation}
\label{eq-simply-connected-trivialization-compatibility}
\begin{tikzcd}[column sep = 1.5em]
	\res_{\mathbb G_m}(\ord(Q(\alpha))\cdot f_{\alpha_1}^*(\mu)) \ar[r, phantom, "\cong"]\ar[d, "\cong"] & * \ar[d, "\cong"] \\
	 \res_{\mathbb G_m}(\ord(Q(\alpha))\cdot f_{\alpha_2}^*(\mu)) \ar[r, phantom, "\cong"] & *
\end{tikzcd}
\end{equation}
where the horizontal morphisms are the trivializations constructed above, the right vertical arrow is the identity, and the left vertical arrow is induced from the equality $f_{\alpha_1}^*(\mu) = f_{\alpha_2}^*(\mu)$, arising from $\mathrm{int}_g^*(\mu) = \mu$, for $g\in G_{\ad}$ the unique section carrying the first pinning to the second (so in particular, $f_{\alpha_2} = \mathrm{int}_g\circ f_{\alpha_1}$). The commutativity of \eqref{eq-simply-connected-trivialization-compatibility} is clear, as is its compatibility when a third pinning is present.
\end{proof}

\begin{void}
Using the Construction of \S\ref{void-trivialization-commutative-cover-simply-connected} and the identification $\hat Z_H \cong \Lambda^{\sharp}/\Lambda^{\sharp, r}$, we find a canonical morphism:
\begin{equation}
\label{eq-construction-E-infinity-morphism}
\underline{\Gamma}{}_e(BG, B^4A(1))_Q \rightarrow \underline{\Maps}{}_{\mathbb E_{\infty}}(\hat Z_H, B^2A).
\end{equation}

We set $F_{\mu}$ to be the image of $\mu$ under \eqref{eq-construction-E-infinity-morphism}. This concludes the construction of the pair $(H, F_{\mu})$ alluded to in \S\ref{void-metaplectic-dual-context}.
\end{void}

\subsection{Splitting $F_{\mu}$}

\begin{void}
We keep the notations of the previous \S\ref{sect-metaplectic-dual-pair}. Given a split reductive group $G$ together with an \'etale metaplectic cover $\mu$, we have constructed a pair $(H, F_{\mu})$ from the combinatorial data associated to $G$ and $\mu$.

Following \cite{MR3769731}, it is possible to re-package the pair $(H, F_{\mu})$ as a triple $(H, F_{\mu}^0, \epsilon)$ by writing $F_{\mu}$ as the sum of a $\mathbb Z$-linear morphism $\hat Z_H\rightarrow B^2A$ and an $\mathbb E_{\infty}$-morphism with trivial underlying $\mathbb E_1$-monoidal structure.

Recall that $\underline{\Maps}{}_{\mathbb E_{\infty}}(\hat Z_H, B^2A)$ fits into a fiber sequence:
\begin{equation}
\label{eq-commutative-cover-character-fiber-sequence}
\underline{\Maps}{}_{\mathbb Z}(\hat Z_H, B^2A) \rightarrow \underline{\Maps}{}_{\mathbb E_{\infty}}(\hat Z_H, B^2A) \rightarrow \underline{\Maps}{}_{\mathbb Z}(\hat Z_H/2, A),
\end{equation}
where the second map sends an $\mathbb E_{\infty}$-monoidal morphism $\hat Z_H\rightarrow B^2A$ to the map $\hat Z_H/2\rightarrow A$ defined by its commutativity constraint as in \S\ref{void-self-commutativity-functor}.
\end{void}

\begin{void}
\label{void-commutative-cover-splitting}
The fiber sequence \eqref{eq-commutative-cover-character-fiber-sequence} admits a canonical splitting.
\end{void}

\begin{proof}[Construction]
Since $A$ is a subsheaf of $E^{\times}$, its subgroup $A_{[2]}$ of $2$-torsion elements belongs to $\{\pm 1\}$. There is nothing to construct if $A_{[2]} = 0$ as the first map in \eqref{eq-commutative-cover-character-fiber-sequence} becomes an isomorphism, so let us assume $A_{[2]} = \{\pm 1\}$.

In this case, any homomorphism $\hat Z_H\rightarrow A_{[2]}$ is of the form $\lambda\mapsto (-1)^{\epsilon(\lambda)}$ for a homomorphism $\epsilon : \hat Z_H \rightarrow \mathbb Z/2$. We define an $\mathbb E_{\infty}$-monoidal morphism $\hat Z_H\rightarrow B^2A$ by the associated extension of sheaves of symmetric monoidal groupoids:
$$
B(A) \rightarrow (\hat Z_H)^{\dagger} \rightarrow \hat Z_H,
$$
where $(\hat Z_H)^{\dagger} := B(A)\times \hat Z_H$ as a sheaf of monoidal groupoids, but with commutativity constraint defined by:
$$
(-1)^{\epsilon(\lambda_1)\epsilon(\lambda_2)} : \lambda_1 + \lambda_2 \cong \lambda_2 + \lambda_1,
$$
for any elements $\lambda_1,\lambda_2\in \hat Z_H$, viewed as sections of $(\hat Z_H)^{\dagger}$.
\end{proof}

\begin{rem}
The construction of \S\ref{void-commutative-cover-splitting} remains valid when $\hat Z_H$ is replaced by any sheaf of abelian groups. Recall that for the constant abelian group $\mathbb Z$, there is another splitting constructed in Remark \ref{rem-splitting-theta-data} which may be restricted to $A(-1)_{[2]}\subset A(-1)$.

These two splittings are \emph{a priori} different, since the splitting in Remark \ref{rem-splitting-theta-data} produces $\mathbb E_{\infty}$-monoidal morphisms $\mathbb Z\rightarrow B^2A$ whose underlying $\mathbb E_1$-monoidal morphism are not trivialized. They do become isomorphic if a fourth root of unity is chosen.
\end{rem}

\begin{void}
\label{void-metaplectic-dual-data-triple}
Using the splitting of \S\ref{void-commutative-cover-splitting}, we obtain from $F_{\mu}$ a $\mathbb Z$-linear morphism $F_{\mu}^0 : \hat Z_H\rightarrow B^2A$. Writing $\hat Z_H^*$ for the $\mathbb Z$-linear dual of $\hat Z_H$. It is a complex in degrees $[0,1]$:
\begin{align*}
\hat Z_H^* &\cong [\check{\Lambda}^{\sharp} \rightarrow \check{\Lambda}^{\sharp, r}] \\
&\cong [\Lambda_{T_H} \rightarrow \Lambda_{T_{H_{\ad}}}],
\end{align*}
where $\Lambda_{T_H}$ denotes the character group of the maximal torus $T_H\subset H$ and $T_{H_{\ad}}$ that of the adjoint group $H_{\ad}$.

Hence $F_{\mu}^0$ may also be viewed as a section of $\hat Z_H^*\otimes A[2]$. Inducing along the map $A\subset\overline{\mathbb Q}{}_{\ell}^{\times}$ and using the identification $\hat Z_H^*\otimes\overline{\mathbb Q}{}_{\ell}^{\times}\cong Z_H(\overline{\mathbb Q}{}_{\ell})$, we obtain an \'etale $Z_H(\overline{\mathbb Q}{}_{\ell})$-gerbe $\cal G_{Z_H(\overline{\mathbb Q}{}_{\ell})}$.

Finally, we let $\epsilon$ be the image of $F_{\mu}$ along the second map in \eqref{eq-commutative-cover-character-fiber-sequence}: it is also given by restricting the quadratic form $Q$ to $\Lambda^{\sharp}$, which is linear and factors through $\hat Z_H$, with image in $A(-1)_{[2]}\cong A_{[2]}$ (Proposition \ref{prop-commutative-cover-classification}).
\end{void}

\begin{rem}
If $S = X$ is a smooth algebraic cuver over a field $k$, our triple $(H, \cal G_{Z_H(\overline{\mathbb Q}{}_{\ell})}, \epsilon)$ is closely related to the metaplectic dual data of \cite{MR3769731}. However, we caution the reader that our $\cal G_{Z_H(\overline{\mathbb Q}{}_{\ell})}$ differs slightly from the $Z_H(\overline{\mathbb Q}{}_{\ell})$-gerbe defined in \emph{op.cit.}.

More precisely, there is a canonical $Z_H(\overline{\mathbb Q}{}_{\ell})$-gerbe $\omega_{X/k}^{\epsilon}$ on $X$ defined by inducing the $\{\pm 1\}$-gerbe $\omega_{X/k}^{1/2}$ of square roots of the canonical sheaf along $\epsilon$, viewed as a map $\{\pm 1\} \rightarrow Z_H(\overline{\mathbb Q}{}_{\ell})$. Then the $Z_H(\overline{\mathbb Q}{}_{\ell})$-gerbe of \emph{op.cit.}~is equivalent to our $\cal G_{Z_H(\overline{\mathbb Q}{}_{\ell})}$ tensored with $\omega_{X/k}^{\epsilon}$.
\end{rem}

\subsection{The dual category $\underline{\Rep}{}_{H, F_{\mu}}$}
\label{sect-dual-category}

\begin{void}
We use the term \emph{tensor category} to refer to a symmetric monoidal $E$-linear abelian category, and \emph{tensor functor} to refer to a symmetric monoidal $E$-linear additive functor between them.

Denote by $\Vect_E^f$ the tensor category of finite-dimensional $E$-vector spaces.
\end{void}

\begin{void}
\label{void-representaitons-on-lisse-sheaves}
Consider the stack of tensor categories $\Lis_E$ over $S_{\et}$ whose value at $S_1\rightarrow S$ is the category of lisse (=locally constant constructible) sheaves of $E$-vector spaces. As usual, this category is bootstrapped from lisse $\cal O_E/\fr m^n$-modules for $n\ge 1$:
$$
\Lis_E(S_1) := (\lim_{n} \Lis_{\cal O_E/\fr m^n}(S_1))[\frac{1}{\ell}],
$$
where we invert $\ell$ on the Hom-modules.

The category $\Lis_E(S_1)$ is tensored over $\Vect_E^f$. Its ind-completion $\Ind(\Lis_E(S_1))$ is tensored over the category $\Vect_E$ of all $E$-vector spaces. In particular, for any coalgebra in $\Vect_E$, one may consider its comodules in the category $\Ind(\Lis_E(S_1))$.

Given an affine group scheme $H$ over $E$, an \emph{$H$-representation} in $\Lis_E(S_1)$ is an object $V\in\Lis_E(S_1)$ equipped with the structure of an $\cal O_H$-comodule. Denote the category they form by:
$$
\underline{\Rep}{}_H(S_1) := \Comod_{\cal O_H}(\Lis_E(S_1)).
$$
Then $\underline{\Rep}{}_H$ is itself a stack of tensor categories over $S_{\et}$.
\end{void}

\begin{void}
Let $\Gamma$ be a finitely generated abelian group and $\hat{\Gamma}$ be its Cartier dual group scheme over $\Spec(E)$. Then there is a canonical equivalence:
$$
\underline{\Rep}{}_{\hat{\Gamma}}(S_1) \cong \bigoplus_{\lambda\in\Gamma} \Lis_E(S_1),
$$
where the copy of $\Lis_E(S_1)$ corresponding to $\lambda\in\Gamma$ has $\hat{\Gamma}$-action of weight $\lambda$.

This is a classical fact if $S_1$ is the spectrum of a separably closed field (\cite[I, Proposition 4.7.3]{SGA3}). The proof is unchanged in our setting: $\cal O_{\hat{\Gamma}} \cong E[\Gamma]$ and $V\in\underline{\Rep}_{\hat{\Gamma}}(S_1)$ decomposes according to the image of the coaction map $V \rightarrow V\otimes E[\Gamma] \cong\bigoplus_{\lambda\in\Gamma}V$.
\end{void}

\begin{void}
Let $H$ be a split reductive group scheme over $\Spec(E)$ with center $Z_H$. Then we have a direct sum decomposition via restriction to the $Z_H$-action:
\begin{equation}
\label{eq-metaplectic-dual-category-decomposition-untwisted}
\underline{\Rep}{}_H(S_1) \cong \bigoplus_{\lambda\in\hat Z_H} \underline{\Rep}{}_H^{\lambda}(S_1).
\end{equation}
Indeed, this is because $H$-action fixes $Z_H$-weights. By varying $S_1$, we obtain a direct sum decomposition of the stack $\underline{\Rep}{}_H$.
\end{void}

\begin{void}
The constant sheaf of abelian groups $\underline E^{\times}$ acts multiplicatively on $\id_{\Rep_{H, S}}$ (\S\ref{void-abelian-group-acting-on-category}). In particular, given an $\mathbb E_{\infty}$-monoidal morphism $F : \hat Z_H \rightarrow B^{2}(\underline E^{\times})$ over $S$, we obtain another stack of tensor categories $\underline{\Rep}{}_{H, F}$ by the twisting construction of \S\ref{void-twisting-construction-symmetric-monoidal}.

It inherits a decomposition from \eqref{eq-metaplectic-dual-category-decomposition-untwisted}:
\begin{equation}
\label{eq-metaplectic-dual-category-decomposition}
\underline{\Rep}{}_{H, F} \cong \bigoplus_{\lambda\in \hat Z_H} \underline{\Rep}{}_{H, F(\lambda)}^{\lambda},
\end{equation}
where the summands are stacks of abelian categories $\underline{\Rep}{}_H^{\lambda}$ twisted by $F(\lambda)$.
\end{void}

\begin{void}
\label{void-metaplectic-dual-category-construction}
Let $G\rightarrow S$ be a reductive group scheme equipped with a pointed morphism $\mu : BG \rightarrow B^4A(1)$. Over an \'etale cover $S_1\rightarrow S$ splitting $G$, we obtain from \S\ref{sect-metaplectic-dual-pair} a pair $(H, F_{\mu})$ where $H \rightarrow \Spec(\mathbb Z)$ is a pinned reductive group and $F_{\mu} : \hat Z_H \rightarrow B^2(A)$ is an $\mathbb E_{\infty}$-monoidal morphism. Using the maps $\mathbb Z\rightarrow E$, $A\subset\underline E^{\times}$, we may form a stack of tensor categories $\underline{\Rep}{}_{H, F_{\mu}}$ over $S_1$.

Since $\underline{\Rep}{}_{H, F_{\mu}}$ is functorially attached to the pair $(G, \mu)$ and stacks of tensor categories are \'etale local objects, we find a stack of tensor categories over $S$, to be denoted by the same notation:
\begin{equation}
\label{eq-metaplectic-dual-category}
(G,\mu) \mapsto \underline{\Rep}{}_{H, F_{\mu}}.
\end{equation}
\end{void}

\begin{rem}
\label{rem-metaplectic-dual-category-tannakian-reconstruction}
The tensor category $\underline{\Rep}{}_{H, F_{\mu}}(S)$ satisfies the following conditions:
\begin{enumerate}
	\item every object $V\in\underline{\Rep}{}_{H, F_{\mu}}(S)$ admits a dual;
	\item the natural map $E \rightarrow \End(\mathbf 1)$, for $\mathbf 1\in\underline{\Rep}{}_{H, F_{\mu}}(S)$ the monoidal unit, is an isomorphism.
\end{enumerate}
Indeed, both statements may be verified \'etale locally on $S$. Statement (2) follows from the fact that $\mathbf 1$ belongs to the zero-weight part in \eqref{eq-metaplectic-dual-category-decomposition}, where the category is untwisted.

In other words, $\underline{\Rep}{}_{H, F_{\mu}}(S)$ is a \emph{cat\'egorie tensorielle} in the sense of \cite[\S1.2]{MR1106898}. If it admits a fiber functor, then we may identify $\underline{\Rep}{}_{H, F_{\mu}}(S)$ with the category of representations of an affine group scheme \cite[Th\'eor\`eme 1.12]{MR1106898}.
\end{rem}

\begin{rem}
\label{rem-metaplectic-dual-category-lisse-module}
Note that $\underline{\Rep}{}_{H, F_{\mu}}$ is naturally tensored over $\Lis_E$. \'Etale locally on $S$, this $\Lis_E$-action preserves the decomposition \eqref{eq-metaplectic-dual-category-decomposition}. Moreover, acting on the monoidal unit $\mathbf 1 \in \underline{\Rep}{}_{H, F_{\mu}}$ induces a \emph{symmetric} monoidal functor:
\begin{equation}
\label{eq-metaplectic-dual-category-trivial-representation}
\Lis_E \rightarrow \underline{\Rep}{}_{H, F_{\mu}}.
\end{equation}
It encodes the operation of viewing a lisse sheaf as a trivial $F_{\mu}$-twisted $H$-representation.
\end{rem}

\begin{void}
\label{void-center-gerbe-nonsplit}
For nonsplit $G$, it is possible to assign intrinsic meanings to $H$ and $F_{\mu}$ in the notation \eqref{eq-metaplectic-dual-category} as follows:
\begin{enumerate}
	\item $H$ is a locally constant sheaf (over $S$) of pinned reductive groups (over $\Spec(\mathbb Z)$);
	\item $\hat Z_H$ is a locally constant sheaf of abelian groups and:
	$$
	F_{\mu} : \hat Z_H \rightarrow B^2(A)
	$$
	is an $\mathbb E_{\infty}$-monoidal morphism.
\end{enumerate}
There is always a faithful exact tensor functor $\underline{\Rep}{}_{H, F_{\mu}} \rightarrow \underline{\Rep}{}_{Z_H, F_{\mu}}$.
\end{void}

\begin{void}
Due to the commutativity constraint of $F_{\mu}$, it is unreasonable to expect there to be fiber functors out of $\underline{\Rep}{}_{H, F_{\mu}}(S)$ in general. In order to define $\tn L$-parameters, we replace $F_{\mu}$ by its $\mathbb Z$-linear component $F_{\mu}^0$ in the formation of $\underline{\Rep}{}_{H, F_{\mu}}$.

Recall from \S\ref{void-metaplectic-dual-data-triple} that $F_{\mu}^0 : \hat Z_H \rightarrow B^2(A)$ is a $\mathbb Z$-linear morphism functorially attached to $F_{\mu}$ (hence to $(G, \mu)$). Repeating the construction in \S\ref{void-metaplectic-dual-category-construction} gives an assignment:
\begin{equation}
\label{eq-metaplectic-dual-category-linearized}
(G, \mu) \mapsto \underline{\Rep}{}_{H, F_{\mu}^0}.
\end{equation}
Thus $\underline{\Rep}{}_{H, F_{\mu}^0}$ is canonically equivalent to $\underline{\Rep}{}_{H, F_{\mu}}$ as stacks of \emph{monoidal} categories, but the commutativity constraint is modified by the values of $\epsilon$ on $Z_H$-weights.

Note that $\underline{\Rep}{}_{H, F_{\mu}^0}$ is naturally tensored over $\Lis_E$ like its sister $\underline{\Rep}{}_{H, F_{\mu}}$ (Remark \ref{rem-metaplectic-dual-category-lisse-module}).
\end{void}

\begin{void}
A functor $F : \cal C_1 \rightarrow \cal C_2$ of categories tensored over a monoidal category $\cal A$ is said to \emph{commute} with the $\cal A$-action if it is equipped with a natural isomorphism:
\begin{equation}
\label{eq-commutation-with-action}
F(a\otimes c)\cong a\otimes F(c),\quad a\in\cal A, c\in\cal C_1,
\end{equation}
compatible with the monoidal structure of $\cal A$.

If $F : \cal C_1\rightarrow \cal C_2$ is a tensor functor between tensor categories and the $\cal A$-actions on them come from tensor functors $\varphi_1 : \cal A \rightarrow \cal C_1$, $\varphi_2 : \cal A\rightarrow\cal C_2$, then the datum \eqref{eq-commutation-with-action} is equivalent to an isomorphism of tensor functors $F\circ\varphi_1\cong\varphi_2$.
\end{void}

\begin{defn}
\label{defn-tannakian-L-parameter}
A \emph{Tannakian $\tn L$-parameter} is a faithful exact tensor functor:
$$
\underline{\Rep}{}_{H, F_{\mu}^0}(S) \rightarrow \Lis_E(S)
$$
commuting with the $\Lis_E(S)$-actions.
\end{defn}

\begin{rem}
\label{rem-tannakian-L-parameter-split}
Let us illustrate the idea behind Tannakian $\tn L$-parameters using the example of a split reductive group $G$ over a field $F$ with no metaplectic cover. We fix a geometric point $\bar s\rightarrow S := \Spec(F)$.

In this case, the Langlands dual group is the pinned split reductive group $\check G \rightarrow \Spec(E)$ and a classical $\tn L$-parameter is a morphism $\pi_1(S, \bar s) \rightarrow \check G(E)$ of topological groups---the source is equipped with the profinite topology and the target the $\ell$-adic topology---defined up to $\check G(E)$-conjugation.

There is a natural functor between groupoids:
\begin{equation}
\label{eq-split-reductive-group-tannakian-L-parameter}
\Hom(\pi_1(S, \bar s), \check G(E))/\check G(E) \rightarrow \Hom^{\otimes}(\Rep_{\check G}, \Lis_E(S)),
\end{equation}
where the target consists of faithful exact tensor functors $\Rep_{\check G} \rightarrow \Lis_E(S)$, where $\Lis_E(S)$ is identified with continuous $\pi_1(S, \bar s)$-representations on a finite-dimensional $E$-vector space.

\emph{Claim}: \eqref{eq-split-reductive-group-tannakian-L-parameter} is fully faithful. Indeed, it suffices to prove that $\Hom(\pi_1(S, \bar s), \check G(E))$ maps bijectively to the set of such functors $\Rep_{\check G} \rightarrow \Lis_E(S)$ equipped with a rigidification along $\bar s$, i.e.~a commutative diagram
$$
\begin{tikzcd}[column sep = 1.5em, row sep = 1.5em]
	\Rep_{\check G} \ar[r]\ar[dr, swap, "e^*"] & \Lis_E(S) \ar[d, "\bar s^*"] \\
	& \Vect^f_E
\end{tikzcd}
$$
where $e^*$ is the forgetful functor (pullback along $e : \Spec(E) \rightarrow B\check G$).

This follows from the fact that continuous homomorphisms $\pi_1(S, \bar s) \rightarrow \check G(E)$ are identified with \emph{algebraic} homomorphisms $\pi_1^{\alg}(S, \bar s) \rightarrow \check G$, where the pro-algebraic completion $\pi_1^{\alg}(S, \bar s)$ agrees with the automorphism group scheme of $\bar s^*$, c.f.~\cite[\S1.2]{MR2530856}.

The proof also shows that the essential image of \eqref{eq-split-reductive-group-tannakian-L-parameter} consists of functors $T : \Rep_{\check G} \rightarrow \Lis_E(S)$ such that $\bar s^*\circ T$ is isomorphic to $e^*$ as tensor functors. The $E$-scheme parametrizing such isomorphisms $\bar s^*\circ T\cong e^*$ form a $\check G$-torsor over $\Spec(E)$, for which a section exists over a finite extension of $E$. In particular, taking colimit of \eqref{eq-split-reductive-group-tannakian-L-parameter} over finite extensions of $\mathbb Q_{\ell}$ defines an equivalence between:
\begin{enumerate}
	\item continuous homomorphisms $\pi_1(S, \bar s) \rightarrow \check G(\overline{\mathbb Q}_{\ell})$ factoring through a finite extension of $\mathbb Q_{\ell}$, modulo $\check G(\overline{\mathbb Q}_{\ell})$-conjugation; and
	\item faithful exact tensor functors $\Rep_{\check G, \overline{\mathbb Q}_{\ell}} \rightarrow \Lis_{\overline{\mathbb Q}_{\ell}}(S)$ defined over a finite extension of $\mathbb Q_{\ell}$. (Here, $\Rep_{\check G, \overline{\mathbb Q}_{\ell}}$ stands for the category of finite-dimensionl algebraic representations over $\overline{\mathbb Q}_{\ell}$; see \cite[\S4.6]{MR1012168} for extending scalars of a Tannakian category.)
\end{enumerate}

For $G$ nonsplit and equipped with an \'etale metaplectic cover $\mu$, the ``variance along $S$'' is encoded in the definition of $\underline{\Rep}{}_{H, F_{\mu}^0}$ together with its $\Lis_E$-module structure. We shall make a similar comparison between Definition \ref{defn-tannakian-L-parameter} and an $\tn L$-parameter of classical flavor in the next subsection.
\end{rem}

\subsection{The classical $\tn L$-group}
\label{sect-classical-L-group}

\begin{void}
\label{void-second-twist-conditions-on-base}
When $S$ is the spectrum of a field or a discrete valuation ring, it is possible to obtain an $\tn L$-group in the style of Langlands, i.e.~an extension of the \'etale fundamental group of $S$ by $E$-points of a split reductive group. The construction explained below paraphrases \cite[\S5 \& 19]{MR3802418}.

To be precise about our hypotheses, we shall consider a connected scheme $S$ satisfying the following conditions:
\begin{enumerate}
	\item every reductive group scheme $G\rightarrow S$ splits over a finite \'etale cover of $S$;
	\item any gerbe banded by a sheaf of finite abelian groups is trivial over a finite \'etale cover of $S$ (e.g.~$S$ is a $\tn K(\pi, 1)$-scheme).
\end{enumerate}

In what follows, we fix a geometric point $\bar s\rightarrow S$ and write $\pi_1(S, \bar s)$ for the (pro-finite) \'etale fundamental group of $S$ with base point $\bar s$. 

In the construction of \eqref{eq-weissman-L-group-center} below, we will need $E$ to be large enough relative to $(G, \mu)$; see \emph{loc.cit.}~for the precise meaning of this condition.
\end{void}

\begin{void}
Suppose that $G\rightarrow S$ is a reductive group scheme and $\mu : BG \rightarrow B^4A(1)$ is a pointed morphism.

The base change $G_{\bar s} := G\times_S\bar s$ being split, we find a pinned reductive group scheme $H \rightarrow \Spec(E)$. Let $T_H$ denote its maximal torus and $T_{H_{\ad}}$ the induced maximal torus of the adjoint group $H_{\ad}$.

Since $G\rightarrow S$ splits over a finite \'etale cover and the construction of $H$ is functorial, we obtain a ``continuous'' $\pi_1(S, \bar s)$-action on $H$ preserving the pinning. Here, ``continuity'' means that the action factors through a finite quotient of $\pi_1(S, \bar s)$.

In particular, the character group $\hat Z_H$ of the center of $H$ is a $\pi_1(S, \bar s)$-module. Its $\mathbb Z$-linear dual is the complex:
\begin{equation}
\label{eq-metaplectic-dual-group-dual-character-center}
\hat Z_H^* \cong [\Lambda_{T_H} \rightarrow \Lambda_{T_{H_{\ad}}}]
\end{equation}
of $\pi_1(S, \bar s)$-modules in degrees $[0, 1]$.
\end{void}

\begin{void}
According to \S\ref{void-center-gerbe-nonsplit}, we obtain an \'etale sheaf of abelian groups $\hat Z$ over $S$ together with a section $F_{\mu}^0$ of $\hat Z^*\otimes A[2]$ over $S$.

Recall the equivalence between \'etale sheaves of finite abelian groups on $S$ and finite $\pi_1(S, \bar s)$-modules, given by taking stalks $K\mapsto K_{\bar s}$ at $\bar s$.

Under this equivalence, $\hat Z$ passes to the $\pi_1(S, \bar s)$-module $\hat Z_H$. The presentation \eqref{eq-metaplectic-dual-group-dual-character-center} corresponds to a presentation of the dual $\hat Z^*$ as a complex of \'etale sheaves.
\end{void}

\begin{void}
Fix a trivialization of $F_{\mu}^0$ along $\bar s$. We shall construct a short exact sequence of topological groups:
\begin{equation}
\label{eq-weissman-L-group}
1 \rightarrow H(E) \rightarrow {}^LH_S \rightarrow \pi_1(S, \bar s) \rightarrow 1,
\end{equation}
which is induced from a finite quotient of $\pi_1(S, \bar s)$.
\end{void}

\begin{proof}[Construction]
The passage from $F_{\mu}^0$ to an extension of $\pi_1(S, \bar s)$ by $Z_H(E)$ arises from the standard correspondence between \'etale cochains and Galois cochains.

More precisely, suppose that $K$ is an \'etale sheaf of finite abelian groups on $S$. Consider the groupoid of extensions:
\begin{equation}
\label{eq-central-extension-fundamental-group}
1 \rightarrow K_{\bar s} \rightarrow E \rightarrow \pi_1(S, \bar s) \rightarrow 1
\end{equation}
which are induced from a finite quotient of $\pi_1(S, \bar s)$ and such that $E$-conjugation on $K_{\bar s}$ factors through the $\pi_1(S, \bar s)$-action.

By \'etale descent, there is a functor from such extensions to the groupoid of \'etale $K$-gerbes on $S$ rigidified along $\bar s$. It is fully faithful with essential image being those \'etale $K$-gerbes trivialized over a \emph{finite} \'etale cover of $S$. By our assumption on $S$ in \S\ref{void-second-twist-conditions-on-base}, this is in fact an equivalence of groupoids.

Since $F_{\mu}^0$ is rigidified along $\bar s$, we obtain from this equivalence an extension of $\pi_1(S, \bar s)$ by $\Lambda_{T_H}\otimes A_{\bar s}$ whose induced extension by $\Lambda_{T_{H_{\ad}}}\otimes A$ is equipped with a splitting.

Suppose $E$ is sufficiently large so that $\Lambda_{T_H}\otimes E^{\times} \rightarrow \Lambda_{T_{H_{\ad}}}\otimes E^{\times}$ is \emph{surjective}. Using the inclusion $A\subset\underline E^{\times}$, we obtain an extension of topological groups:
\begin{equation}
\label{eq-weissman-L-group-center}
1 \rightarrow Z_H(E) \rightarrow {}^L(Z_H)_S \rightarrow \pi_1(S, \bar s) \rightarrow 1,
\end{equation}
which is induced from a finite quotient of $\pi_1(S, \bar s)$ and such that the ${}^L(Z_H)_S$-conjugation action on $Z_H(E)$ factors through the natural $\pi_1(S, \bar s)$-action.

Finally, the extension \eqref{eq-weissman-L-group} is formed out of \eqref{eq-weissman-L-group-center} by inducing along the $\pi_1(S, \bar s)$-equivariant map $Z_H(E)\subset H(E)$, namely:
\begin{equation}
\label{eq-weissman-L-group-formation}
{}^LH_S := (H(E) \rtimes {}^L(Z_H)_S)/Z_H(E),
\end{equation}
where the ${}^L(Z_H)_S$-action on $H(E)$ factors through $\pi_1(S,\bar s)$ and the embedding of $Z_H(E)$ is the anti-diagonal one.
\end{proof}

\begin{rem}
The extension \eqref{eq-weissman-L-group} may be viewed as the $\tn L$-group of $(G, \mu)$. Contrary to the $\tn L$-group of reductive group schemes, it may \emph{not} be a semi-direct product.

If $\mu$ is obtained from a central extension of $G$ by $\underline K_2$ (see \S\ref{sect-relation-with-K2}), then our definition of ${}^LH_S$ agrees with Weissman's ``second twist''. When $S$ is the spectrum a local or global field, incorporating the ``first twist'' amounts to replacing \eqref{eq-weissman-L-group-center} by its Baer sum with the extension of $\pi_1(S, \bar s)$ by $Z_H(E)$ given by inducing the meta-Galois group:
$$
1 \rightarrow \{\pm 1\} \rightarrow \tilde{\pi}_1(S, \bar s) \rightarrow \pi_1(S, \bar s)\rightarrow 1
$$
along the map $\epsilon : \{\pm 1\} \rightarrow Z_H(E)$ of \S\ref{void-metaplectic-dual-data-triple}.
\end{rem}

\begin{rem}
Suppose that $\mu$ is the pullback along the map $G\rightarrow \pi_1(G)\otimes\mathbb G_m$, or equivalently its associated quadratic form vanishes, see Proposition \ref{prop-reductive-classification}. In this case, the construction of $F_{\mu}^0$ simplifies significantly and so does the construction of the $\tn L$-group.

Indeed, since $Q = 0$, we have $H$ being the usual Langlands dual group $\check G$, equipped with the natural $\pi_1(S, \bar s)$-action. The \'etale sheaf $\hat Z$ corresponding to the character group of its center $Z_{\check G}$ is naturally isomorphic to $\pi_1(G)$.

On the other hand, Proposition \ref{prop-reductive-classification} defines a $\mathbb Z$-linear morphism $\pi_1(G)\rightarrow B^2A$, hence a $\mathbb Z$-linear morphism $\hat Z \rightarrow B^2(E)$. If $E$ is sufficiently large (e.g.~$E = \overline{\mathbb Q}{}_{\ell}$), we obtain an extension of $\pi_1(S, \bar s)$ by $Z_{\check G}(E)$ as in \eqref{eq-weissman-L-group-center}, and thus the $\tn L$-group by the formation \eqref{eq-weissman-L-group-formation}.
\end{rem}

\begin{void}
The extension \eqref{eq-weissman-L-group} gives an alternative description of global sections of the stack of tensor categories $\underline{\Rep}{}_{H, F_{\mu}^0}$.

To be precise, we write $\Rep_{{}^LH_S}^{\alg}$ for the category of continuous ${}^LH_S$-representations on finite-dimensional $E$-vector spaces, such that the action is algebraic on $H(E)$, i.e.~the induced $H(E)$-action comes from an algebraic $H$-representation.

If ${}^LH_S = H(E)\times \pi_1(S, \bar s)$, then $\Rep_{{}^LH_S}^{\alg}$ is identified with the category of $H$-representations on lisse sheaves over $S$, i.e.~the category $\underline{\Rep}{}_H(S)$ introduced in \S\ref{void-representaitons-on-lisse-sheaves}.
\end{void}

\begin{prop}
\label{prop-metaplectic-dual-category-classical-description}
There is an equivalence of tensor categories:
\begin{equation}
\label{eq-dual-category-as-representations-of-L-group}
\underline{\Rep}{}_{H, F_{\mu}^0}(S) \cong \Rep_{{}^LH_S}^{\alg}.
\end{equation}
\end{prop}
\begin{proof}
Fix a finite \'etale Galois cover $S_1 \rightarrow S$ with structure group $\Gamma$ and a lift $\bar s_1$ of $\bar s$ such that the following statements hold:
\begin{enumerate}
	\item the $\pi_1(S, \bar s)$-action on $H$ factors through $\Gamma$;
	\item the restriction of $F_{\mu}^0$ to $S_1$, viewed as a section of the constant \'etale sheaf $\hat Z_H^*\otimes A[2]$, is trivial.
\end{enumerate}
We furthermore fix a trivialization of $F_{\mu}^0$ over $S_1$ extending the given trivialization over $\bar s_1$.

This trivialization induces an equivalence of tensor categories:
\begin{equation}
\label{eq-dual-category-as-representations-of-L-group-local}
\underline{\Rep}{}_{H, F_{\mu}^0}(S_1) \cong \Rep_{H(E) \times \pi_1(S_1, \bar s_1)}^{\alg},
\end{equation}
as both sides are equivalent to $H$-representations on lisse sheaves over $S_1$. We shall obtain \eqref{eq-dual-category-as-representations-of-L-group} as the equivalence of $\Gamma$-equivariant objects of \eqref{eq-dual-category-as-representations-of-L-group-local}, for naturally defined $\Gamma$-actions as endofunctors on both tensor categories.

The $\Gamma$-action on $\underline{\Rep}{}_{H, F_{\mu}^0}(S_1)$ comes from the descent data for $F_{\mu}^0$ and the $\Gamma$-action on $H$. The fact that $\underline{\Rep}{}_{H, F_{\mu}^0}(S)$ is identified with $\Gamma$-equivariant objects in $\underline{\Rep}{}_{H, F_{\mu}^0}(S_1)$ follows from the fact that $\underline{\Rep}{}_{H, F_{\mu}^0}$ is an \'etale stack.

The $\Gamma$-action on $\Rep^{\alg}_{H(E)\times\pi_1(S_1, \bar s_1)}$ comes from the short exact sequence:
$$
1 \rightarrow H(E)\times\pi_1(S_1, \bar s_1) \rightarrow {}^LH_S \rightarrow \Gamma \rightarrow 1.
$$
Indeed, ${}^LH_S$-conjugation on $\Rep_{H(E)\times\pi_1(S_1, \bar s_1)}^{\alg}$ factors through $\Gamma$ since inner automorphisms induce the identity functor on the category of representations. It follows from general principles that $\Rep_{{}^LH_S}^{\alg}$ is identified with $\Gamma$-equivariant objects in $\Rep^{\alg}_{H(E)\times\pi_1(S_1, \bar s_1)}$.

We omit the verification that \eqref{eq-dual-category-as-representations-of-L-group-local} is compatible with these two $\Gamma$-actions, which follows routinely from the construction of ${}^LH_S$.
\end{proof}

\begin{rem}
Under the tensor equivalence \eqref{eq-dual-category-as-representations-of-L-group}, restriction along $\bar s\rightarrow S$ on the left-hand-side corresponds to restriction along $H(E)\subset {}^LH_S$ on the right-hand-side. Namely, we have a commutative diagram:
$$
\begin{tikzcd}[column sep = 1.5em]
	\underline{\Rep}{}_{H, F_{\mu}^0}(S) \ar[d, "\bar s\rightarrow S"]\ar[r, "\cong"] & \Rep_{{}^LH_S}^{\alg} \ar[d, "H(E)\subset {}^LH_S"] \\
	\underline{\Rep}{}_{H, F_{\mu}^0}(\bar s) \ar[r, "\cong"] & \Rep_H
\end{tikzcd}
$$
where the bottom equivalence comes from the rigidification of $F_{\mu}^0$ along $\bar s$.

Under the equivalence \eqref{eq-dual-category-as-representations-of-L-group}, the $\Lis_E(S)$-action on $\underline{\Rep}{}_{H, F_{\mu}^0}(S)$ corresponds to tensoring with ${}^LH_S$-representations restricted along the map ${}^LH_S \rightarrow \pi_1(S, \bar s)$.
\end{rem}

\begin{void}
Write $\Hom_{/\pi_1(S,\bar s)}(\pi_1(S,\bar s), {}^LH_S)$ for the set of continuous sections of \eqref{eq-weissman-L-group}:
$$
\sigma : \pi_1(S, \bar s) \rightarrow {}^LH_S.
$$
From $\sigma$, we obtain a faithful exact tensor functor $T := \sigma^*$ rigidified along $\bar s$:
\begin{equation}
\label{eq-rigidification-tannakian-L-parameter}
\begin{tikzcd}[column sep = 1.5em]
\Rep_{{}^LH_S}^{\alg} \ar[r, "T"]\ar[dr, swap, "e^*"] & \Lis_E(S) \ar[d, "\bar s^*"] \\
& \Vect_E^f
\end{tikzcd}
\end{equation}
where $e^*$ stands for the tautological forgetful functor.

The fact that $\sigma$ is a section implies that $T$ commutes with natural $\Lis_E(S)$-actions and the isomorphism $\bar s^*\circ T\cong e^*$ exhibited in \eqref{eq-rigidification-tannakian-L-parameter} is compatible with $\Lis_E(S)$-actions (where $\Lis_E(S)$ acts on $\Vect_E^f$ via $\bar s^*$), i.e.~$T$ is \emph{$\Lis_E(S)$-linearly rigidified} along $\bar s$.

Under the equivalence of Proposition \ref{prop-metaplectic-dual-category-classical-description}, we obtain a Tannakian $\tn L$-parameter $T$ which is $\Lis_E(S)$-linearly rigidified along $\bar s$, see Definition \ref{defn-tannakian-L-parameter}.

The $H(E)$-conjugation action on $\Hom_{/\pi_1(S,\bar s)}(\pi_1(S,\bar s), {}^LH_S)$ corresponds to modifying the isomorphism $\bar s^*\circ T\cong e^*$ by a $\Lis_E(S)$-linear automorphism of $e^*$. In particular, we obtain a functor of groupoids:
\begin{equation}
\label{eq-metaplectic-L-parameter-classical-comparison}
	\Hom_{/\pi_1(S,\bar s)}(\pi_1(S,\bar s), {}^LH_S)/H(E) \rightarrow \Hom^{\otimes}_{\Lis_E(S)}(\underline{\Rep}{}_{H, F_{\mu}^0}(S), \Lis_E(S)),
\end{equation}
where the target is the groupoid of Tannakian $\tn L$-parameters.

The following fact is the ``$\Lis_E(S)$-linear version'' of Remark \ref{rem-tannakian-L-parameter-split}.
\end{void}

\begin{prop}
The functor \eqref{eq-metaplectic-L-parameter-classical-comparison} is fully faithful. It induces an equivalence upon taking colimit over finite extensions $\mathbb Q_{\ell}\subset E$ contained in $\overline{\mathbb Q}_{\ell}$.
\end{prop}
\begin{proof}
The assertions follow from the statements below:
\begin{enumerate}
	\item $\Hom_{/\pi_1(S,\bar s)}(\pi_1(S,\bar s), {}^LH_S)$ maps bijectively to $\Lis_E(S)$-linearly rigidified Tannakian $\tn L$-parameters;
	\item $\Lis_E(S)$-linear rigidifications of a Tannakian $\tn L$-parameter along $\bar s$ form an $H$-torsor over $\Spec(E)$. (In particular, it splits over a finite extension of $E$.)
\end{enumerate}

To prove these statements, it is convenient to use an algebraic version of the $\tn L$-group \eqref{eq-weissman-L-group}. Namely, there is a short exact sequence of affine group schemes over $\Spec(E)$:
\begin{equation}
\label{eq-algebraic-L-group}
1 \rightarrow H \rightarrow {}^LH_S^{\alg} \rightarrow \pi_1^{\alg}(S, \bar s) \rightarrow 1,
\end{equation}
where ${}^LH_S^{\alg}$ is the automorphism group scheme of the fiber functor:
\begin{equation}
\label{eq-metaplectic-dual-category-fiber-functor}
\Rep_{{}^LH_S}^{\alg} \rightarrow \Rep_H \xrightarrow{e^*} \Vect^f_E.
\end{equation}
In particular, $\Rep_{{}^LH_S}^{\alg}$ (resp.~$\Lis_E(S)$) is recovered as the category of finite-dimensional algebraic representations of ${}^LH_S^{\alg}$ (resp.~$\pi_1^{\alg}(S, \bar s)$); see Remark \ref{rem-metaplectic-dual-category-tannakian-reconstruction}.

The $E$-points of \eqref{eq-algebraic-L-group} receive continuous maps from \eqref{eq-weissman-L-group}, exhibiting ${}^LH_S$ as the pullback of ${}^LH_S^{\alg}(E)$ along $\pi_1(S, \bar s) \rightarrow \pi_1^{\alg}(S, \bar s)(E)$. By the universal property of $\pi_1^{\alg}(S, \bar s)$, continuous sections of ${}^LH_S\rightarrow\pi_1(S, \bar s)$ are in bijection with algebraic sections of ${}^LH_S^{\alg} \rightarrow \pi_1^{\alg}(S, \bar s)$. Under Tannakian duality, the latter are in bijecion with faithful exact tensor functors:
\begin{equation}
\label{eq-tannakian-L-parameter-classical-description}
T : \Rep_{{}^LH_S}^{\alg} \rightarrow \Lis_E(S)
\end{equation}
equipped with a $\Lis_E(S)$-linear rigidification along $\bar s$.

Given a faithful exact tensor functor \eqref{eq-tannakian-L-parameter-classical-description}, the scheme of $\Lis_E(S)$-linear rigidifications of $\bar s^*\circ T$ is a torsor under the group scheme of $\Lis_E(S)$-linear automorphisms of the fiber functor \eqref{eq-metaplectic-dual-category-fiber-functor}. The latter is equivalent to the group scheme of automorphisms of the fiber functor $e^* : \Rep_H \rightarrow \Vect_E^f$, i.e.~$H$.
\end{proof}

\medskip

\section{The De Rham context}
\label{sec-de-Rham}

The structure theory of \'etale metaplectic covers developed in sections \S\ref{sec-cup-product}--\ref{sec-classification} is closely related to ``quantum parameters'', or ``levels'' of affine Kac--Moody Lie algebras. Indeed, it is a classical observation that the level $\kappa$ of an affine Kac--Moody Lie algebra $\hat{\fr g}^{\kappa}$ over $\mathbb C$ has a natural intepretation as a class in $H^4_{\dR}(BG, \mathbb C)$. Na\"ively, one would then take $\Gamma_{\dR, e}(BG, \mathbb C[4])$ to be the space of quantum parameters.

When the base is a smooth, proper curve $X$, a good notion of quantum parameters is supposed to induce rings of twisted differential operators on the moduli stack of $G$-bundles on $X$. For this to happen, we need to replace $\Gamma_{\dR, e}(BG, \mathbb C[4])$ by those de Rham cochains which ``belong to the Hodge filtration $F^{\ge 2}$.''

In \S\ref{sect-quantum-parameters}, we turn this idea into a precise definition and show that it recovers the usual notion of quantum parameters. In \S\ref{sect-rational-covers}, we compare the space of quantum parameters with integral metaplectic covers as studied in \S\ref{sect-relation-with-K2}.

\subsection{Quantum parameters}
\label{sect-quantum-parameters}

\begin{void}
\label{void-differential-forms}
In this section, we work over a field $k$ of characteristic zero and denote by $\Sm_{/k}$ the category of smooth $k$-schemes.

For each integer $p\ge 0$ and $S\in\Sm_{/k}$, there is a complex of \'etale sheaves of $k$-vector spaces concentrated in cohomological degrees $\ge p$:
$$
\Omega^{\ge p}_S := [\Omega_S^p \xrightarrow{d} \Omega_S^{p+1} \xrightarrow{d}\cdots],
$$
where each $\Omega_S^n$ denotes the sheaf of $n$th differential forms on $S$ relative to $k$.

The association $S\mapsto \Gamma(S, \Omega_S^{\ge p})$ is an \'etale sheaf of $k$-module spectra on $\Sm_{/k}$, with functoriality defined by pulling back differential forms.

Let $S\mapsto\QCoh(S)$ denote the functor assigning the stable $\infty$-category of quasi-coherent $\cal O_S$-modules to an affine $k$-scheme $S$. It extends to the cattegory of algebraic stacks over $k$, by the operation of right Kan extension.
\end{void}

\begin{void}
Fix $S\in\Sm_{/k}$. Let $G\rightarrow S$ be a reductive group scheme. Denote by $\fr g$ the Lie algebra of $G$, viewed as a locally free sheaf of $\cal O_S$-modules equipped with a $G$-action.

The stable $\infty$-category $\QCoh(BG)$ is identified with quasi-coherent $\cal O_S$-modules equipped with a $G$-action (i.e.~a comodule structure over $\cal O_G$). Under this identification, the cotangent complex $L_{BG/S}$ corresponds to $\fr g^*[-1]$, where $\fr g^*$ denotes the $\cal O_S$-linear dual of $\fr g$ equipped with the co-adjoint $G$-action.
\end{void}

\begin{void}
Define:
$$
\underline{\Gamma}(BG, \Omega_{BG}^{\ge p}) := \lim_{[n]} \underline{\Gamma}(G^{\times n}, \Omega^{\ge p}_{G^{\times n}})
$$
as an \'etale sheaf of $k$-module spectra on $S$.

We also write $\underline{\Gamma}{}_e(BG, \Omega_{BG}^{\ge p})$ for the rigidified version, i.e.~the fiber of the morphism of complexes $e^* : \underline{\Gamma}(BG, \Omega_{BG}^{\ge p}) \rightarrow \Omega_S^{\ge p}$.

We define \emph{quantum parameters} to be sections of $\tau^{\le 0}\underline{\Gamma}{}_e(BG, \Omega_{BG}^{\ge 2}[4])$, i.e.~rigidified section of $\Omega_{BG}^{\ge 2}[4]$ over $BG$. The name is justified in Remark \ref{rem-quantum-parameters} below.
\end{void}

\begin{void}
Denote by $\Omega_S^{p, \cl}\subset\Omega_S^p$ the subsheaf of closed $p$-forms. We shall encounter the following complex in degrees $[-1, 0]$:
\begin{equation}
\label{eq-TDO-complex}
[\Omega_S^1 \xrightarrow{d} \Omega_S^{2,\cl}].
\end{equation}

As a sheaf of Picard groupoids, sections of \eqref{eq-TDO-complex} consist of $\Omega_S^1$-torsors whose induced $\Omega_S^{2,\cl}$-torsor is equipped with a trivialization. Such objects are equivalent to rings of twisted differential operators on $S$, see \cite[Lemma 2.1.6]{MR1237825}.

The following calculation is an analogue of Proposition \ref{prop-reductive-classification}.
\end{void}

\begin{prop}
\label{prop-de-rham-triangle}
There is a canonical triangle of complexes of sheaves of $k$-vector spaces:
\begin{equation}
\label{eq-de-rham-triangle}
(\fr g^*)^G \otimes [\Omega_S^1\xrightarrow{d}\Omega_S^{2,\cl}] \rightarrow \tau^{\le 0}\underline{\Gamma}{}_e(BG, \Omega_{BG}^{\ge 2}[4]) \rightarrow \Sym^2(\fr g^*)^G.
\end{equation}
\end{prop}

\begin{rem}
The invariants $(\fr g^*)^G$, $\Sym^2(\fr g^*)^G$ are \emph{a priori} derived, i.e.~given as the image of the corresponding objects under $\pi_* : \QCoh(BG) \rightarrow \QCoh(S)$. They are in fact concentrated in cohomological degree $0$.

To prove this assertion, we may perform an \'etale base change on $S$ and assume that $G$ is split. In particular, $G = G_0\times S$ for a reductive group $G_0$ over $k$. Then $\fr g^*$ (resp.~$\Sym^2(\fr g^*)$) is the pullback of the quasi-coherent sheaf $\fr g_0^*$ (resp.~$\Sym^2(\fr g_0^*)$) over $BG_0$. So it remains to show that the direct image along $BG_0 \rightarrow \Spec(k)$ is of cohomological amplitude $\le 0$ and commutes with arbitrary base change.

Identifying $\QCoh(BG_0)$ with the derived category of algebraic $G_0$-representations over $k$, the first statement follows from the fact that reductive groups over a field of characteristic zero are linearly reductive, and the second statement follows from the first by \cite[Corollary B.16]{halpern2014mapping}.

In particular, writing $\fr g_{\ab}$ for the Lie algebra of the maximal quotient torus $G_{\ab} := G/G_{\der}$, the natural map $\fr g\rightarrow \fr g_{\ab}$ defines an isomorphism $\fr g^*_{\ab}\cong (\fr g^*)^G$.
\end{rem}

\begin{void}
Using flat descent of cotangent complexes, we see that $\underline{\Gamma}(BG, \Omega_{BG}^{\ge 0})$ admits a filtration with associated graded pieces:
$$
\mathrm{Gr}^p\underline{\Gamma}(BG, \Omega_{BG}^{\ge 0}) \cong \underline{\Gamma}(BG, (\wedge^p L_{BG})[-p]),
$$
see \cite[Corollary 1.1.6]{MR4475269}. This is the ``Hodge filtration'' of $BG$, studied in detail in \emph{op.cit.}.

The proof of Proposition \ref{prop-de-rham-triangle} will follow from this filtration, combined with a calculation of the Hodge cohomology of $BG$ given below.
\end{void}

\begin{lem}
\label{lem-hodge-triangle}
There is a canonical triangle of complexes of sheaves of $k$-vector spaces:
\begin{equation}
\label{eq-hodge-triangle}
(\fr g^*)^G\otimes \Omega_S^{p-1}[1] \rightarrow \tau^{\le 0}\underline{\Gamma}{}_e(BG, (\wedge^pL_{BG})[2]) \rightarrow \Sym^2(\fr g^*)^G\otimes \Omega_S^{p-2}.
\end{equation}
\end{lem}
\begin{proof}
Denote by $\pi : BG \rightarrow S$ the projection map. The canonical triangle of cotangent complexes, combined with the identification $L_{BG/S}\cong\fr g^*[-1]$, yields a triangle in $\QCoh(BG)$:
\begin{equation}
\label{eq-cotangent-complex-triangle}
\pi^*\Omega_{S} \rightarrow L_{BG} \rightarrow \fr g^*[-1].
\end{equation}
Therefore, $\wedge^pL_{BG}$ admits a filtration with associated graded pieces $\Sym^n(\fr g^*)[-n]\otimes\pi^*\Omega_S^{p-n}$ for $0\le n\le p$. In particular, the filtrants for $n\ge 3$ do not contribute to the connective truncation $\tau^{\le 0}\underline{\Gamma}(BG, (\wedge^pL_{BG})[2])$.

Next, we compute by the projection formula:
\begin{align*}
	\pi_*(\Sym^n(\fr g^*)[-n]\otimes \pi^*\Omega_S^{p-n}) &\cong \pi_*(\Sym^n(\fr g^*))[-n] \otimes\Omega_S^{p-n} \\
	& \cong \Sym^n(\fr g^*)^G[-n]\otimes\Omega_S^{p-n}
\end{align*}
The filtrant for $n = 0$ defines the canonical map $\Omega_S^p[2] \rightarrow \tau^{\le 0}\underline{\Gamma}(BG, (\wedge^pL_{BG})[2])$, for which $e^*$ is a splitting. Hence the complex of rigidified sections fits into a triangle \eqref{eq-hodge-triangle}, as desired.
\end{proof}

\begin{proof}[Proof of Proposition \ref{prop-de-rham-triangle}]
The complex $\underline{\Gamma}{}_e(BG, \Omega^{\ge 2}_{BG}[4])$ admits a filtration whose associated grade pieces are $\underline{\Gamma}{}_e(BG, (\wedge^pL_{BG})[4-p])$ for $p\ge 2$. It follows from Lemma \ref{lem-hodge-triangle} that the connective truncations:
$$
\tau^{\le 0}\underline{\Gamma}{}_e(BG, (\wedge^p L_{BG})[4 - p]) = 0,\quad \text{for }p\ge 4.
$$

Therefore, $\tau^{\le 0}\underline{\Gamma}{}_e(BG, \Omega_{BG}^{\ge 2}[4])$ is the connective fiber of a map of complexes, corresponding to the above filtration in degrees $p=2,3$:
$$
\begin{tikzcd}[column sep = 1em]
	& \tau^{\le 0}\underline{\Gamma}{}_e(BG, \Omega_{BG}^{\ge 2}[4])\ar[d] & \\
	(\fr g^*)^G\otimes\Omega_S^1[1] \ar[r]\ar[d, "d"] & \tau^{\le 0}\underline{\Gamma}{}_e(BG, (\wedge^2L_{BG})[2]) \ar[r]\ar[d] & \Sym^2(\fr g^*)^G\otimes\cal O_S\ar[d, "d"] \\
	(\fr g^*)^G\otimes \Omega_S^2[1] \ar[r] & \tau^{\le 0}\underline{\Gamma}{}_e(BG, (\wedge^3L_{BG})[2]) \ar[r] & \Sym^2(\fr g^*)^G\otimes\Omega_S
\end{tikzcd}
$$
Here, the rows are split triangles from Lemma \ref{lem-hodge-triangle} and the outer vertical arrows are induced by the differentials on $\Omega_S^n$ ($n=0,1$), as follows from the description of the Hodge filtration on a smooth scheme. The desired split triangle thus follows.
\end{proof}

\begin{rem}
\label{rem-quantum-parameters}
If $G$ is defined over $k$, then the triangles \eqref{eq-de-rham-triangle} and \eqref{eq-hodge-triangle} are canonically split. Indeed, the hypothesis implies that \eqref{eq-cotangent-complex-triangle} is canonically split, from which we deduce the splittings of the other triangles by following their constructions.

If $S = X$ is an algebraic curve, then sections of $\tau^{\le 0}\underline{\Gamma}{}_e(BG, \Omega_{BG}^{\ge 2}[4])$ are given by a pair $(\kappa, E)$ where $\kappa$ is a $G$-invariant symmetric form on $\fr g$ and $E$ is an extension of quasi-coherent $\cal O_X$-modules:
$$
0 \rightarrow \Omega_X \rightarrow E \rightarrow \fr g_{\ab}\otimes\cal O_X \rightarrow 0.
$$
Such pairs are indeed the ``quantum parameters'' which appear in the geometric Langlands program, see \cite{zhao2017quantum}.
\end{rem}

\subsection{Comparison with integral metaplectic covers}
\label{sect-rational-covers}

\begin{void}
We continue to work over a field of characteristic zero. Recall the complex $\mathbb Z(n)$ of \'etale sheaves of $\mathbb Z$-modules from \S\ref{void-motivic-complex}. Write $\mathbb Q(n) := \mathbb Z(n)\otimes\mathbb Q$. There is a canonical $\mathbb Q$-linear map:
\begin{equation}
\label{eq-hodge-realization-map}
\mathbb Q(n)[n] \rightarrow \Omega^n_S
\end{equation}
natural in $S\in\Sm_{/k}$.

To define \eqref{eq-de-rham-realization-map}, we may assume $k = \bar k$ by \'etale descent. Viewing $\Omega^n : S\mapsto H^0(S, \Omega^n_S)$ as an \'etale sheaf of $k$-vector spaces on $\Sm_{/k}$, it has a subsheaf $\Omega_{\log}^n$ of differential forms of logarithmic singularity along $D$, for any smooth compactification $S\subset\bar S$ with a normal crossing boundary divisor $D$.

The degeneration of the Hodge-de Rham spectral sequence for logarithmic forms \cite[Corollaire 3.2.13(ii)]{MR498551} gives a natural isomorphism:
$$
H^0(S, \Omega_{\log}^n) \cong F^{\ge n}H^n_{\dR}(S, k),
$$
where $F^{\ge n}$ denotes the Hodge filtration. In particular, $\Omega_{\log}^n$ is an $\mathbb A^1$-invariant \'etale sheaf of $k$-vector spaces on $\Sm_{/k}$.

The morphism:
$$
\mathbb G_m^{\times n} \rightarrow \Omega_{\log}^n,\quad f_1,\cdots, f_n\mapsto d\log(f_1)\wedge\cdots d\log(f_n)
$$
induces a morphism $\mathbb Z(n)[n] \rightarrow \Omega_{\log}^n$ of complexes using the $\mathbb A^1$-invariance of the target, as in \S\ref{void-motivic-complex-to-K-theory}. Embedding the target in $\Omega^n$ and using its $\mathbb Q$-linear structure, we obtain the morphism \eqref{eq-de-rham-realization-map}.

By construction, the composition of \eqref{eq-de-rham-realization-map} with the $d : \Omega^n_S \rightarrow \Omega^{n+1}_S$ vanishes. Since $\mathbb Q(n)$ is concentrated in cohomological degrees $\le n$, it lifts to a morphism:
\begin{equation}
\label{eq-de-rham-realization-map}
	\mathbb Q(n) \rightarrow \Omega_S^{\ge n}.
\end{equation}
\end{void}

\begin{void}
The morphism \eqref{eq-de-rham-realization-map} allows us to relate integral metaplectic covers to de Rham ones. Namely, it induces a functor:
\begin{align}
\notag
\tau^{\le 0}\underline{\Gamma}{}_e(BG, \mathbb Z(2)[4])\otimes\mathbb Q &\cong \tau^{\le 0}\underline{\Gamma}{}_e(BG, \mathbb Q(2)[4]) \\
 \label{eq-de-rham-realization-map-metaplectic-cover}
&\rightarrow\tau^{\le 0}\underline{\Gamma}{}_e(BG, \Omega_{BG}^{\ge 2}[4]).
\end{align}
\end{void}

\begin{void}
Let us suppose that $G$ contains a maximal torus $T$ with sheaf of cocharacters $\Lambda$. Theorem \ref{thm-motivic-versus-K2}, together with \cite[Theorem 7.2]{MR1896177}, shows that $\tau^{\le 0}\underline{\Gamma}{}_e(BG, \mathbb Z(2)[4])$ fits into a canonical triangle:
\begin{equation}
\label{eq-motivic-complex-fiber-sequence}
\underline{\Hom}(\pi_1(G), \mathbb Z(1)[2]) \rightarrow \tau^{\le 0}\underline{\Gamma}{}_e(BG, \mathbb Z(2)[4])\rightarrow \Sym^2(\check{\Lambda})^W,
\end{equation}
where $W$ denotes the Weyl group $N_G(T)/T$.

The morphism \eqref{eq-de-rham-realization-map-metaplectic-cover} then induces a morphism of triangles:
\begin{equation}
\label{eq-rational-to-quantum-fiber-sequences}
\begin{tikzcd}[column sep = 1em]
	\underline{\Hom}(\pi_1(G), \mathbb Q(1)[2]) \ar[r]\ar[d, "(1)"] & \tau^{\le 0}\underline{\Gamma}{}_e(BG, \mathbb Q(2)[4]) \ar[r]\ar[d] & \Sym^2(\check{\Lambda})^W\otimes\mathbb Q \ar[d, "(2)"] \\
	(\fr g^*)^G\otimes [\Omega_S^1\xrightarrow{d}\Omega_S^{2,\cl}] \ar[r] & \tau^{\le 0}\underline{\Gamma}{}_e(BG, \Omega_{BG}^{\ge 2}[4]) \ar[r] & \Sym^2(\fr g^*)^G
\end{tikzcd}
\end{equation}
where the bottom triangle comes from Proposition \ref{prop-de-rham-triangle}.
\end{void}

\begin{rem}
The two outer morphisms of \eqref{eq-rational-to-quantum-fiber-sequences} have the following descriptions:
\begin{enumerate}
	\item Writing $\pi_1(G) = \Lambda/\Lambda_{\mathrm{sc}}$ for $\Lambda_{\mathrm{sc}}$ the sheaf of cocharacters of the induced maximal torus of $G_{\mathrm{sc}}$, we have:
	\begin{align*}
	\underline{\Hom}(\pi_1(G), \mathbb Q(1)[2]) &\cong [\check{\Lambda}\otimes\mathbb Q\rightarrow\check{\Lambda}_{\mathrm{sc}}\otimes\mathbb Q]\otimes\mathbb G_m[1]\\
	&\cong[\check{\Lambda}_{\ab}\otimes\mathbb Q]\otimes\mathbb G_m[1],
	\end{align*}
	where $\Lambda_{\ab}$ denotes the sheaf of cocharacters of $G_{\ab}$.
	
	This complex admits a natural map to $(\fr g^*)^G\otimes [\Omega_S^1\xrightarrow{d}\Omega_S^{2,\cl}]$, induced from the maps $\check{\Lambda}_{\ab}\otimes\mathbb Q \rightarrow \fr g^*_{\ab} \cong (\fr g^*)^G$ and $d\log : \mathbb G_m[1] \rightarrow [\Omega_S^1\xrightarrow{d}\Omega_S^{2, \cl}]$.
	
	\item Writing $\fr t$ for the Lie algebra of $T$, we have:
	$$
	\Sym^2(\check{\Lambda})^W\otimes\mathbb Q \rightarrow \Sym^2(\fr t^*)^W \cong \Sym^2(\fr g^*)^G,
	$$
	where the second isomorphism comes from Chevalley's theorem.
\end{enumerate}
\end{rem}

\begin{rem}
The following diagram summarizes the relationship among cochains on $BG$ in motivic, \'etale, and de Rham cohomological contexts:
$$
\begin{tikzcd}[column sep = 0.5em]
	\tau^{\le 0}\underline{\Gamma}{}_e(BG, \mathbb Z(2)[4])\ar[d] \\
	\tau^{\le 0}\underline{\Gamma}{}_e(BG, \mathbb Q(2)[4]) \ar[r]\ar[d] & \tau^{\le 0}\underline{\Gamma}{}_e(BG, \Omega^{\ge 2}_{BG}[4]) \\
	\tau^{\le 0}\underline{\Gamma}{}_e(BG, \mathbb Q/\mathbb Z(2)[4]) \ar[r, phantom, "\cong"] & \colim_N \underline{\Gamma}{}_e(BG, B^4\mu_N^{\otimes 2})
\end{tikzcd}
$$
Here, the bottom isomorphism appeals to \cite[Theorem 10.3]{MR2242284}.

If the ground field $k$ is algebraically closed and coincides with the coefficient field, there is a canonical character $\mathbb Q/\mathbb Z(1) \rightarrow k^{\times}$ defined by the inclusion $\mu(k)\subset k^{\times}$. Hence, any object of the bottom groupoid defines a metaplectic $\tn L$-group as in \S\ref{sec-metaplectic-dual}. Furthermore, any object of $\tau^{\le 0}\underline{\Gamma}{}_e(BG, \mathbb Q(2)[4])$ defines a pair of quantum parameters for $G$ and its Langlands dual group $\check G$, by \cite{zhao2017quantum}. In this context, one may formulate compatibility statements between the quantum and the metaplectic Langlands programs. They are, however, beyond the scope of the present article.
\end{rem}

\medskip

\appendix

\section{Twisting construction}
\label{sec-twisting-construction}

Given a sheaf of groups $A$, an $A$-torsor $P$, and a sheaf of sets $X$ equipped with an $A$-action, we obtain another sheaf of sets $X_P := P\times^AX$: the twist of $X$ by $P$. When $A$ is abelian, two $A$-torsors $P_1$, $P_2$ define a third one $P_1\otimes P_2$. When $X$ has a monoid structure, we obtain a multiplication map:
\begin{equation}
\label{eq-multiplication-torsor}
X_{P_1} \times X_{P_2} \rightarrow X_{P_1\otimes P_2},
\end{equation}
if the equality $(a_1a_2)\cdot (x_1x_2) = (a_1\cdot x_1)(a_2\cdot x_2)$ holds for all $a_1,a_2\in A$ and $x_1,x_2\in X$.

Let us involve another piece of structure: $X$ is now a sheaf of $E$-algebras equipped with a grading $X = \bigoplus_{\lambda\in\Gamma} X_{\lambda}$ by some abelian group $\Gamma$, such that $1 \in X_0$ and $x_1x_2$ has grading $\lambda_1 + \lambda_2$ if $x_1$, $x_2$ have gradings $\lambda_1$, $\lambda_2$. Then any multiplicative $A$-torsor $P$ on $\Gamma$ defines a new sheaf of $E$-algebras $X_P := \bigoplus_{\lambda\in\Gamma} (X_{\lambda})_{P(\lambda)}$ with multiplicative rule \eqref{eq-multiplication-torsor}. Here, $P(\lambda)$ denotes the $A$-torsor $P\times_{\Gamma}\{\lambda\}$.

The goal of this section is to explain an analogous construction where $X$ is replaced by a $\Gamma$-graded stack of tensor categories $\cal C$. We explain the meaning of an $A$-action on $\id_{\cal C}$ and how to form a twisted stack of tensor categories $\cal C_F$ for an $\mathbb E_{\infty}$-functor $F : \Gamma \rightarrow B^{(2)}(A)$.

\begin{rem}
There are analogous constructions when $\cal C$ is monoidal (resp.~braided monoidal) and $F : \Gamma \rightarrow B^{(2)}(A)$ is $\mathbb E_1$-monoidal (resp.~$\mathbb E_2$-monoidal). The result $\cal C_F$ is then a sheaf of monoidal (resp.~braided monoidal) categories.
\end{rem}

\subsection{Actions}

\begin{void}
\label{void-abelian-group-acting-on-category}
Suppose that $A$ is an abelian group. Let $\cal C$ be a category. We say that $A$ \emph{acts on} $\id_{\cal C}$ if there is a group homomorphism $A \rightarrow \Aut(\id_{\cal C})$. Here, $\id_{\cal C}$ denotes the identity functor viewed as an object of the category of endofunctors of $\cal C$. Concretely, an $A$-action on $\id_{\cal C}$ means that to each $a\in A$ and $c\in\cal C$, there is an isomorphism $a_c : c\cong c$. They satisfy:
\begin{enumerate}
	\item $1_c$ is the identity for all $c\in\cal C$;
	\item $(a_1a_2)_c = (a_1)_c\circ (a_2)_c$ for all $a_1,a_2\in A$ and $c\in\cal C$;
	\item $a_{c_2}\circ f = f\circ a_{c_1}$ for all $a\in A$ and $f : c_1\rightarrow c_2$ in $\cal C$.
\end{enumerate}
Suppose that $\cal C$ is a symmetric monoidal category. We say that an $A$-action on $\id_{\cal C}$ is \emph{multiplicative} if the following condition is satisfied:
\begin{enumerate}
	\item[(4)] $(a_1)_{c_1}\otimes (a_2)_{c_2} = (a_1a_2)_{c_1\otimes c_2}$ for all $a_1,a_2\in A$ and $c_1,c_2\in\cal C$.
\end{enumerate}
\end{void}

\begin{void}
Denote by $B(A)$ the groupoid with a single object $*$ and $\Aut(*) = A$ (i.e.~the Bar construction in $\Spc$). It has a symmetric monoidal structure defined by the group structure of $A$. The notion of an $A$-action on $\id_{\cal C}$ is really a description of a $B(A)$-action on $\cal C$, in the following sense: it defines a groupoid object $[n]\mapsto \cal C^{[n]}$ in the $2$-category of categories covering the groupoid object $[n]\mapsto B(A)^{\times [n]} := B(A)^{\times n}$, together with an isomorphism $\cal C^{[0]}\cong \cal C$, such that the following diagram is Cartesian for both $i = 0$, $1$:
$$
\begin{tikzcd}
	\cal C^{[1]} \ar[r, "\delta^i"]\ar[d] & \cal C^{[0]} \ar[d] \\
	B(A) \ar[r] & *
\end{tikzcd}
$$
The groupoid object $[n]\mapsto\cal C^{[n]}$ is explicitly constructed by $[n]\mapsto B(A)^{\times n}\times\cal C$. One of the boundary maps, say $\delta^0$, passes to projection onto $\cal C$. The other one $\delta^1$ is given by:
\begin{equation}
\label{eq-action-map-on-category}
	\act : B(A)\times\cal C \rightarrow \cal C,\quad (a, f)\mapsto af := a_{c_2}\circ f = f\circ a_{c_1}.
\end{equation}
(The formula in \eqref{eq-action-map-on-category} describes what $\act$ does to morphisms $a\in A$ and $f : c_1\rightarrow c_2$.) The higher boundary maps are compositions of actions and projections. The degeneracy maps are insertions along $*\in B(A)$. Conditions (1) and (2) of \S\ref{void-abelian-group-acting-on-category} ensure that the simplicial object is well defined. Taking geometric realization of the morphism $\cal C^{[n]} \rightarrow B(A)^{[n]}$, we obtain a functor of $2$-categories (see \S\ref{void-iterated-bar-construction-of-abelian-groups}):
\begin{equation}
\label{eq-twisting-construction-base}
\cal C^{[-1]} \rightarrow B^{(2)}(A).
\end{equation}
\end{void}

\begin{void}
If $\cal C$ is a symmetric monoidal category and the $A$-action on $\cal C$ is multiplicative, then \eqref{eq-action-map-on-category} is itself a functor of symmetric monoidal categories. The isomorphism between $\act(*, c_1)\otimes\act(*, c_2)$ and $\act(*, c_1\otimes c_2)$ is the obvious one. However, it demands a commutative diagram:
$$
\begin{tikzcd}
	c_1 \otimes d_1 \ar[r, "\cong"]\ar[d, "(a_1f)\otimes (a_2g)"] & c_1\otimes d_1 \ar[d, "(a_1a_2)\otimes (fg)"] \\
	c_2 \otimes d_2 \ar[r, "\cong"] & c_2\otimes d_2
\end{tikzcd}
$$
This follows from conditions (3) and (4) of \S\ref{void-abelian-group-acting-on-category}. It follows that $[n]\rightarrow\cal C^{[n]}$ is a simplicial object in the $2$-category of symmetric monoidal categories. Furthermore, the morphism $\cal C^{[n]}\rightarrow B(A)^{[n]}$ is a morphism of such. It follows that \eqref{eq-twisting-construction-base} lifts to a functor of $\mathbb E_{\infty}$-monoidal $2$-categories. (We have used the fact that forgetting the $\mathbb E_{\infty}$-structure commutes with sifted colimits, see the proof of Lemma \ref{lem-functors-structured-spaces-properties}.)
\end{void}

\begin{void}
\label{void-sheaf-theoretic-twisting-construction-base}
The above constructions carry sheaf-theoretic meaning. Fix a site $\cal X$ and let $A$ be a sheaf of abelian groups, $\cal C$ be a stack of categories. Then an $A$-action on $\id_{\cal C}$ defines a morphism of simplicial stacks of categories $\cal C^{[n]} \rightarrow B(A)^{[n]}$. Here, $B(A)$ denotes the Bar construction of $A$ in the $\infty$-category of $\Spc$-valued sheaves. By taking the geometric realization, we obtain a morphism of sheaves of $2$-categories:
\begin{equation}
\label{eq-multiplicative-action}
	\cal C^{[-1]} \rightarrow B^{(2)}(A).
\end{equation}
If $\cal C$ carries a symmetric monoidal structure and the $A$-action is multiplicative, then \eqref{eq-multiplicative-action} lifts to a morphism of sheaves of $\mathbb E_{\infty}$-monoidal $2$-categories.
\end{void}

\subsection{How $(\cal C, F)$ defines $\cal C_F$}

\begin{void}
\label{void-twisting-construction-plain}
We continue to fix a site $\cal X$ and let $A$ be a sheaf of abelian groups, $\cal C$ be a stack of categories. Suppose that $A$ acts on $\id_{\cal C}$. Given any section $F$ of $B^{(2)}(A)$ over $x$, the fiber product of \eqref{eq-multiplicative-action} with $F : x\rightarrow B^{(2)}(A)$ defines a stack of categories over $x$. We denote it by $\cal C_F$ and view it as the ``$F$-twist of $\cal C$.''

Note that for any $x_1\rightarrow x$ such that the pullback $F_{x_1}$ of $F$ is trivialized, i.e., factors as $x_1\rightarrow * \rightarrow B^{(2)}(A)$, the pullback of $\cal C_F$ to $x_1$ is isomorphic to the pullback of $\cal C$.
\end{void}

\begin{void}
If $E$ is a ring and $\cal C$ is an $E$-linear category, the same structure is inherited by $\cal C_F$. Indeed, we may let $\cal S\subset\Hom(-, x)$ be the covering sieve consisting of morphisms $x_1\rightarrow x$ such that $F_{x_1}$ is trivial. Then we have:
$$
\cal C_F(x) \cong \lim_{(x_1\rightarrow x)\in\cal S} \cal C_F(x_1).
$$
On the right-hand-side, $\cal C_F(x_1)$ has an $E$-linear structure by choosing \emph{any} trivialization of $F_{x_1}$ and transport the $E$-linear structure from $\cal C(x_1)$. Two distinct trivializations of $F_{x_1}$ differ by a section of $B(A)$, which locally defines an automorphism of $\cal C(x_1)$ by the action map \eqref{eq-action-map-on-category}. Since this action map is $E$-linear on the Hom-sets, the category $\cal C_F(x_1)$ acquires an $E$-linear structure independently of the trivialization of $F_{x_1}$. The same structure then passes to $\cal C_F(x)$.
\end{void}

\begin{void}
\label{void-graded-tensor-category}
Let us now suppose that $\cal C$ is a stack of symmetric monoidal $E$-linear additive categories, together with a decomposition $\cal C = \bigoplus_{\lambda\in\Gamma}\cal C_{\lambda}$ for an abelian group $\Gamma$ such that:
\begin{enumerate}
	\item $\mathbf 1_{\cal C} \in \cal C_0$;
	\item $c_1 \otimes c_2 \in \cal C_{\lambda_1 + \lambda_2}$ if $c_1\in \cal C_{\lambda_1}$ and $c_2\in\cal C_{\lambda_2}$.
\end{enumerate}
These data may be packaged differently: let $\cal C^{\sqcup}$ denote the stack of categories over $\Gamma$ whose fiber at $\lambda\in\Gamma$ is $\cal C_{\lambda}$. Then $\cal C^{\sqcup}$ admits a symmetric monoidal structure such that $\cal C^{\sqcup} \rightarrow\Gamma$ is a symmetric monoidal functor. (Here, $\Gamma$ is viewed as a discrete category whose symmetric monoidal structure comes from the group operations.)

Conversely, given a symmetric monoidal functor $\cal D \rightarrow \Gamma$ where $\cal D$ is a stack of symmetric monoidal $E$-linear categories whose \emph{fibers} over $\Gamma$ are additive, we obtain a stack of symmetric monoidal $E$-linear additive categories $\cal D^{\oplus} := \bigoplus_{\lambda\in\Gamma}\cal D_{\lambda}$ by taking direct sum of the fibers.
\end{void}

\begin{void}
\label{void-twisting-construction-symmetric-monoidal}
Let $\cal C$ be as in \S\ref{void-graded-tensor-category}. Suppose that $A$ acts multiplicatively on $\id_{\cal C}$. Then it induces a multiplicative action on $\id_{\cal C^{\sqcup}}$. It also acts trivially on $\id_{\Gamma}$ and the functor $\cal C^{\sqcup}\rightarrow\Gamma$ is tautologically compatible with the actions. The construction of \eqref{eq-multiplicative-action} is functorial in $\cal C$. In particular, the symmetric monoidal functor $\cal C^{\sqcup}\rightarrow\Gamma$ yields an $\mathbb E_{\infty}$-monoidal functor:
\begin{equation}
\label{eq-monoidal-functor-parametrized-by-abelian-group}
\cal C^{\sqcup, [-1]} \rightarrow \Gamma \times B^{(2)}(A).
\end{equation}
Suppose that $F : \Gamma \rightarrow B^{(2)}(A)$ is an $\mathbb E_{\infty}$-monoidal morphism. Taking fiber product of \eqref{eq-monoidal-functor-parametrized-by-abelian-group} with $(\id_{\Gamma}, F)$ yields an $\mathbb E_{\infty}$-monoidal functor $\cal C^{\sqcup}_F \rightarrow \Gamma$.

Finally, we apply the construction of \S\ref{void-graded-tensor-category} to obtain $\cal C_F := (\cal C_F^{\sqcup})^{\oplus}$, which is a stack of symmetric monoidal $E$-linear additive categories equipped with a compatible $\Gamma$-grading. This is the ``$F$-twist'' of $\cal C$.
\end{void}

\begin{rem}
Let us give an informal account of $\cal C_F$. It admits a $\Gamma$-grading:
$$
\cal C_F \cong \bigoplus_{\lambda\in\Gamma} (\cal C_{\lambda})_{F(\lambda)},
$$
where $(\cal C_{\lambda})_{F(\lambda)}$ is the $F(\lambda)$-twist of $\cal C_{\lambda}$ in the sense of \S\ref{void-twisting-construction-plain}. The monoidal operation on $\cal C_F$ is given as follows: for $\lambda_1,\lambda_2\in\Gamma$, we have $(\cal C_{\lambda_1})_{F(\lambda_1)}\times (\cal C_{\lambda_2})_{F(\lambda_2)} \rightarrow (\cal C_{\lambda_1 + \lambda_2})_{F(\lambda_1 + \lambda_2)}$ coming from the monoidal operation on $\cal C$ and the monoidal structure on $F$. The symmetric monoidal structure is likewise induced from those of $\cal C$ and $F$.
\end{rem}

\bibliographystyle{amsalpha}
\bibliography{../biblio_mathscinet.bib}

\end{document}